\documentclass[twoside,leqno]{article}

\usepackage[letterpaper]{geometry}

\usepackage{ltexpprt}
\usepackage{hyperref}



\usepackage{bm}
\usepackage[utf8]{inputenc}
\usepackage{mathrsfs}
\usepackage[T1]{fontenc}

\usepackage{comment}
\usepackage{listings}
\usepackage{algorithmicx}
\usepackage[ruled]{algorithm}
\usepackage{algpseudocode}
\usepackage{algpascal}
\usepackage{algc}

\usepackage{bigints}
\usepackage{amsmath,amsfonts,amssymb}
\usepackage{MnSymbol}

\usepackage{graphicx}

\usepackage{hyperref} 
\hypersetup{
    colorlinks=true,
    linkcolor=blue,
    filecolor=magenta,
    urlcolor=cyan,
}

\usepackage{ulem}

\usepackage{cleveref} 
\crefformat{equation}{(#2#1#3)}
\crefformat{line}{Line #1}
\crefformat{table}{Table #2#1#3}
\usepackage{multirow}
\crefformat{lemma}{Lemma #2#1#3}
\crefformat{proposition}{Proposition #2#1#3}
\crefformat{theorem}{Theorem #2#1#3}
\crefformat{corollary}{Corollary #2#1#3}
\crefformat{question}{Question #2#1#3}

\usepackage[english]{babel}
\usepackage[hang,small]{caption}
\usepackage{boxedminipage}
\usepackage{dsfont}

\usepackage[usenames,dvipsnames]{xcolor}
\usepackage{colortbl}

\usepackage[customcolors]{hf-tikz}

\tikzset{style green/.style={
    set fill color=green!50!lime!60,
    set border color=white,
  },
  style cyan/.style={
    set fill color=cyan!90!blue!60,
    set border color=white,
  },
  style orange/.style={
    set fill color=orange!80!red!60,
    set border color=white,
  },
  hor/.style={
    above left offset={-0.15,0.31},
    below right offset={0.15,-0.125},
    #1
  },
  ver/.style={
    above left offset={-0.1,0.3},
    below right offset={0.15,-0.15},
    #1
  }
}
\usepackage[textwidth=1.5cm, textsize=scriptsize]{todonotes}
\usepackage{pdfpages}

\def\beq{\begin{equation}}
  \def\eeq{\end{equation}}
  \def\beqn{\begin{eqnarray*}}
  \def\eeqn{\end{eqnarray*}}
  \def\bitem{\begin{itemize}}
  \def\eitem{\end{itemize}}
  \def\benum{\begin{enumerate}}
  \def\eenum{\end{enumerate}}
  \def\bmult{\begin{multline*}}
  \def\emult{\end{multline*}}
  \def\bcenter{\begin{center}}
  \def\ecenter{\end{center}}

\newcolumntype{P}[1]{>{\centering\arraybackslash}p{#1}}
\DeclareMathOperator*{\argmax}{arg\, max}
\newcommand{\argmin}{\mathop{\mathrm{arg\,min}}}

\def\cF{\mathcal{F}}
\def\cG{\mathcal{G}}
\def\cH{\mathcal{H}}

\def\cP{\mathcal{P}}

\def\1{{\mathbf 1}}
\def\0{{\mathbf 0}}

\newcommand\bPi{{\boldsymbol\Pi}}

\def\bbC{\mathbb{C}}

\def\bbR{\mathbb{R}}

\DeclareMathAlphabet\mathbfcal{OMS}{cmsy}{b}{n}

\newcommand{\E}{\operatorname{\mathbb{E}}}

\newcommand{\floor}[1]{\left\lfloor#1\right\rfloor}

\def\proscal<#1,#2>{\langle #1,#2\rangle}
\newcommand{\poubelle}[1]{}

 \newcommand{\algoPrincipal}{\mathbf{ISR}}

 \newcommand{\iso}{\mathrm{iso}}

 \newcommand{\reco}{\mathrm{reco}}
 \newcommand{\perm}{\mathrm{perm}}

 \def\1{{\mathbf 1}}
 \def\0{{\mathbf 0}}
 
 \def\bbC{\mathbb{C}}

 \def\bbR{\mathbb{R}}

 \def\GS{\mathfrak{S}}

 \DeclareMathAlphabet\mathbfcal{OMS}{cmsy}{b}{n}
 \alglanguage{pseudocode}
\begin{document}

\newcommand\relatedversion{}
\renewcommand\relatedversion{\thanks{A shorter version of this review paper is to be published in the proceedings of the ICM 2026.}} 

\title{\Large Statistical and computational challenges in ranking\relatedversion
}
\author{
Alexandra Carpentier\thanks{Institüt für Mathematik, Universität Potsdam, Potsdam, Germany}
\and Nicolas Verzelen\thanks{UMR MISTEA, INRAE, Univ Montpellier, Institut Agro, 34060 Montpellier, France.}
}

\date{}

\maketitle
\begin{abstract} \small\baselineskip=9pt 
We consider the problem of ranking $n$ experts according to their abilities, based on the correctness of their answers to $d$ questions. This is modeled by the so-called crowd-sourcing model, where the answer of expert $i$ on question $k$ is modeled by a random entry, parametrized by $M_{i,k}$ which is increasing linearly with the expected quality of the answer. To enable the unambiguous ranking of the experts by ability, several assumptions on $M$ are available in the literature. We consider here the general isotonic crowd-sourcing model, where $M$ is assumed to be isotonic up to an unknown permutation $\pi^*$ of the experts - namely, $M_{\pi^{*-1}(i),k} \geq M_{\pi^{*-1}(i+1),k}$ for any $i\in [n-1], k \in [d]$. Then, ranking experts amounts to constructing an estimator of $\pi^*$. 
In particular, we investigate here the existence of statistically optimal and computationally efficient procedures and we describe recent results that  disprove the existence of computational-statistical gaps for this problem. 
To provide insights on the key ideas,  we start by discussing 
simpler and yet related sub-problems, namely sub-matrix detection and estimation. This corresponds to specific instances of the ranking problem where the matrix $M$ is constrained to be of the form $\lambda \mathbf 1\{S\times T\}$ where $S\subset [n], T\subset [d]$. This model has been extensively studied. We provide an overview of the results and proof techniques for this problem with a particular emphasis on the computational lower bounds based on low-degree polynomial methods. Then, we build upon this instrumental sub-problem to discuss existing results and algorithmic ideas for the general ranking problem.

\end{abstract}

\section{Introduction}

  \subsection{Ranking problems.} 
Ranking problems have received a lot of attention in the statistical, machine learning, and computer science literature. This includes a variety of practical situations ranging from ranking experts/workers in crowd-sourced data, ranking players in a tournament or equivalently sorting objects based on pairwise comparisons. To fix ideas, let us consider a problem where we have $n$ experts (or workers) and $d$ questions (or tasks). Given an expert $i\in [n]$ and a question $k\in [d]$, the random variables $Y_{i,k}$ stands for the validity of the response of expert $i$ to question $k$. More specifically, we have $Y_{i,k}=1$ if expert $i$ answers correctly to question $k$ and $Y_{i,k}=-1$ otherwise. It is standard\footnote{There exist other parametrizations but those turn out to be equivalent.} to assume that the $Y_{i,k}$'s are independent and that there exists there exists an unknown matrix $M\in [0,1]^{n\times d}$ such that 
$$ \mathbb{P}_{M}[Y_{i,k} =1]= \frac{1+M_{i,k}}{2} \quad \quad \quad \quad   \mathbb{P}_{M}[Y_{i,k} =-1]= \frac{1-M_{i,k}}{2}\enspace , $$
so that $\mathbb{E}_{M}[Y_{i,k}]=1$.   Here, the matrix $M$ encodes the average performance of each expert on each question with $M_{i,k}=0$ when the expert does not answer the question better than random guess and $M_{i,k}=1$ the expert's knowledge is perfect on that question. Depending on the problem at hand, the statistical may fully observe the matrix $Y$ or observe a subset of its entries (partial observations). For tournament problems, we have $n=d$ players (or objects) and $(1+M_{i,k})/2$ stands for the probability that player $i$ wins against player $k$. Based on these noisy data, the general goal is to provide, based on the matrix $Y$, a full ranking of the experts or of the players. 
  
\subsubsection{Literature review.} Originally, these problems have been tackled using parametric model for the matrix $M$. Notably for tournament problem, this includes the noisy sorting model \cite{braverman2008noisy} or Bradley-Luce-Terry model \cite{bradley1952rank}. Still, it has been observed that these simple models are often unrealistic and do not tend to fit data well.   This has spurred a recent stream  of literature where strong parametric assumptions are replaced by non-parametric shape-constrained assumptions \cite{shah2015estimation,shah2016stochastically,shah2019feeling,shah2020permutation, mao2020towards,mao2018breaking,liu2020better,flammarion2019optimal,bengs2021preference,saad2023active, pilliat2022optimal, pilliat2024optimal, graf2024optimal}. In what follows, a bounded matrix  $A\in [0,1]^{n\times d}$ is said to be isotonic if all its columns are non-increasing, that is $A_{i,k}\geq A_{i+1,k}$ for any $i\in [n-1]$ and $k \in [d]$. Besides, a matrix $A$ is said to be bi-isotonic if both its rows and its columns are non-increasing. The most emblematic shape-constrained non-parametric models for ranking can be divided into three classes.
  \begin{itemize}
  \item For tournament problems, the strong stochastically transitive (SST) model presumes that the square matrix $M$ is, up to a common permutation $\pi^*$ of the rows and of the columns, bi-isotonic and satisfies the skew symmetry condition $M_{i,k}+M_{k,i}=1$. Optimal rates for estimation of the permutation $\pi^*$ have been pinpointed in the earlier paper of Shah et al.~\cite{shah2016stochastically}. However at that time there remained a large gap between these optimal rates and the best known performances of polynomial-time algorithms. Although this has been partially filled over the year, this has led to  conjecture the existence of a \emph{statistical-computational gap}~\cite{mao2020towards, liu2020better} - namely, that no ranking algorithms was at the same time statistically optimal, and computationally efficient, i.e.~polynomial time. See Subsection~\ref{ss:compLB}, and also later in this paper, for a more precise explanation. 
  \item For crowdsourcing data, the counterpart of the SST model is the so-called bi-isotonic model, where the rectangular matrix $M$ is bi-isotonic, up to an unknown permutation $\pi^*$ of its rows and an unknown permutation $\eta^*$ of its columns. 
   From a modeling perspective, this  presumes the existence of two unambiguously orderings $\pi^*$ and $\eta^*$ of the experts --from the best-one to the wort-one-- and of the questions -- from the simplest one to the most difficult one.  This model turns out to be really similar to the SST model and the existence of a statistical-computational gap has also been conjectured~\cite{mao2020towards}. See however~\cite{liu2020better, pilliat2022optimal,graf2024optimal} for two slight variations on this bi-isotonic setting where this conjecture has been disproved, and where polynomial-time and statistically optimal algorithms have been provided.
  \item Finally and again for crowdsourcing data, another relevant model is the so-called isotonic model, where the rectangular matrix $M$ is (only) isotonic, up to an unknown permutation $\pi^*$ of its rows~\cite{flammarion2019optimal}. In this case, there is no structure in $M$ along its columns, namely according to the questions, so that the isotonic model more flexible than the bi-isotonic and SST models. It is in fact the most general model under which an unambiguous ranking of the experts is well-defined. In this model as well, there was originally a gap between the (statistical) optimal rates, and the rate obtained by the (polynomial-time) algorithm in~\cite{flammarion2019optimal} so that again,  a statistical-computational gap has been widely conjectured. Interestingly, this conjecture was proven false in this model~\cite{pilliat2024optimal}, and a polynomial-time and statistically optimal algorithms has been provided, thereby settling the absence of any computational gap in this model. On a side note, there exist variants of these models~\cite{shah2020permutation}, coined as crowd-labeling, where the learner does not know the ground truth answer to the questions so that $Y$ is only observed up to multiplications by $\{-1,1\}$ of the columns. This corresponds to situations where the objective is both to rank the experts and to learn the answers. Despite their importance, we leave them aside in this survey. 
  \end{itemize}
  One may object that such shape-constraint conditions are quite strong and, despite their generality, may not be realistic in practical situations where e.g. different experts have different fields of expertise. In fact, there is an important literature on weakening such assumptions --see~\cite{bengs2021preference}. This is however made at the price of the ambiguity for the ordering of the experts. In this survey, we focus on these permuted shape-constrained for their mathematical elegance as well as the deep statistical and computational questions.

  \subsubsection{Isotonic model for crowdsourcing.}
  Although we will discuss results for SST and bi-isotonic models, we shall put further emphasis on the isotonic model for crowdsourcing. Let us introduce additional notation to formalize the problem.  Henceforth, we write $\mathbb{C}_{\iso}$ for the collection of all $n\times d$  isotonic matrices taking values in $[0,1]$ that is, the collection of matrices whose columns are non-increasing. We assume that there exists a permutation $\pi^*$ of $[n]$ such that the matrix $M_{\pi^{*-1}}$ defined by $(M_{\pi^{*-1}})_{i,k}= (M_{\pi^{*-1}(i),k})$ has non-increasing columns, that is $M_{\pi^{*-1}} \in \bbC_{\iso}$. Henceforth, $\pi^*$ is referred as an oracle permutation.  We define 
  \begin{equation}\label{eq:isotonic_matrix_up_to_permutation}
\overline{\bbC}_{\iso}:= \{M: \exists \pi^{*}\in\bPi_n: M_{\pi^{*-1}}\in \bbC_{\iso}\}\enspace ,
\end{equation}
for the class of permuted isotonic matrices. Here, $\bPi_n$ stands for the group of all permutations. 
Using the terminology of crowdsourcing, we refer to  $i^{\mathrm{th}}$ row of $M$ as \emph{expert} $i$ and to $k^{\mathrm{th}}$ column as  \emph{question} $k$.

We are primarily interested in estimating an oracle permutation $\pi^*$ based on the observed matrix $Y$. However, $\pi^*$ is neither unique nor identifiable if we only assume that $M\in \overline{\bbC}_{\iso}$. One may put additional constrains on the separation of the rows of $M$ to make $\pi^*$ identifiable and then directly estimate $\pi^*$ with respect to some distance on the permutation set --see e.g~\cite{mao2018minimax}. Here, we take a different route. Given an estimator $\hat{\pi}$, we use  the square Frobenius norm loss defined by  
\begin{align}\label{eq:estpermM}
\|M_{\hat{\pi}^{-1}} - M_{\pi^{*-1}} \|_F^2\enspace .
\end{align}
to quantify the \emph{ranking} error. 
It quantifies the distance between the matrix $M$ reordered according to the estimated permutation $\hat{\pi}$ and the matrix $M$ sorted according to the oracle permutation $\pi^*$.  Importantly, if the oracle permutation $\pi^*$ is not unique, the value of the loss~\eqref{eq:estpermM} does not depend on its choice and this loss is therefore well defined. This loss is explicitly used in~\cite{liu2020better,pilliat2022optimal} and is implicit in earlier works --see e.g.~\cite{shah2016stochastically,mao2020towards}.

In an important part of the literature~\cite{shah2016stochastically,mao2020towards,liu2020better}, the main focus is not only to estimate $\pi^*$, but also to reconstruct the unknown matrix $M$. In the latter, the \emph{estimation} or \emph{reconstruction} error of an estimator $\hat M$ is measured through reconstruction loss 
\begin{align}\label{eq:reconstruction:loss}
\|\hat M - M\|_F^2\enspace . 
\end{align}
It turns out that reconstructing the matrix $M$ is more challenging than estimating a permutation $\pi^*$: the optimal error for reconstruction is of the order of the sum of the optimal error for permutation estimation --in the sense of~\eqref{eq:estpermM} and the optimal error for reconstructing an isotonic matrix --with known permutation.

In this review, our primary objective is to provide the key ideas for establishing the information-theoretical error for permutation estimation in the general isotonic model $\overline{\bbC}_{\iso}$ and to explain how this statistically optimal error is achievable by a polynomial-time estimator. To gain insight, we shall first consider the much simpler problems of sub-matrix detection and estimation. Indeed, solving the isotonic ranking problem will be argued to be akin in some sense to solving in a robust way several (interacting) problems that are related to (partial) sub-matrix estimation. This is alluded to in~\cite{pilliat2022optimal,pilliat2024optimal}, but not made explicit, and we want to contextualise this further in this survey - which is what we will do in Sections~\ref{sec:peeling}~and~\ref{sec:ranking}.

\subsection{Sub-matrix detection and estimation.}\label{ss:subma}

Rectangular sub-matrix models can be interpreted as a specific instance of crowdsourcing models  where $M$ is of the form $\lambda \mathbf 1\{S\times T\}$ with $\lambda \in [0,1]$, and where $S\subset [n]$ and $T\subset [d]$. In the random graph terminology~\cite{wu2018statistical}, $(S,T)$ is sometimes called an hidden bi-clique or an hidden clique when $M$ is assumed to be symmetric. To simplify the discussion and to further connect it to ranking problems, we restrict here our attention to binary data $Y\in \{-1,1\}^{n\times d}$ but such models are usually defined and studied with general subGaussian distributions.

Within the important literature on submatrix models, we can distinguish two classes of statistical problems. In submatrix \emph{detection}, one aims at testing whether  $\lambda=0$ (or equivalently $M=0_{n\times d}$) against the alternative that $\lambda>0$. In submatrix \emph{estimation} or \emph{localization}, the objective is to reconstruct the unknown matrix $M$, usually in Frobenius norm, which in some way related to estimating both $S$ and $T$. In this survey, as our ultimate objective is to consider general ranking problems, we also consider the intermediary problem of solely estimating the subset $S$, that is the support of the rows.

For $\lambda=1$ and for symmetric matrices $M$ and data $Y$, this submatrix model is equivalent to the random planted clique model, one of the most iconic probabilistic models where a so-called \emph{computational-statistical gap} is strongly conjectured. In particular, whereas it is statistically possible to detect the submatrix or estimate the submatrix with a small error as soon as $|S|$ is large compared to $\log(n)$, the best available polynomial-time procedures are only proved to do this for $|S|$ large compared to $\sqrt{n}$ --see e.g.~\cite{DM15} for tight results. This is for instance achieved using the largest singular vector/eigenvalue of the matrix $Y$. over the last decade, a long line of literature has flourished towards providing evidence of the impossibility for polynomial-time procedures below the $\sqrt{n}$ threshold. In the next subsection, we shall shortly review such approaches. 
In the general rectangular submatrix models with general $n$, $d$, or $\lambda$, other phenomenons may arise. For instance, in some regimes, it is possible to detect in polynomial the presence of the submatrix, while it is impossible to estimate efficiently the matrix $M$ (or equivalently $S$ and $T$) with a small error~\cite{butucea2015sharp}. This phenomenon
coined as a \emph{detection-to-estimation gap} has important consequences, most notably for characterizing the statistical and computational complexity of the problem. For these reasons, submatrix problems have attracted a lot of attention in the recent statistical and machine learning literature --see~\cite{butucea2013detection,ma2015computational,butucea2015sharp,arias2014community,brennan2019universality,cai2017computational,chen2016statistical,wu2018statistical,hajek2017information,gamarnik2021overlap,hajek2018submatrix,cai2020statistical,chhor2025optimal} to name a few. An important part of the literature have been dedicated to the case where the matrix $Y$ is a sparse random graph~\cite{verzelen2015community,hajek2015computational,zdeborova2016statistical,chen2016statistical,montanari2015finding,yu2025counting}, although this will note be our purpose here.  Also, we point out that in highly rectangular regimes (e.g $n\ll d$), some phenomenons may arise due to the asymmetry; for instance, in some regimes it is possible to estimate the subset $S$ but not $T$. While this has been somewhat disregarded in the literature but see~\cite{chhor2025optimal} among others, this will be important towards our ranking purpose.

\subsection{Computational-statistical gaps.} \label{ss:compLB}

Characterizing computational-statistical gaps - following what we introduced in Subsection~\ref{ss:subma} - has been attracting a lot of interest in the recent years. Since statistical problems involve random instances, classical worst-case complexity classes (P, NP, etc.) are not well suited for characterizing hardness. Instead, computational lower bounds are typically established within specific models of computation, such as the sum-of-squares (SoS) hierarchy~\cite{HopkinsFOCS17,Barak19}, the overlap gap property~\cite{gamarnik2021overlap}, the statistical query framework~\cite{kearns1998efficient,brennan2020statistical}, and the low-degree polynomial model~\cite{hopkins2018statistical,KuniskyWeinBandeira,SchrammWein22}, sometimes in combination with reductions between statistical problems~\cite{brennan2020reducibility,pmlr-v30-Berthet13,brennan2018reducibility}.

Among these, low-degree polynomial (LD) lower bounds have recently emerged as a powerful tool for establishing state-of-the-art computational lower bounds in a variety of detection problems---including community detection~\cite{Hopkins17}, spiked tensor models~\cite{Hopkins17,KuniskyWeinBandeira}, sparse PCA~\cite{ding2024subexponential}  among others---and estimation problems---including submatrix estimation~\cite{SchrammWein22}, stochastic block models and graphons~\cite{luo2023computational,SohnWein25}, dense cycle recovery~\cite{mao2023detection}, and planted coloring~\cite{kothari2023planted}---see~\cite{SurveyWein2025} for a recent survey. In the LD framework, we restrict attention to estimators---or test statistics---that are multivariate polynomials of degree at most $D$ in the observations. The central conjecture in the LD literature is that, for many problems, degree-$O(\log n)$ polynomials are as powerful as any polynomial-time algorithm. Consequently, proving failure for all degree-$O(\log n)$ polynomials  provides strong evidence~\cite{KuniskyWeinBandeira} of polynomial-time hardness.
The LD framework connects to several other computational models, including statistical queries~\cite{brennan2020statistical}, free-energy landscapes from statistical physics~\cite{bandeira2022franz}, and approximate message passing~\cite{montanari2025equivalence}. 

In this paper, we focus on low-degree polynomials lower bounds for proving computational hardness for the problems of sub-matrix detections and estimations, which is a sub-problem of ranking. We summarize and re-prove there a few results for this problem - see Sections~\ref{sec:detection} and~\ref{sec:rankingrest}.

\subsection{Notation and organization.}
 For two matrices $A$ and $B$ in $\mathbb{R}^{n\times d}$, we write that $A\geq B$ when $A_{i,j}\geq B_{i,j}$ for all $(i,j)\in [n]\times [d]$. Given a set $A$, we write $|A|$ for its cardinality. Given a matrix $M\in [0,1]^{n\times d}$, we write $\mathbb{E}_{M}[.]$ for the expectation to emphasize the dependency on $M$. We will write $\Pi_n$ for the collection of all permutations of $[n]$. Given $x>0$, we write $\log_2(x)$ for the logarithmic in base 2.

This paper is organized as follows. In Section~\ref{sec:peeling}, we describe the connection between permuted isotonic matrices $\overline{\bbC}_{\iso}$ and submatrix models. This motivates the analysis of sub-matrix detection in Section~\ref{sec:detection}. We 
review the statistically optimal  conditions  and establish LD lower bounds. We deduce from that some consequences in the general permuted isotonic model $\overline{\bbC}_{\iso}$. 
In section~\ref{sec:rankingrest}, we move to the ranking problem by first restricting our attention to submatrices $M = \lambda \mathbf 1\{S\times T\}$. In this case, the ranking problem is akin to estimating $S$. Similarly to the previous section, we establish information-theoretical rates and give LD lower bounds. On the technical side, our new proof of the LD lower-bound, in our opinion, more intuitive and transparent than usual techniques. 
Finally, in Section~\ref{sec:ranking}, we go back to the general ranking problem over $\overline{\bbC}_{\iso}$. We explain how this problem can be decomposed as a hierarchy of (interacting) problems of sub-matrix (partial) estimation of the form $\lambda \mathbf 1\{S\times T\}$. We present the main result of the literature in this model and explain why no statistical-computational gap arise in the model $\overline{\bbC}_{\iso}$. 
We finish by providing a more in-depth discussion of the literature. Proofs are given in the appendix.

\section{From ranking to sub-matrix detection and estimation models}\label{sec:peeling}

A block matrix is a matrix of the form $M'= \lambda \mathbf 1\{S\times T\}$ for some $\lambda>0$, and some subsets $S\subset[n]$ and $T\subset [d]$. Obviously, such a matrix $M'$ belong to $\overline{\bbC}_{\iso}$. With the crowdsourcing terminology, $S$ corresponds to the set of 'good experts' and $[n]\setminus S$ to that of bad experts. The ranking problem then amounts to estimating $S$. Although such block matrices are arguably quite specific, we establish in this section that any matrix $M\in \overline{\bbC}_{\iso}$ can be lower bounded by a block matrix $M'$  whose Frobenius norm is of same order as that of $M$, up to a logarithmic factor in  $n$, $d$ and up to a negligible additional term. In this way, we relate the model of crowdsourcing in the isotonic model to the one of planted sub-matrix.

Given $\lambda\in (0,1)$, $K_n\in [n]$ and $K_d\in [d]$, we define the collection $\cH(\lambda, K_n, K_d)$ of block matrices of size $K_n$ and $K_d$ by 
$$\cH(\lambda, K_n, K_d) := \left\{M=  \lambda \mathbf 1\{S\times T\}:\quad  |S|= K_n, |T|=K_d\right\}\enspace .$$ 
As explained previously, we have  $\cH(\lambda, K_n, K_d)\subset \overline{\bbC}_{\iso}$. Next, we define the collection $\overline{\cH}(\lambda, K_n, K_d)$ of matrices that are lower bounded by a block matrix in $\cH(\lambda, K_n, K_d)$, that is 
$$\overline{\cH}(\lambda, K_n, K_d) = \left\{M\in  \overline{\bbC}_{\iso} ~\mathrm{s.t.}~\exists M_1\in\cH(\lambda, K_n, K_d)\, \, \mathrm{ with }\, \, M\geq M_1\right\}\enspace ,$$
where $M\geq M_1$ means that all the entries of $M$ are higher or equal to that of $M_1$. 
In what follows and when there is no ambiguity, we will write $\cH$ for $\cH(\lambda, K_n, K_d,n.d)$ and $\overline{\cH}$ for $\overline{\cH}(\lambda, K_n, K_d)$.

The following lemma actually states that any matrix $M\in \overline{\bbC}_{\iso}$ belongs to some  collection  $\overline{\cH}(\lambda, K_n, K_d)$ where 
where $\|M\|^2_F$ is not much greater than $\lambda^2 K_nK_d$.

\begin{lemma}\label{lem:peeling}
	Fix any matrix $M\in \overline{\bbC}_{\iso}$ and any positive integer $p$. 
    There exist $(\lambda , K_n,K_d)$ and a matrix $M' \in \cH(\lambda, K_n, K_d)$, such that $M\geq M'$ and   
    \[
    8 p \log_2(n\land d) \lambda^2K_nK_d + 2^{-2p} nd \geq \|M\|_F^2\enspace .     
    \]
\end{lemma}
This lemma is proved by a combination of a discretization --or peeling-- scheme together with a pigeonhole argument. See Figures~\ref{fig:sample_figure} and~\ref{fig:sample_figure2} for proof ideas. In those figures, the matrix $M$ belongs to $\bbC_{\iso}$ only for visualization purposes. The full proof is given in the appendix.

This lemma is pivotal on the one hand to provide intuitions of worst-case configurations for $M\in  \overline{\bbC}_{\iso}$ and on the other hand to deduce efficient ranking estimators from procedures dedicated to block matrices. Besides, the decomposition in Lemma~\ref{lem:peeling} can be applied to any  submatrix of $M$ opening the door towards hierarchical-type algorithms --see Section~\ref{sec:ranking}.

\begin{figure}
\centering
\begin{minipage}[c]{0.3\textwidth}
\centering
    \includegraphics[width=2.0in]{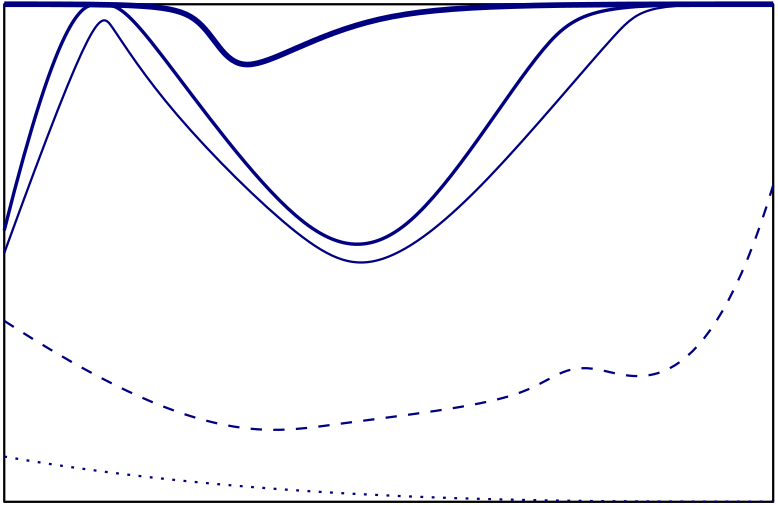}
\end{minipage}
\begin{minipage}[c]{0.3\textwidth}
\centering
    \includegraphics[width=2.0in]{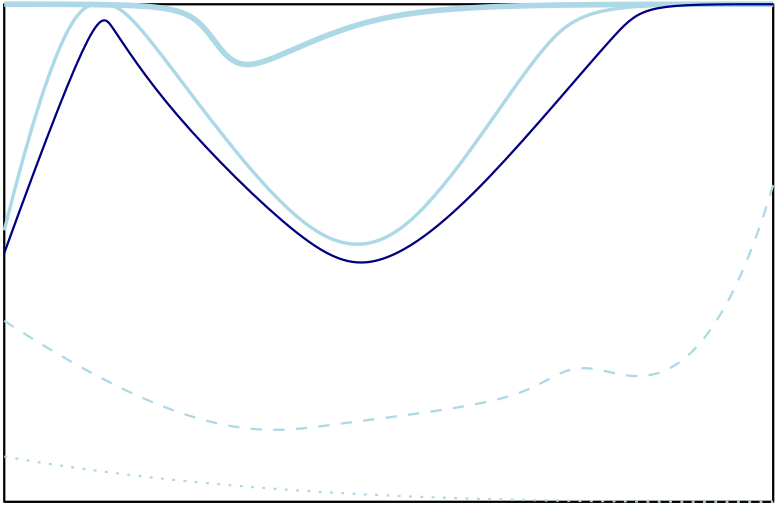}
\end{minipage}
\begin{minipage}[c]{0.3\textwidth}
\centering
    \includegraphics[width=2.0in]{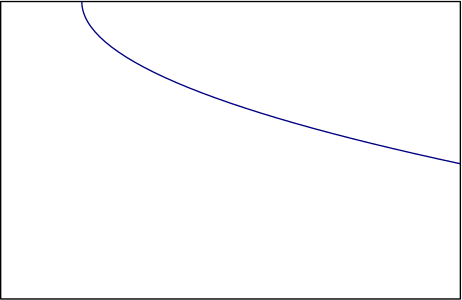}
\end{minipage}
    \caption{All pictures represent matrices of dimension $n \times d$, the rectangle's axis being resp.~the $j \in [d]$ index (abscissa) and the $i\in [n]$ index (ordinate). {\bf The first picture (left)} represents the level sets of the isotonic matrix $M\in \bbC_{\iso}$, on a logarithmic scale. Namely, the heavier line represents the $1/2$ level set --every entry up this line is higher than $1/2$--, the next line the $1/4$ level set, etc. We consider $p$ such level sets (in the picture $p=5$). Any entry between two level sets $u$ and $u+1$ has its value between $2^{-u}$ and $2^{-(u+1)}$. The remaining entries are all smaller than $2^{-p}$. {\bf In the second picture (middle)}, we choose the level $u^*$ set such that the number of entries above it times $2^{-2u^*}$ is maximized (in darker blue). This level set, thresholded at its minimal value $2^{-u^*}$ - which is then a lower bound on $M$ - contains a significant fraction of the $l_2$ norm of the matrix $M$, up to a $p$ factor and up to the remaining terms that lead to a norm of at most $nd 2^{-2p}$. We can therefore focus on this thresholded level set. {\bf The third picture (right)} represents the thresholded level set from the {\bf second picture} after a permutation of the $d$ lines of the matrix so that it becomes bi-isotonic - it is possible to do this as the thresholded level set only takes two values $\{0, 2^{-u^*}\}$ and is isotonic. }
 \label{fig:sample_figure}
\end{figure}

\begin{figure}
\centering
\begin{minipage}[c]{0.3\textwidth}
\centering
    \includegraphics[width=2.0in]{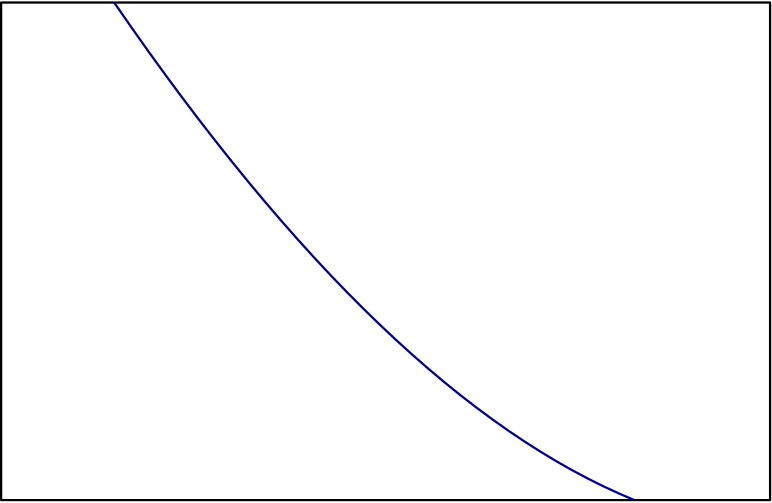}
\end{minipage}
\begin{minipage}[c]{0.3\textwidth}
\centering
    \includegraphics[width=2.0in]{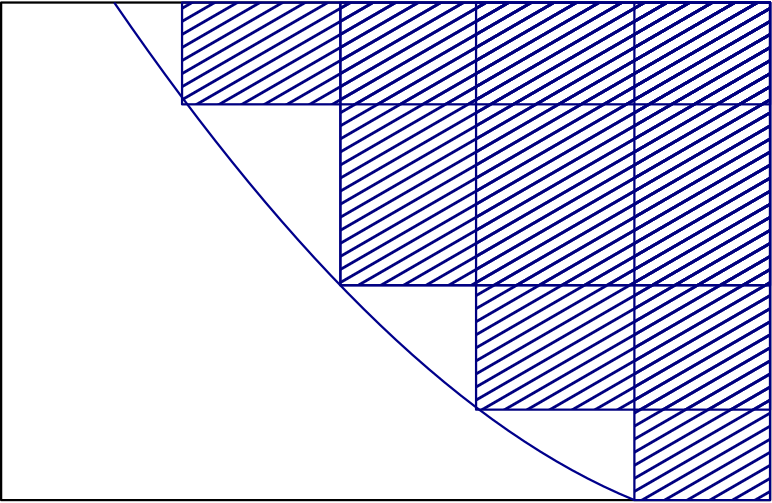}
\end{minipage}
\begin{minipage}[c]{0.3\textwidth}
\centering
    \includegraphics[width=2.0in]{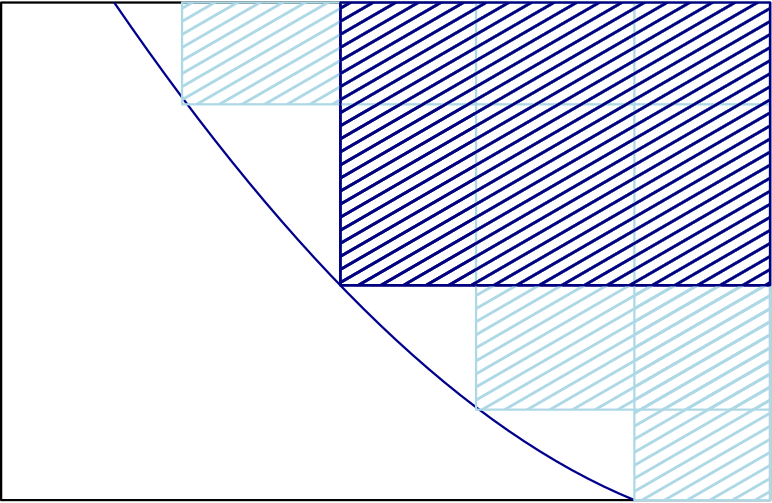}
\end{minipage}
    \caption{All pictures represent matrices of dimension $n \times d$, the rectangle's axis being resp.~the $j \in [d]$ index (abscissa) and the $i\in [n]$ index (ordinate). {\bf The first picture (left)} is the level set of a bi-isotonic matrix taking only two values, representing a permuted and thresholded level set of $M\in \bbC_{\iso}$ as in the third picture of Figure~\ref{fig:sample_figure}. {\bf In the second picture (middle)} - assuming w.l.o.g. $d \geq n$ - we insert in this level set $\lfloor \log_2(n)\rfloor$ rectangles of ordinate side $\{1,2,4,8,\ldots, 2^{\lfloor \log_2(n)\rfloor}\}$, and with one corner in the bottom right. These $\lfloor \log_2(n)\rfloor$ rectangles are each in the level set, and if we double their ordinate side, they completely cover the level set. For this reason {\bf in the third picture (right)}, if we take the one of these $\lfloor \log_2(n)\rfloor$ rectangles of maximal surface, we have a lower bound on the two-valued bi-isotonic matrix that is also capturing a $1/ \lfloor \log_2(n)\rfloor$ fraction of its norm - see rectangle in darker blue. This can be used as a lower bound of the permuted matrix $M$.}
 \label{fig:sample_figure2}
\end{figure}

\section{Submatrix detection and Ranking detection}\label{sec:detection}

As a first step, we consider the simpler problem of detecting whether all experts are similar and questions are similar. Denoting $M_0=0_{n\times d}$ for the null $n\times d$ matrix, this amounts to testing whether the matrix $M$ is equal to $M_0$ or not. As our ultimate objective is to estimate the matrix $M$ or an underlying ranking $\pi^*$, we shall slightly deviate from the typical perspective on testing problems and rather treat it as  a functional estimation problem. Define the functional 
\begin{equation}\label{eq:definition:x_0}
x_0 := \mathbf 1\{M \neq M_0\} \enspace .
\end{equation}
We quantify the quality of an estimator $f(Y)$ of $x_0$ using the quadratic loss $\mathbb E_M[(f(Y) - x_0)^2]$. If $f(Y)$ corresponds to a test, $\mathbb E_M[(f(Y) - x_0)^2]$ simply corresponds to the type I error probability when $M=M_0$ and a type II error probability for $M\neq M_0$.

In light of the reduction from $\overline{\bbC}_{\iso}$ to enveloppes $\overline{\cH}(\lambda,K_n,K_d)$ of collection of submatrices $H(\lambda,K_n,K_d)$, we focus on the risks $\mathbb E_M[(f(Y) - x_0)^2]$ under the null hypothesis $M=M_0$ and under alternative hypotheses of the form $\overline{\cH}(\lambda,K_n,K_d)$. This problem has attracted a lot of attention and tight upper and lower bounds for both Gaussian~\cite{butucea2013detection} and Bernoulli~\cite{arias2014community} version of this problems have been established. Besides, submatrix detection is the archetypical problem  for which computational-statistical gaps occur~\cite{ma2015computational,cai2020statistical,brennan2019universality,KuniskyWeinBandeira}. The main purpose of the section is to review optimal rates for this problem as well as optimal computational rates.

\subsection{Estimators of $x_0$.} 
First, we introduce several simple and, yet optimal estimators of $x_0$ for matrices $\cH \cup \{M_0\}$. They mimic known optimal tests for detecting a planted clique in a dense random graph which interprets as a symmetric counterpart of the submatrix detection problem. Fix a tuning parameter $\delta\in (0,1/2)$. 
The four following procedures respectively amount to detecting the signature of $M\in \overline{\cH}$ by computing the global sum ($f_{gs}$) of the entries of $Y$, the maximum row sum ($f_{\mathrm{rs}}$), the maximum column sum ($f_{\mathrm{cs}}$), or to scan the maximum sum over submatrices of size $m$  $(f_{\mathrm{ss},m})$. 
\begin{align*}
f_{\mathrm{gs}} = \mathbf{1}\left\{\sum_{i=1}^n\sum_{k=1}^{d} Y_{i,k} \geq \sqrt{2nd\log(1/\delta)}\right\} ;
f_{\mathrm{rs}} &= \mathbf{1}\left\{\sup_{i=1}^n \sum_{k=1}^d Y_{i,k} \geq \sqrt{2d\log(n/\delta)}\right\}\ ;  f_{\mathrm{cs}} = \mathbf{1}\left\{\sup_{k=1}^d \sum_{i=1}^n Y_{i,k} \geq \sqrt{2n\log(d/\delta)}\right\}\ ;\\
f_{\mathrm{ss},m} &= \mathbf{1}\left\{\sup_{S\subset [n], T\subset [d]:\  |S|\lor |T| \leq m}\sum_{(i,k)\in S\times T} Y_{i,k} \geq m\sqrt{2m\log(nd/\delta)}\right\}\enspace.
\end{align*}

\begin{proposition}\label{thm:UBtest}
    Fix any $m\leq n\wedge d$ and $\delta\in (0,1)$. First, for any $f= f_{\mathrm{gs}}$, $f_{\mathrm{rs}}$,  $f_{\mathrm{cs}}$, $f_{\mathrm{ss},m}$, we have $\E_{M_0}[(f- x_0)^{2}]\leq \delta$. For any $\lambda \in (0,1)$, $K_n$,  $K_d$,  and any $M\in \overline{\cH}(\lambda,K_n,K_d)$, we have 
	\begin{enumerate}
		\item $\E_{M} [(f_{\mathrm{gs}}-x_0)^2]\leq \delta$ if $\lambda \frac{K_dK_n}{\sqrt{nd}} \geq 2\sqrt{2\log(1/\delta)}$.
		\item $\E_{M} [(f_{\mathrm{rs}}-x_0)^2]\leq \delta$ if~~~$\lambda \frac{K_d}{\sqrt{d}} \geq 2\sqrt{2\log(n/\delta)}$~~~~and $\E_{M} [(f_{\mathrm{cs}}-x_0)^2]\leq \delta$ if~~~$\lambda \frac{K_n}{\sqrt{n}} \geq 2\sqrt{2\log(d/\delta)}$.
		\item $\E_{M} [(f_{\mathrm{ss},m}-x_0)^2]\leq \delta$ if~~~$\lambda \sqrt{m} \geq 2\sqrt{2\log(nd/\delta)}$ and $m \leq K_n \land K_d$.
	\end{enumerate}
\end{proposition}
Such results are classical --see e.g.~\cite{butucea2013detection,ma2015computational}-- and are easily shown using Bernstein's inequality. Still, for the sake of completeness, a proof is provided in the appendix. 

\begin{corollary}\label{cor:est}
Consider any $(\lambda,K_n,K_d,m)$ and any $\delta\in (0,1)$.  
Assume that
\begin{equation}\label{eq:condition_minimax_detection}
\left[\lambda \frac{K_dK_n}{\sqrt{nd}}\right] \lor \left[\lambda \frac{K_d}{\sqrt{d}} \right]\lor \left[\lambda \frac{K_n}{\sqrt{n}} \right] \lor  \left[\lambda \sqrt{m\land K_n\land K_d} \right]\geq 2\sqrt{2\log(nd/\delta)}\enspace . 
\end{equation}
For any integer $1 \leq m \leq n\land d$, the estimator $f_{m,K_n,K_d} = f_{\mathrm{gs}}\lor f_{\mathrm{rs}}\lor  f_{\mathrm{cs}}\lor f_{\mathrm{ss},m\land K_n \land K_d}$ satisfies
$$\sup_{M\in \overline{\cH}(\lambda,K_n,K_d) \cup \{M_0\}}\E_M[(f_{m,K_n,K_d}-x_0)^2] \leq 4\delta.$$
\end{corollary}
The first three procedures $f_{\mathrm{gs}}$, $f_{\mathrm{rs}}$, and  $f_{\mathrm{cs}}$ are computed in less than $2nd$ operations. In contrast, a naïve implementation of the scan test  $f_{\mathrm{ss},m}$  requires  at least $\binom{n}{m}\binom{d}{m}$ operations which is of the order of $(nd)^m$ if $m$ is much smaller than $n\wedge d$. This raises the twin  questions (i) whether, for $m=K_n\wedge K_d$, Condition~\eqref{eq:condition_minimax_detection} can be achieved by a computationally efficient procedure $f$ or (ii) of if there exists an intrinsic barrier for polynomial-time procedures. 
In the next subsection, we discuss this optimality by establishing minimax and low-degree polynomial lower bounds.

\subsection{Minimax and LD lower bounds.}

Fix $(\lambda,K_n,K_d)$. As usual in the field, we adopt the minimax paradigm and we consider the maximum risk $\sup_{M\in \overline{\cH} \cup \{M_0\}}\E_M[(f-x_0)^2]$ of a procedure $f$. We aim at characterizing the smallest maximum risk over all procedures $f$ or over related subclasses of polynomials. 
For the former, the standard approach known as the second moment moment or the  generalized Le Cam's method~\cite{tsybakov} amounts to introducing a prior distribution on the parameters. Here, we choose $\mu$ as the uniform distribution over $\cH$. We also write $\mathbb P_{M\sim \mu}$ for the marginal distribution of $Y$ when $M$ follows $\mu$. Then, Le Cam's method characterizes the worst-case risk in terms of the $\chi^2$ discrepancy between $\mathbb P_{M\sim \mu}$ and $\mathbb{P}_{M_0}$. Here, we follow a different path to also encompass the LD lower bound. Given an estimator $f$, we define the advantage of $f$ by 
$$\mathrm{Adv}(f) := \frac{\mathbb E_{M\sim \mu}f}{  \mathbb E^{1/2}_{M_0} f^2}\enspace ,$$
with the conventions $0/0=1$. 
The following result mimics Le Cam's Lemma and  lower bounds the worst-case quadratic loss for any estimator $f$ in terms of $\mathrm{Adv}(f)$.
\begin{lemma}\label{lem:lower_bound_loss}
We have
$$\sup_{M\in \cH \cup \{M_0\}}\E_M[(f-x_0)^2] \geq \frac{1}{2(1+ \mathrm{Adv}^2(f))} \geq  \frac{1}{4}\left[1 - (|\mathrm{Adv}(f)|-1)\right]\enspace .$$
\end{lemma}
The  constant estimator $f=1/2$ a.s. satisfies $\E_M[(f-x_0)^2]=1/4$. Hence, the above lemma states that the maximum risk is no smaller than that of constant estimator over a collection $\mathcal{F}$ of estimators if $\max_{f\in \cF}\mathrm{Adv}(f)$ is close to one.

Given a positive integer $D$, we define $\mathcal P_{\leq D}$ as the collection of polynomials of degree at most $D$ with respect to the entries of $Y$ and we define the maximum advantage~\cite{KuniskyWeinBandeira}
\[
\mathrm{Adv}_{\leq D}= \sup_{f\in \cP_{\leq D}}\mathrm{Adv}(f)\enspace . 
\]
 Since the entries of $Y$ are in $\{-1,1\}$, any function $f$ of $Y$ writes as a polynomial of degree $nd$. As a consequence, it suffices to control $\mathrm{Adv}_{\leq D}$ for all $D=nd$ to lower bound the minimax risk $\inf_{f \text{measurable}}\sup_{M\in \cH \cup \{M_0\}}\E_M[(f-x_0)^2]$, whereas bounds of $\mathrm{Adv}_{\leq D}$ for smaller $D$ allow to lower bound the risk of low-degree polynomials and provide evidence of computational-statical gaps.

Given a subset $S\subset [n]\times [d]$, we define the monomial $Y^S = \prod_{(i,k)\in S} Y_{i,k}$ with the convention $Y^{\emptyset}=1$. Equipped with this notation, it is clear that the family $(Y^{S})_{S\subset [n]\times [d]: |S| \leq D}$ is a basis of $\cP_{\leq D}$. Moreover, under $\mathbb{P}_{M_0}$, the entries of $Y$ follow independent Rademacher distributions. As a consequence, for any set $S$ and $S'$, we have $\mathbb{E}_{M_0}[Y^{S}Y^{S'}]= \mathbf{1}\{S= S'\}$ and the family 
$(Y^{S})_{S\subset [n]\times [d]: |S| \leq D}$ is therefore orthonormal under $\mathbb{P}_{M_0}$. Given $S\subset [n]\times [d]$, we write $\omega_S= \mathbb{E}_{M\sim \mu}[Y^{S}]$ and we define the vector $\omega=(\omega_S)_{S:\ |S|\leq D}$. Equipped with this notation, we obtain the following simple form for $\mathrm{Adv}_{\leq D}$.
\begin{align}\label{eq:expression_ADV_D}
\mathrm{Adv}_{\leq D}= \max_{\alpha =(\alpha_{S})_{|S|\leq D}} \frac{\sum_{S: |S|\leq D}\alpha_S\mathbb{E}_{M\sim \mu}[Y^{S}]}{\sqrt{\sum_{S}\alpha_S^2}}=\sqrt{\sum_{S: |S|\leq D}\mathbb{E}^2_{M\sim \mu}[Y^{S}]}\  ,
\end{align}
by Cauchy-Schwarz equality. As a consequence, it suffices to control first moments $\mathbb{E}_{M\sim \mu}[Y^{S}]$. Such direct computations are done e.g.~in~\cite{KuniskyWeinBandeira} for related models. In this manuscript, instead of directly bounding $\sum_{S: |S|\leq D}\mathbb{E}^2_{M\sim \mu}[Y^{S}]$, we take a little detour and introduce a permutation-invariant basis. This does not  simplify much the arguments for controlling $\mathrm{Adv}_{\leq D}$ but this will serve as an important warm-up step towards establishing low-degree lower bounds in submatrix estimation in the next section.

\subsubsection{Permutation-invariant polynomials and Bi-partite graph formalism.}

Both the distribution $\mathbb{P}_{M_0}$ and the distribution $\mathbb{P}_{M\sim \mu}$ are invariant by any permutation of the rows of $Y$ and of the columns of $Y$. As a consequence, it turns out that the maximum advantage $\mathrm{Adv}_{\leq D}$ is achieved by a \emph{permutation-invariant} polynomial and we can restrict our attention to this subclass of polynomials. Before formalizing this property, we introduce a suitable collection of permutation-invariant polynomial. 

Henceforth, we define bi-partite graphs $G= (V,W,E)$ where $V=\{v_1,\ldots v_{r}\}$, $W=\{w_1,w_2,\ldots, w_s\}$ are two nodes sets and where $E\subset V\times W$ is the edge set.
Two such bipartite graphs $G^{(1)}=(V^{(1)},W^{(1)},E^{(1)})$ and  $G^{(2)}=(V^{(2)},W^{(2)},E^{(2)})$ are said to be isomorphic if there exists two bijections $\tau_v: V^{(1)} \mapsto V^{(2)}$, $\tau_w: W^{(1)} \mapsto W^{(2)}$ such that that $\tau_v, \tau_w$ preserve the edges. 
Such a couple of bijections $(\tau_v,\tau_w)$ is called an automorphism of $G^{(1)}$ if $G^{(1)}=G^{(2)}$. In what follows, we write $\mathrm{Aut}(G)$ 
for the  automorphism group of a bi-partite graph $G$.

Let $\mathcal{G}_{\leq D}$ be any maximum collection of bi-partite graphs $G=(V,W,E)$ such that such (i) $G$ does not contain any isolated node,  (ii) $|E|\leq D$, and (iii) no two graphs in $\mathcal{G}_{\leq D}$ are isomorphic. In fact, $\mathcal{G}_{\leq D}$ corresponds to the collection of equivalence classes of bipartite graphs with at most $D$ edges and without isolated nodes. Henceforth, we refer to $\mathcal{G}_{\leq D}$ as the collection of \emph{templates}.

Consider a template $G= (V,W,E)$ where $V=\{v_1,v_2,\ldots, v_r\}$, $W=\{w_1,w_2,\ldots, w_s\}$. Let $\GS_V$ (resp. $\GS_W$) denote the set of injective mapping from $V$ to $[n]$ (resp. from $W\rightarrow [d]$). For $\sigma_r\in \GS_V, \sigma_c\in \GS_W$, we define the polynomials
\begin{equation}\label{eq:definition:P_G}
P_{G,\sigma_{r},\sigma_c}(Y)= \prod_{(i,k)\in E} Y_{\sigma_r(i),\sigma_c(k)} ; \quad \quad P_G = \sum_{\sigma_r\in \GS_V, \sigma\in \GS_W} P_{G,\sigma_{r},\sigma_c}.
\end{equation}
For short, we sometimes write $P_G$ for $P_G(Y)$ when there is no ambiguity. Obviously, $P_G$ is invariant to any permutation of the rows of $Y$ or the columns of $Y$. Since, under $\mathbb{P}_{M_0}$, the entries of $Y$ follow independent Rademacher distributions, we derive that $\mathbb{E}_{M_0}[P_G^2]= |\GS_V||\GS_W||\mathrm{Aut}(G)|$. With this observation, we obtain the normalized polynomials
\begin{equation}\label{eqn:variance of graph}
\Psi_G := \frac{P_G}{\sqrt{\mathbb V(G)}},~~~~~\mathrm{where}~~~\mathbb V(G) = \frac{n!}{(n-|V|)!} \frac{d!}{(d-|W|)!} |\mathrm{Aut}(G)|\enspace . 
\end{equation}
It follows from these definitions that $(1,(\Psi_G)_{G\in \cG_{\leq D}})$ is an orthonormal family of polynomials as summarized in the following lemma. 
\begin{lemma}\label{lem:ortho1}
For any $G^{(1)},G^{(2)} \in \mathcal G_{\leq D}$, we have $\mathbb E_{M_0}[\Psi_{G^{(1)}}]=0$ and $\mathbb E_{M_0}[\Psi_{G^{(1)}} \Psi_{G^{(2)}}]= \1\{G^{(1)}=G^{(2)}\}$.
\end{lemma}
The next lemma build upon the permutation invariance of both $\mathbb{P}_{M_0}$ and $\mathbb{P}_{M\sim \mu}$ to characterize $\mathrm{Adv}_{\leq D}$. 
\begin{lemma}\label{lem:reduction::permutation}
Fix any positive integer $D\in [0,nd]$. Then, we have 
\begin{align}\label{eq:badv}
\mathrm{Adv}^2_{\leq D}= 1 + \sum_{G\in \cG_{\leq D}} \mathbb{E}_{M\sim \mu} [\Psi^2_G] \enspace .
\end{align}
\end{lemma}

\subsubsection{Bound on the worst-case loss for fixed degree polynomials}

As a consequence of Lemma~\ref{lem:reduction::permutation} as well of Lemma~\ref{lem:lower_bound_loss}, we only have to control explicitly $\sum_{G\in \cG_{\leq D}} \mathbb{E}_{M\sim \mu} [\Psi^2_G]$ as a function of $\lambda$, $K_n$, and $K_d$.

\begin{theorem}[Fixed-degree polynomial lower bound]\label{thm:worst_case_loss}
Consider any $(\lambda,K_n, K_d)$ and any positive degree $D\leq nd$. Assume that, for $\overline{c}\geq 10$, we have 
	\begin{align}\label{eq:assump}
\left[ \lambda \sqrt{D\land K_n\land K_d}\right] \lor  {\left(\lambda \frac{K_d}{\sqrt{d}}\right)} \lor {\left(\lambda \frac{K_n}{\sqrt{n}}\right)} \lor\left[\lambda \frac{K_nK_d}{\sqrt{nd}}\right] \leq  2^{-2- \bar c}\enspace ,
\end{align} 
Then, we have
$$ \mathrm{Adv}_{\leq D}^2(f) \leq  1+ 2^{-  \bar c},~~~~~\mathrm{so~that}~~~~~\inf_{f\in \mathcal P_{\leq D}} \sup_{M\in \cH \cup \{M_0\}} \E_M[(f-x_0)^2] \geq \frac{1}{4} [1 - 2^{- \bar c}]\enspace . $$
\end{theorem}
The proof of this theorem is postponed to the appendix. As discussed previously, two specific choices of $D$ are of particular interest: $D=nd$ for establishing a minimax lower bound and the case $D$ of the order of $\log(nd)$ which, according to the low-degree conjecture, provides a strong indication to rule out all polynomial-time algorithms. We would like to emphasize that these lower bounds results are proven on the sub-model $\cH$, and that they are then a fortiori valid for the larger $\overline{\cH}$ that are more akin to ranking.

\begin{corollary}[Minimax lower bound]
		Assume that, for $\overline{c}\geq 10$, we have 
	\begin{align}\label{eq:assump2}
\left[\lambda\sqrt{K_n\land K_d}  \right] \lor  {\left(\lambda \frac{K_d}{\sqrt{d}}\right)} \lor {\left(\lambda \frac{K_n}{\sqrt{n}}\right)} \lor\left[\lambda \frac{K_nK_d}{\sqrt{nd}}\right] \leq  2^{-2- \bar c}\enspace . 
\end{align} 
This implies that 
$$\inf_{f~\mathrm{measurable~in~}Y}  \sup_{M\in \cH \cup \{M_0\}} \E_M[(f-x_0)^2] \geq \frac{1}{4} [1 - 2^{- \bar c}]\enspace .$$
\end{corollary}
Here,~\eqref{eq:assump2} matches (up to numerical constants) the sufficient condition~\eqref{eq:condition_minimax_detection} for detection and is therefore optimal. Note that an asymptotic lower bound with tight constants (in some regimes) is provided in~\cite{butucea2013detection}, but this requires more delicate arguments than that of Lemma~\ref{lem:lower_bound_loss}.

\begin{corollary}[LD polynomial lower bound]
		Assume that for $\bar c \geq 10$. 
	\begin{align}\label{eq:assump3}
 \lambda \lor {\left(\lambda \frac{K_d}{\sqrt{d}}\right)} \lor {\left(\lambda \frac{K_n}{\sqrt{n}}\right)} \lor\left[\lambda \frac{K_nK_d}{\sqrt{nd}}\right] \leq  2^{-2- 2\bar c} \log(nd)^{-1/2}\enspace ,
\end{align} 
This enforces 
$$\inf_{f\in \mathcal P_{\leq 2^{2\bar c} \log(nd)}}  \sup_{M\in \cH \cup \{M_0\}} \E_M[(f-x_0)^2] \geq \frac{1}{4} [1 - 2^{- \bar c}]\enspace .$$
\end{corollary}
This result enforces, that in the low-degree polynomial framework, up to logarithms, the aggregation of the global-sum, row-sum, and the column-sum test $f_{\mathrm{gs}}\vee f_{\mathrm{rs}}\vee f_{\mathrm{rs}}$ is nearly optimal. In the square-case $n=d$, a tighter result is already established in~\cite{KuniskyWeinBandeira}. Also for square matrices, other works~\cite{ma2015computational,brennan2019universality} have also established the optimality of Condition~\eqref{eq:assump3} by a careful reduction to the planted clique conjecture. However, it seems that one current limitation of reduction-to-planted clique techniques, is that they do not allow to handle highly rectangular matrices.

Theorem~\ref{thm:worst_case_loss} can also be used to establish impossibility results for other problems such that of reconstructing the matrix $M$ or that ranking $M$. Unfortunately, the corresponding lower bounds are not always tight because, in some regimes, detection is easier than estimation. We come back to this in the next section.

\subsection{Global separation distance over $\overline{\bbC}_{\iso}$.}

In this subsection, we simply combine our analysis of the submatrix detection problems with the reduction argument of Section~\ref{sec:peeling} - to characterize the \emph{minimax separation distance} in Frobenius norm for testing $M=M_0$ in the general crowd-sourcing model $M\in \overline{\bbC}_{\iso}$. Given $\rho>0$, define the collection of permuted isotonic matrices whose Frobenius norm is at least $\rho$. 
$$
\overline{\bbC}_{\iso}(\rho) =\{ M\in \overline{\mathbb{C}}_{\iso},\,\mathrm{and}~\|M\|_F \geq \rho \}\enspace .
$$
From this, we can characterize the \textit{separation distance} $\rho_\epsilon(f)$ that is needed by $f$ to guarantee a worst-case loss of at most $\epsilon>0$ over $ \{M_0\} \cup \overline{\bbC}_{\iso}(\rho)$ in the estimation of $x_0$
$$\rho_\epsilon(f) = \inf\left\{\rho>0: \sup_{M \in  \{M_0\} \cup \overline{\bbC}_{\iso}(\rho)} \E_M[(f-x_0)^2] \leq \epsilon\right\}\enspace .$$
Taking e.g.~if $n\geq d$, the submatrix model with $K_n = \sqrt{n}\sqrt{D\land \sqrt{d}}, K_d = \sqrt{d}, \lambda = 2^{-24}/\sqrt{D \land \sqrt{d}}$, Theorem~\ref{thm:worst_case_loss} implies that, for any $D\geq 1$ and any $\epsilon<1/4$, we have 
$$\inf_{f\in \mathcal P_{\leq D}}\rho^2_\epsilon(f) \geq c\frac{\sqrt{nd}}{\sqrt{D \land \sqrt{n \land d}}} \enspace ,$$
where $c>0$ is a positive numerical constant.

Conversely, for a fixed $m$, we  consider dyadic collection $\mathcal{K}_n:=\{1,2,4,\ldots,2^{\floor \log_2(n)\rfloor}\}$ and $\mathcal{K}_d:=\{1,2,4,\ldots,2^{\floor \log_2(d)\rfloor}\}$ for $K_n$ and $K_d$ and compute the aggregated statistic $f_m=\vee_{K_n\in \mathcal{K}_n, K_d\in \mathcal{K}_d}f_{m,K_n,K_d}$ with $\delta\leq \epsilon/4$. Then, combining Corollary~\ref{cor:est} and Lemma~\ref{lem:peeling}, we deduce that $f_m$ (computable with less than  $c\log^2(nd)(nd)^m$ elementary operations), satisfies
$$\rho^2_\epsilon(f_m) \leq c'\frac{\sqrt{nd}}{\sqrt{m\land \sqrt{n\land d}}} \log^2(1/\delta) \log^3_2(nd)\enspace ,$$
where $c'$ is a positive numerical constant.

\medskip

This highlights that, contrary to problem of ranking  - as we will see later in Sections~\ref{sec:rankingrest} and~\ref{sec:ranking} - there is evidence for a statistical-computational gap for detection in $\overline{\bbC}_{\iso}$. In particular, the optimal separation required by a measurable estimator is of order $\sqrt{nd}/(n\land d)^{1/4}$ up to logarithmic terms in $n,d$ while we established a LD polynomial lower bound of the order of  $\sqrt{nd}$ up to logarithmic terms in $n,d$ (based on the lower bound for $D$ of order $\log(nd)$).

\section{Ranking the experts: warm-up over submatrix estimation}\label{sec:rankingrest}

As an intermediary step between detecting the non-nullity of $M$ and ranking general permuted isotonic matrices in $\overline{\mathbb C}_{\mathrm{iso}}$, we restrict our attention in this section to estimation problems for block matrices belonging to collections $\cH(\lambda,K_n,K_d)$. For such matrices, estimating  a permutation $\pi^*$ that is compatible with $M$ is equivalent to selecting the non-zero rows/experts of $M$. Although this slightly departs from the classical submatrix localization/estimation problem~\cite{hajek2018submatrix,butucea2015sharp,cai2017computational} where the objective is to recover the matrix $M$, most of the techniques are quite similar --we shall further discuss this at the end of the section.   

Without loss of generality, we focus on the first row and we consider  the problem of estimating the function $x^{\star}$ defined by 
$$x^{\star} := \mathbf 1\{\exists k: M_{1,k} >0 \}\ . $$
It is one if and only if the first rows belongs to the hidden block of $M$. As in the previous case, we consider the quadratic risk $\mathbb E_M[(f(Y) - x^{\star})^2]$.

\subsection{Estimators and upper bounds.} 

Fix a tuning parameter $\delta\in (0,1/2)$. Since estimating $x^\star$ amounts to testing whether the first row is non-zero, a first simple estimator computes the row sums of $M$ and selects those that are unusually large. Define
$f_{SE} = \mathbf 1\{i\in [n]: \sum_{k=1}^d Y_{1,k} \geq \sqrt{2d\log(n/\delta)}\}$. An alternative strategy amounts to first selecting relevant columns with unusually large entries and then to use them to detect the nullity of $M$. 
\[
\hat S_Q = \left\{k\in [d]: \sum_{i=2}^n Y_{i,k} \geq \sqrt{2n\log(d/\delta)}\right\}\quad \quad \quad f_{SE,2}= \1\left\{\sum_{k\in \hat S_Q} Y_{1,k} \geq \sqrt{2|\hat S_Q|\log(n/\delta)}\right\}\enspace . 
\]
Finally, similarly to the scan test introduced in the previous section, we can  select relevant questions by summing over submatrices. Fix $m>0$ a positive integer.
\[
\hat S_Q^{(m)}\in  \arg\max_{T\subset [d]: |T|=m}\{\max_{S\subset [2,n], \ |S|= m} \sum_{(i,k)\in S\times T}Y_{i,j}\}\ ; \quad \quad  f^{(m)}_{SE,2}= \1\left\{\sum_{k\in \hat S^{(m)}_Q} Y_{1,k} \geq \sqrt{2|\hat S^{(m)}_Q|\log(n/\delta)}\right\}\enspace . 
\]

\begin{proposition}\label{prop:ubestM}
    Consider any $(\lambda,K_n,K_d)$ and any $M \in \cH(\lambda,K_n,K_d)$. 
	\begin{enumerate}
		\item If $\lambda \frac{K_d}{\sqrt{d}} \geq 2\sqrt{2\log(n/\delta)}$,then $\E_{M}[(f_{SE}-x^{\star})^2] \leq \delta$.
		\item If $K_n\geq 2$, $\lambda \frac{K_n-1}{\sqrt{n}} \geq 2\sqrt{2\log(d/\delta)}$, and $\lambda \sqrt{K_d} \geq 2\sqrt{2\log(n/\delta)}$, then $\E_{M}[(f_{SE,2}-x^{\star})^2] \leq 2\delta$.
		\item If $K_n\geq 2$, $\lambda\sqrt{m} \geq 4\sqrt{2\log(nd/\delta)}$ then $\E_{M}[(f_{SE,2}^{(m)}-x^{\star})^2] \leq 2\delta$, for $m \leq (K_n-1) \land K_d$.
	\end{enumerate}
\end{proposition} 
This proposition is proved in the appendix. Although $f_{SE}$ and $f_{SE,2}$ have a linear complexity, this is not the case of $f_{SE,2}^{(m)}$ whose naïve implementation requires of the order $(nd)^{m}$ operations. 
\begin{corollary}\label{cor:UB:estimation}
Fix any positive integer $m\leq n\land d$ and assume that
\begin{align}\label{eq:upper_bound_estimation}
\left[\left(\lambda \frac{K_n}{\sqrt{n}}\right) \land (\lambda\sqrt{K_d})\right]\lor \left[\lambda \frac{K_d}{\sqrt{d}} \right] \lor  \left[\lambda \sqrt{m\land (K_n-1)\land K_d} \right]\geq 2\sqrt{2\log(nd/\delta)}\enspace .
\end{align} 
The estimator $f_{\star}^{(m)} = f_{SE} \lor f_{SE,2} \lor  f_{SE,2}^{(m\land K_n\land K_d)}$ satisfies
\[
\sup_{M\in \cH(\lambda,K_n,K_d)}\E_{M}[(f_{\star}^{(m)}-x^{\star})^2] \leq 5\delta.
\]
\end{corollary}

Extending this analysis to the worst case of $K_n,K_d,\lambda$ and relying on $f_{SE}$ and $f_{SE,2}$, we easily deduce the following bound for the ranking error of block matrices.
\begin{corollary}\label{cor:UBinfM}
	There exists a linear-time  estimator $\hat \pi$ of $\pi^*$ that satisfies
$$\sup_{K_n,K_d,\lambda}\sup_{M\in \cH(K_n,K_d,\lambda)}\mathbb E_M\|M_{\hat \pi^{-1}} - M_{\pi^{*-1}}\|_F^2 \leq c(n+\sqrt{nd})\log(nd)\enspace , $$
\end{corollary}
where $c$ is a numerical constant.

\subsection{Informational lower bounds.}

To establish a matching information bound, we consider, for the sake of simplicity, the  Bayesian model  $M= \lambda\1_{S_n\times T_n}$ where the random variables $\1\{i\in S_n\}$ and $\1\{j\in T_n\}$ respectively follow independent Bernoulli distribution with parameters $K_n/n$ and $K_d/d$. As a consequence, the matrix $M$ is a block matrix with size close to $(K_n,K_d)$. We write $\E[.]$ for the expectation when $M$ follows this prior distribution.

\begin{proposition}[Informational lower bound for estimating $x^{\star}$]\label{prop:LBinfestM}
There exists a numerical constant $c_0$ such that the following holds.  Assume that $\lambda \leq 1/2$,  $K_n \leq n/2$ and
\begin{align}\label{eq:condition:info:lower:estion}
\left[\lambda \sqrt{ K_n\land K_d} \right] \lor \left[\lambda \frac{K_d}{\sqrt{d}} \right]\leq c_0\enspace .
\end{align}
Then, for any estimator $f$, we have $\E[(f-x^{\star})^2] \geq \E[x^{\star 2}]/32$.
\end{proposition}
The condition~\eqref{eq:condition:info:lower:estion} matches (up to polylogarithmic terms) the sufficient condition for $f_{\star}^{(K_n\wedge K_d)}$ to have a small risk over $\cH(\lambda,K_n,K_d)$ and therefore characterizes minimax risk.
We readily deduce from that a lower bound for the minimax maximum ranking error over all submatrices models. 
\begin{corollary}\label{cor:LBinfM}
There exists a universal constant $c>0$ such that
$$\inf_{\hat \pi~\mathrm{measurable}}\sup_{K_n,K_d,\lambda}\sup_{M\in \cH(K_n,K_d,\lambda)} \mathbb E_M\|M_{\hat \pi^{-1}}-M_{\pi^{*-1}}\|_F^2\geq c(n+\sqrt{nd})\enspace .$$
\end{corollary}
This corollary is deduced fom Proposition~\ref{prop:LBinfestM} by taking e.g.~$K_n = n/2$ and $K_d = \sqrt{nd} \land d$, and $\lambda$ of order $(nd)^{-1/4} \lor n^{-1/2}$. In combination with~Corollary~\ref{cor:UBinfM}, we deduce  there is no significant statistical-computational gap for estimating $\pi^*$ over the collection $\bigcup_{K_n,K_d,\lambda} \cH(K_n,K_d,\lambda)$.

\subsection{LD lower bounds.}

In Corollary~\ref{cor:UB:estimation}, the only procedure that manages to achieve a small risk for estimating $x^{\star}$ up to the boundary given by~\eqref{eq:condition:info:lower:estion} suffers from an exponential-time computational complexity. In fact, this corollary suggests that, for estimating $x^{\star}$ over $\cH$, polynomial-time estimators may require a signal condition of the form $\left[\left(\lambda \frac{K_n}{\sqrt{n}}\right) \land (\lambda\sqrt{K_d})\right]\lor \left[\lambda \frac{K_d}{\sqrt{d}} \right]$ where is large. Unfortunately, we cannot deduce a LD lower bound from the detection problem in Section~\ref{sec:detection} because, in some regimes (e.g. $K_d\geq\sqrt{d}$ and $K_n\geq \sqrt{n}$), it is significantly easier to detect the non-nullity of $M$ than to estimate the support. In such situations where a detection-to-estimation gap occurs, one therefore has to establish a dedicated computational lower bound.

In the closely related problem of submatrix estimation/localization, such lower bounds have been obtained by reduction to planted clique~\cite{cai2017computational,brennan2018reducibility} at least when the matrix $M$ is almost square. Other matching lower bounds have been obtained using the statistical queries paradigm~\cite{feldman2017statistical} or the overlap gap property~\cite{gamarnik2021overlap}. In the LD polynomial framework, techniques for establishing estimation lower bounds have been first introduced  in the seminal work~\cite{SchrammWein22} and have been later refined by~\cite{SohnWein25}. In this subsection, we follow the LD paradigm but we take a different proof approach.

As in the previous subsection, we consider the Bayesian model where $M=\lambda \mathbf{1}\{S\times T\}$ where $|S|\sim \mathrm{Bin}(n,K_n/n)$ and $|T|\sim \mathrm{Bin}(d,K_d/d)$. Following~\cite{SchrammWein22}, we define 
$$\mathrm{Corr}^{\star}(f) = \frac{\mathbb E [f x^{\star}]}{\sqrt{\mathbb E[f^2]}}\ ; \quad \quad \quad \mathrm{Corr}^{\star}_{\leq D}= \sup_{f\in \mathcal{P}_{\leq D}}\mathrm{Corr}^{\star}(f) 
\enspace .$$
Indeed, similarly to the advantage ($\mathrm{Adv}(.)$) for detection, the correlation $\mathrm{Corr}^{\star}(f)$ satisfies the following property.
\begin{lemma}\label{lem:lower_bound_loss2}
We have
$$\E[(f-x^{\star})^2] \geq \E[x^{\star 2}] - \mathrm{Corr}^{\star 2}(f)\enspace .$$
\end{lemma}
In the previous section, we derived an explicit form for $\mathrm{Adv}^2_{\leq D}$ relying on an orthonormal family with respect to $\mathbb{E}_{M_0}$. In contrast, the definition of $\mathrm{Corr}_{\leq D}^*$ involves in the denominator the expectation $\mathbb E[f^2]$ with respect to the mixture distribution $\mathbb{P}$. Unfortunately, there is no known simple orthonormal family with respect to this distribution.  For this reason, \cite{SchrammWein22} and~\cite{SohnWein25} introduced several work-arounds to upper bound $\mathrm{Corr}_{\leq D}^*$-- those mainly involve to solve  an over-complete system of linear equations. This approach led to  a flourishing line of literature and allowed to establish tight LD lower bounds for submatrix estimation, clustering in stochastic block models~\cite{SohnWein25,luo2023computational,pmlr-v291-chin25a}, graph coloring~\cite{kothari2023planted}, Gaussian mixtures models~\cite{Even24,Even25a}, or tensor decomposition~\cite{wein2023average}. On the downside, it requires technical virtuosity that may obscure important statistical quantities. Here, we rather follow the more constructive approach introduced in~\cite{CGGV25} by constructing an almost-orthonormal family with respect to $\mathbb{E}[.]$.

\subsubsection{Construction of an almost-orthonormal family of polynomials.}

For detection, we introduced in the previous section, the family $(P_G)_{G\in \cG_{\leq D}}$ of permutation-invariant polynomials. However, this family is not suited for our purpose here. First, the distribution of $(Y,x^{\star})$ is only invariant up to permutations of the columns of $Y$ and up to permutations of rows of $Y$ that let the first row fixed. This is easily dealt with a slight modification of the collections of labelings and of templates. Second and more seriously, the family $P_{G}$ is far from being orthogonal. Indeed, given templates $G^{(1)}$ and $G^{(2)}$ and labelings $\sigma^{(1)}$ and $\sigma^{(2)}$,  $P_{G^{(1)},\sigma^{(1)}_r,\sigma^{(1)}_c}$ and $P_{G^{(2)},\sigma^{(2)}_r,\sigma^{(2)}_c}$ are positively correlated under $\mathbb{E}$ as soon as  the two corresponding labelled graphs have common nodes. The resulting correlation between $P_{G^{(1)}}$ and $P_{G^{(2)}}$ will be too large to ensure the almost-orthogonality of the family. To fix this issue, we modify the polynomials by centering each connected component as proposed in~\cite{CGGV25,kunisky2024tensor,montanari2025equivalence}.

Consider a template $G= (V,W,E)$ where $V=\{v_1,v_2,\ldots, v_r\}$, $W=\{w_1,w_2,\ldots, w_s\}$ with at least one edge and without isolated nodes to the possible exception of $v_1$. Let $\GS_V^{\star}$ the set of injective mapping $\sigma_r$ from $V\rightarrow [n]$ such that $\sigma_r(v_1)=1$. Decompose $G$ into $G_{1},\ldots, G_{cc}$ its $cc$ connected components with at least one edge. Given labelings $\sigma_r$ and $\sigma_c$, we define the polynomials
\begin{equation}\label{eq:definition:P_G:est}
    \overline{P}^{\star}_{G, \sigma_r,\sigma_c}:=\prod_{l=1}^c \left[P_{G_l,\sigma_r,\sigma_l} - \mathbb{E}[P_{G_l,\sigma_r,\sigma_l} ]\right] \ , \quad \quad 
\overline{P}^{\star}_G = \sum_{\sigma_r\in \GS_V^{\star}, \sigma_c\in \GS_W} \overline{P}^{\star}_{G,\sigma_r,\sigma_l} \enspace . 
\end{equation}
Then, we normalize the polynomials
\begin{equation}~\label{eqn:variance of graph2}
\Psi_G^{\star} = \frac{\overline{P}^{\star}_G}{\sqrt{\mathbb V^{\star}(G)}},~~~~~\mathrm{where}~~~\mathbb V^{\star}(G) = \frac{(n-1)!}{(n-|V|)!} \frac{d!}{(d-|W|)!} |\mathrm{Aut}^{\star}(G)|\enspace ,
\end{equation}
where $\mathrm{Aut}^{\star}(G)$ is the set of automorphisms of the graph $G$ restricted to the ones where the node $v_1$ is fixed. 
We say that  two graphs $G^{(1)}=(V^{(1)},W^{(1)},E^{(1)})$ and $G^{(2)}=(V^{(2)},W^{(2)},E^{(2)})$ are equivalent if there exists two bijections $\tau_v: V^{(1)} \mapsto V^{(2)}$, $\tau_w: W^{(1)} \mapsto W^{(2)}$ such that that $\tau_v, \tau_w$ preserve the edges and $\tau_{v}(v_1^{(1)})= v_1^{(2)}$. 
Define $\mathcal G^{\star}_{\leq D}$ as a maximum collection of non-equivalent bipartite graphs $G=(V,W,E)$ with at most $D$ edges, at least one edge, and such that all nodes to the possible exception of $v_1$ are  not isolated.

The next lemma states that $(1,(\Psi_G)_{G\in \cG^{\star}_{\leq D}})$ is a basis of  the subspace of $\cP_{\leq D}$ with the suitable permutation invariance properties. 
\begin{lemma}\label{lem:reduction::permutation2}
Fix any any degree $D\in \mathbb N$. Then, 
$$ \mathrm{Corr}^{\star}_{\leq D} = \sup_{(\alpha_{\emptyset}, (\alpha_G)_{G\in \mathcal G^{\star}_{\leq D}})} \mathrm{Corr}^{\star}\left(\alpha_{\emptyset}+\sum_{G\in \mathcal G^{\star}_{\leq D}} \alpha_G \Psi^{\star}_G\right)\enspace .$$
\end{lemma}

Finally, the next theorem states  that, under some assumptions on $K_n,K_d,\lambda$, the invariant polynomial family $(\Psi^\star_G)_{G \in \mathcal G^\star_{\leq D}}$ is almost orthonormal.
\begin{theorem}\label{thm:isorefo2}
Assume that $D \geq 2$ and that, for some $c_{\texttt{s}}>0$ large enough universal constant, we have 
\begin{align}\label{eq:signal1}
	\left(\frac{K_n}{n} \right)\lor \left(\frac{K_d}{d} \right) \lor \left(\frac{\lambda K_n}{\sqrt{n}} \right)\lor \left(\frac{\lambda K_d}{\sqrt{d}} \right) \lor \lambda  \leq D^{-8c_{\texttt{s}}}\enspace . 
\end{align}
Then, for any $(\alpha_{\emptyset}, (\alpha_G)_{G \in \mathcal G^\star_{\leq D}})$, we have 
    $$(1-D^{-c_{\texttt{s}}/4})\left[\alpha^2_{\emptyset}+ \sum_{G \in \mathcal G^{\star}_{\leq D}}\alpha^2_G\right] \leq \mathbb E \left[\left(\alpha_{\emptyset}+ \sum_{G \in \mathcal G^{\star}_{\leq D}} \alpha_G \Psi^{\star}_G\right)^2\right] \leq (1+D^{-c_{\texttt{s}}/4})\left[\alpha^2_{\emptyset}+ \sum_{G \in \mathcal G^{\star}_{\leq D}}\alpha^2_G\right]\enspace .$$
\end{theorem}
As explained below, Condition~\eqref{eq:signal1} is precisely, up to minor changes, the condition under which reliable estimation of $x^\star$ is not achievabel by low-degree polynomials.

\subsubsection{LD  polynomial lower bound.}

Together with Lemma~\ref{lem:reduction::permutation2}, Theorem~\ref{thm:isorefo2} implies that 
\begin{align}\label{eq:badv2}
\mathrm{Corr}^{\star 2}_{\leq D} 
 &\leq \frac{1}{1-D^{-c_{\texttt{s}}/4}}\left[\E^2[x^{\star}]+ \sum_{G\in \mathcal G^{\star}_{\leq D}}\E^2[x^{\star}\Psi^{\star}_G]\right]\enspace .
\end{align}
As a consequence, we only have to control the first moments $\E[x^{\star}\Psi^{\star}_G]$ to control the low-degree correlation $\mathrm{Corr}^{\star}_{\leq D}$. This leads us to the following.

\begin{theorem}[Computational lower bound for estimating $x^{\star}$] \label{prop:compest}
	Fix $D\geq 2$. Under condition~\eqref{eq:signal1}, we then have 
\[
\inf_{f\in \mathcal P_{\leq D}}\mathbb{E}[(f-x^{\star})^2] \geq \frac{\mathbb E[x^{\star 2}]}{2}\enspace .
\]
\end{theorem}
Following the low-degree conjecture, we focus on the case $D\asymp \log(nds)$ to provide strong evidence that rule out polynomial-time algorithms. In~\eqref{eq:signal1}, the condition $K_n/n\vee K_d/d\vee \lambda\leq D^{-8c_{\mathrm{s}}}$ is mild as, by definition, we have $\lambda \leq 1$, $K_n\leq n$, $K_d\leq d$.
Also, combining the above theorem with the information lower bound of Proposition~\ref{prop:LBinfestM}, we deduce that the optimal low-degree risk 
$\inf_{f\in \mathcal P_{\leq D}}\mathbb{E}[(f-x^{\star})^2]$ is no smaller than $\mathbb E[x^{\star 2}]/2$ as soon as
$$\left[\left(\lambda \frac{K_n}{\sqrt{n}}\right) \land (\lambda\sqrt{K_d})\right]\lor \left[\lambda \frac{K_d}{\sqrt{d}} \right]\leq D^{-c}\enspace ,$$
for a positive constant $c$. Up to polylog terms, this matches the sufficient condition of the estimator  $f_{SE} \lor f_{SE,2}$ --see Proposition~\ref{prop:ubestM}.

Comparing this LD polynomials bounds with the minimax risk, we have provided evidence for a computational-statistical gap for the problem of estimating $x^{\star}$.  Coming back to the ranking problem, we deduce from the computational lower bound of Theorem~\ref{prop:compest}, that it is impossible to build a polynomial-time ranking procedure  that is minimax adaptive over all $\cH(\lambda,K_n,K_d)$. This was already observed e.g. in~\cite{shah2019feeling}. Nevertheless, as already noted, Corollaries~\ref{cor:UBinfM} and~\ref{cor:LBinfM} ensure that the minimax risk for estimating $\pi^*$ over 
$\bigcup_{K_n,K_d,\lambda}\cH(\lambda,K_n,K_d)$ is achieved by a linear-time procedure so that  there is no statistical-computational gap for the worst-case risk over all submatrices models.

\subsection{Reconstruction of the matrix $M$.} The problem of estimating $x^* = \mathbf 1\{\exists k: M_{1,k} >0\}$ is related to that of reconstructing $M$ but is possibly distinct for rectangular settings. Nevertheless, we can easily deduce computational and statistical results for the estimation of $M$. Indeed, estimating the support of $M$ corresponds to inferring 
functionals of the form
$\mathbf 1\{M_{1,1} >0\} = \mathbf 1\{\exists k: M_{1,k} >0\} \mathbf 1\{\exists i: M_{i,1} >0\}$, so that it can be reduced to estimating both $x^* = \mathbf 1\{\exists k: M_{1,k} >0\}$ and its counterpart $\mathbf 1\{\exists i: M_{i,1} >0\}$. As a straightforward consequence, we therefore get, as long as  $\left[\lambda \sqrt{ K_n\land K_d} \right] \lor \left[\lambda \cdot \left[\frac{K_d}{\sqrt{d}} \land \frac{K_d}{\sqrt{d}}\right] \right]$ is large, it is possible to reconstruct $M$ with a small error. Conversely, as long as this quantity small, this is statistically impossible. For polynomial-time algorithms, reconstruction with a small error is possible as soon as $\lambda \cdot \left[\frac{K_d}{\sqrt{d}} \land \frac{K_d}{\sqrt{d}}\right]$
is large, whereas, below this threshold, LD lower bounds provide evidence of polynomial-time algorithms to succeed. Similar results have been for instance established in~\cite{butucea2015sharp,cai2017computational,SchrammWein22}.

\subsection{Comparison between detection and estimation.} The fundamental statistical and computational thresholds for detection, that we presented in Section~\ref{sec:detection}, are unsurprisingly\footnote{Since given an estimator of $\pi^*$, one can compute easily a test for deciding whether $M=M_0$ or not.} lower than their counterpart for the estimation of $\pi^*$ over $ \cH(K_n,K_d,\lambda)$. Namely, the constraints on $K_n,K_d,\lambda$ that we need to ensure detection from resp.~an informational or computational perspective are less constraining than the constraints we need for estimation - and in quite a few regimes of $K_n,K_d,\lambda$ they are strictly less constraining by quite a large amount. 

A straightforward consequence is that this is also the case from a worst-case perspective over $K_n,K_d,\lambda$, namely over $\bigcup_{K_n,K_d,\lambda} \cH(K_n,K_d,\lambda)$. An important remark is that while there is a statistical-computational gap for the problem of detection from Section~\ref{sec:detection} - over $\bigcup_{K_n,K_d,\lambda} \cH(K_n,K_d,\lambda)$ and also over $\overline{\mathbb C}_{\mathrm{iso}}$ - there is no statistical-computational gap for the problem of estimating $\pi^*$ over $\bigcup_{K_n,K_d,\lambda} \cH(K_n,K_d,\lambda)$. Note however that $\bigcup_{K_n,K_d,\lambda} \cH(K_n,K_d,\lambda)$ is much smaller than $\overline{\mathbb C}_{\mathrm{iso}}$, and a remaining topic - that we will tackle in the next section - is then on the thresholds over $\overline{\mathbb C}_{\mathrm{iso}}$. 

\section{Ranking the experts: general case}\label{sec:ranking}

In this section, we come back to the ranking problem over all permuted isotonic matrices $\overline{\mathbb{C}}_{\mathrm{iso}}$ - and we go beyond the sub-models $\cH$.

\subsection{Short review of existing results.}\label{ss:contributions}
  
  In the specific case where $d=1$ (a single column), our model is equivalent to uncoupled isotonic regression and is motivated by optimal transport. Rigollet and Niles-Weed~\cite{rigollet2019uncoupled} have established that the reconstruction error of $M$ is of the order of $n (\tfrac{\log\log(n)}{\log(n)})^2$. 
  
  For the general case $d\geq 1$, Flammarion et al.~\cite{flammarion2019optimal} have shown\footnote[1]{The authors consider the isotonic model as a subcase of a seriation model, where each columns of $M_{\pi^{*-1}}$ is only assumed to be unimodal.} that the optimal reconstruction error in the sense of~\eqref{eq:reconstruction:loss}  is of the order of $n^{1/3}d+n$. However, the corresponding procedure is not efficient. They also introduce a polynomial-time procedure that first estimates $\pi^*$ using a score based on row comparisons on $Y$. Unfortunately, this method only achieves a reconstruction error of the order of $n^{1/3}d + n\sqrt{d}$ which is significantly slower than the optimal one.
  Ma et al.~\cite{ma2021optimal} also focus on the estimation of $\pi^*$ in the isotonic model, but the authors take a different approach. Indeed, they aim to recover $\pi^*$ perfectly under separation assumptions between the rows of $M$, which makes the problem statistically easier. Specifically, under separation assumptions, the authors show that a simple spectral method achieves optimal guarantees for the estimation of $\pi^*$. Informally, their estimator $\hat \pi$ corresponds to the order of the coefficients of the largest left eigenvector of $Y - \overline Y$, where $\overline Y_{i,k} = \tfrac{1}{n}\sum_{i'=1}^n Y_{i',k}$.

Finally, Pilliat et al.~\cite{pilliat2024optimal} prove that, for the problem of estimating $\pi^*$ (in the sense of~\eqref{eq:estpermM})  as well as for the problem of reconstructing $M$, there is no statistical-computational gap. They provide polynomial-time estimators for these problems, based on an intricate algorithm, that we will briefly describe below. We also remind below the (informational) lower bound on the minimax rate of estimation of $\pi^*$ (resp.~ $M$), which can be established by constructing a set of difficult instances that are a precise combination of matrices in $\cH(K_n,K_d,\lambda)$. So that, in what follows, we mostly sum up results from~\cite{pilliat2024optimal}.

\subsection{Estimation of the permutation.} We first consider the problem of estimating the permutation $\pi^*$. The following theorem is a consequence of~\cite{pilliat2024optimal} for the problem of estimating $\pi^*$.
\begin{theorem}[Permutation estimation]\label{th:UB}
For any $n\geq 2$, and $d\geq 1$, define $\rho_{\perm}(n,d)= n^{2/3}\sqrt{d} + n$. We have 
$$ \inf_{\hat \pi~\mathrm{measurable}}  \sup_{M \in \overline{\mathbb{C}}_{\iso}} \mathbb E_{M} [\|M_{\hat \pi^{-1}} - M_{\pi^{*-1}}\|_F^2]\asymp_{\log} \rho_{\perm}(n,d) \enspace,$$
where $u \asymp_{\log} v$ means that, for some positive constants  $c_{0}\log^{-c'_0}(nd) \leq u/v \leq c_1\log^{c'_1}(nd)$. Besides, there exists a polynomial-time estimator that, up to polylog terms, uniformly achieves this risk bound.
\end{theorem}

\medskip 

\noindent
{\bf Lower bound.} The lower bound in this theorem does not follow directly from results over any set $\cH(K_n,K_d,\lambda)$ - as the minimax rate over any $\cH(K_n,K_d,\lambda)$ is much smaller than $n+n^{2/3}\sqrt{d}$ in many cases - but it can nevertheless be deduced from previous results by decomposing the matrices $M$ in a well-chosen way.\\
First, note that the lower bound of order $n$ is a direct consequence of Corollary~\ref{cor:LBinfM}. We therefore focus on the second part, namely the lower bound of order $n^{2/3} \sqrt{d}$, which is only relevant when $\sqrt{d} \geq n^{1/3}$. This can be deduced from Proposition~\ref{prop:LBinfestM}. Indeed, we can divide our $n$ experts in $m$ groups with $N = n/m$ experts - assuming that $n/m\in \mathbb N$. So that $M$ can be written through stacking $m$ matrices of dimension $N\times d$ - and we write $M^{(l)}$ for the $l$-th matrix where $l\in [m]$. In what follows, we consider the problem where each matrix $M^{(l)}$ can be written as $M^{(l)} = \lambda \times (l-1)+ \bar M^{(l)}$ where $\bar M^{(l)} \in \cH(K_N,K_d,\lambda)$. This is within our shape constraint as long as $\lambda m \leq 1$ - and we equalize the two terms and take $\lambda = 1/m$.\\
We focus on the case where $K_N \asymp (1/\lambda^2)\land (N/2) = m^2 \land [n/(2m)]$ and $K_d \asymp (\sqrt{d}/\lambda) \land d$. This satisfies the condition of Proposition~\ref{prop:LBinfestM}, so that we can prove that the probability of making a mistake when ordering any two experts is of constant order - which leads to a squared Frobenius error when ordering the two experts of order $K_d \lambda^2$. Also, setting $m^2 = n/m$ (equalizing the two terms in the definition of $K_N$), leads to taking $m = n^{1/3}$ and leads to $K_d \asymp \sqrt{d}/\lambda= \sqrt{d}m^{1/3} \leq d$ as $\sqrt{d} \geq n^{1/3}$.\\
The permutation loss over the full matrix $M^{(l)}$ is then of order $K_dK_N \lambda^2$, which in turn leads to a full minimax permutation loss $\mathbb E_M \|M_{\hat \pi^{-1}} - M_{\pi^{*-1}}\|_F^2$ of order $mK_dK_N \lambda^2$ by summing over all $l$'s. This is of order $n^{2/3}\sqrt{d}$, as expected. In Figure~\ref{fig:sample_figure3}, we illustrate the typical matrices $M$ that are used in the lower bound construction.

\medskip

\noindent
{\bf Upper bound and associated algorithm.}
The algorithm in~\cite{pilliat2024optimal} is an iterative soft ranking $(\algoPrincipal)$ procedure, which gives an estimator $\hat \pi$ based on the observations.  Informally, this method iteratively updates a weighted directed graph between experts, where the weight between any two experts quantifies the significance of their comparison. The procedure increases the weights at each step. After it stops, the final estimator is an arbitrary permutation $\hat \pi$ that agrees as well as possible  with the final weighted directed graph.

As mentioned by Liu and Moitra \cite{liu2020better}, it is hopeless to use only local information between pairs of experts to obtain a rate of order $n^{7/6}$ up to polylogs, and we must exploit global information. Still, we do it in a completely different way of~\cite{liu2020better} who were building upon the bi-isotonicity of the matrix --here, the matrix $M$ is only, up to a permutation $\pi^*$, isotonic. 
One first main ingredient of our procedure is a new dimension reduction technique. At a high level, suppose that we have partially ranked the rows in such a way that, for a given triplet ($P$, $O$, $I$) of subsets of $[n]$, we are already quite confident that experts in $P$ are below those in $I$ and above those in $O$. Relying on the shape constraint of the matrix $M$, it is therefore possible to build high-probability confidence regions for rows in $P$ based on the rows in $O$ and the rows in $I$ - computing column-averages for each question over resp.~$O$ and $I$, which we term the two "envelopes". If, for a question $k$, the confidence region is really narrow - i.e.~the two envelopes are close from each other - this implies that all experts in $P$ take almost the same value on this column. As a consequence, this question is almost irrelevant for further comparing the experts in $P$.
In summary, our dimension reduction technique selects the set of questions for which the confidence region of $P$ is wide enough - i.e.~the two envelopes are far away from each other - and in that way reduces the dimension of the problem while keeping most of the relevant information. This is illustrated in the right panel of Figure~\ref{fig:sample_figure3}. 

The second main ingredient, once the dimension of $P$ is reduced, is to use a spectral method to capture some global information shared between experts and then to compute the updates of the weighted graph. Although this spectral scheme already appeared  in recent works~\cite{pilliat2022optimal,liu2020better},  those are used here for updating the weight of the comparison graph rather than performing a clustering as in~\cite{liu2020better}. The reason why performing a spectral method is a reasonable idea is tightly related to the fact that any isotonic matrix can be, as we explained in Section~\ref{sec:peeling} and in Lemma~\ref{lem:peeling}, in some sense approximated by a rank one matrix - see also Figures~\ref{fig:sample_figure} and~\ref{fig:sample_figure2} for illustrations.

\begin{figure}
\centering
\begin{minipage}[c]{0.45\textwidth}
\centering
    \includegraphics[width=3.0in]{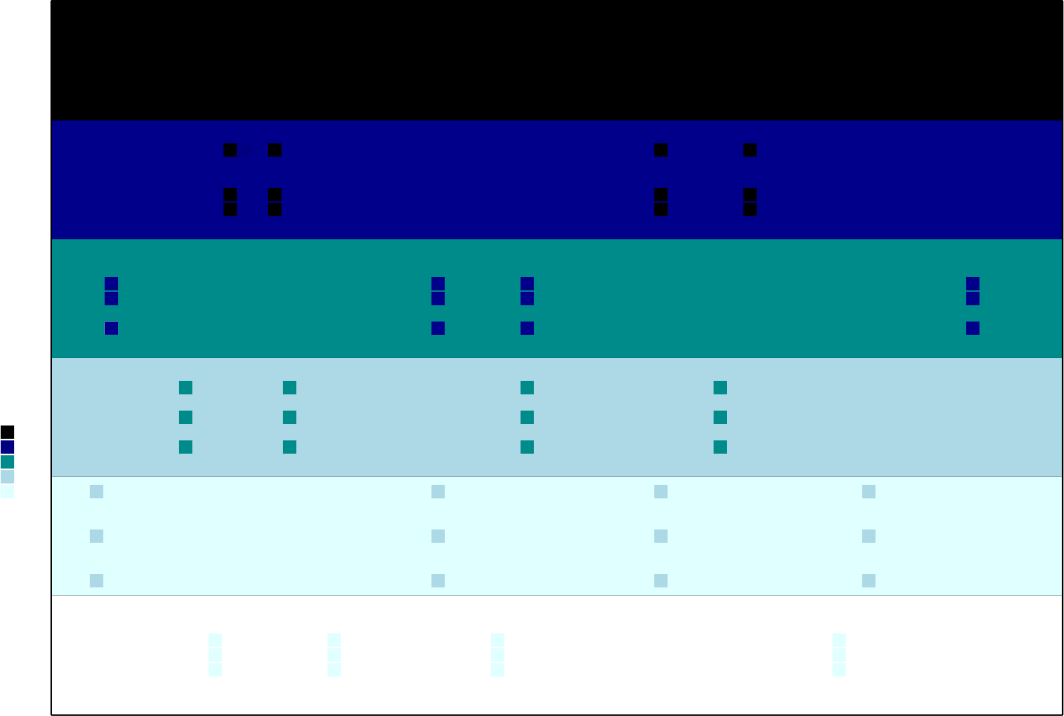}
\end{minipage}
\begin{minipage}[c]{0.45\textwidth}
\centering
    \includegraphics[width=3.0in]{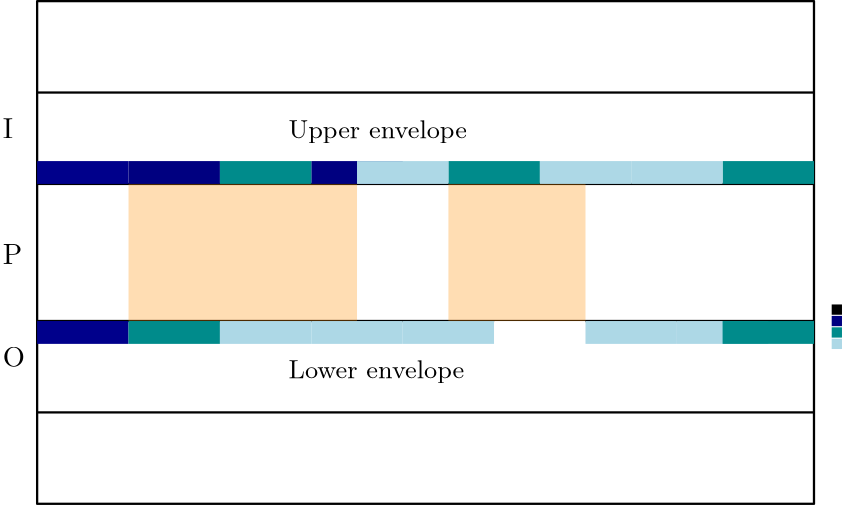}
\end{minipage}
    \caption{Both pictures represent $M$ (after oracle permutation of the rows), the rectangle's axis being resp.~the $j \in [d]$ index (abscissa) and the $i\in [n]$ index (ordinate). The darker the colors, the higher the entries of $M$. On the \emph{left panel}, we represent a typical instance of a matrix $M$ used in the lower bound construction. We have as an illustration $m=6$ groups of experts and in each of the groups, we have $N=8$ experts, among which $K_N=3$ of them are better on $K_d=4$ questions than the others (see the small sub-matrix in each groups which represent the $K_n$ experts and $K_d$ questions).  On the \emph{right panel}, we illustrate the idea of dimension reduction for the proposed polynomial-time procedure. We represent the three sets $O,P,I$, and the upper and lower envelope for $P$. Only questions whose upper and lower envelopes are significantly different are selected, corresponding to the entries in orange in $P$.} \label{fig:sample_figure3}
\end{figure}

\subsection{Reconstruction of the matrix.} We now turn to the problem of reconstructing the signal matrix $M$. Obviously, this problem is at least as difficult as if we knew in advance the permutation $\pi^*$. In this favorable situation, estimating $M$ simply amounts to estimating $d$ isotonic vectors.

\begin{theorem}[Matrix estimation]
For any $(n,d)$, define  $\rho_{\reco}(n,d)=n^{1/3}d + n$. We have 
$$ \inf_{\hat M~\mathrm{measurable}}  \sup_{M\in \overline{\mathbb C}_{\iso}} \mathbb E_{M} [\|\hat M - M\|_F^2]\asymp_{\log } \rho_{\reco}(n,d) \enspace .$$
 Besides, there exists a poly-time estimator uniformly achieving this rate up to a poly-logarithmic factors in $n,d$.
\end{theorem}
In particular, since $\rho_{\perm}(n,d) \leq \rho_{\reco}(n,d)$ in all regimes in $n$, $d$, this proposition implies that the reconstruction of a permuted isotonic matrix is harder than the estimation of the permutation.

\noindent
{\bf Lower bound.} A first part of the lower bound - namely the part of order $n$ - directly follows from our results over worst-case sets $\cH$ for the matrix reconstruction problem-- the arguments are similar to those in Corollary~\ref{cor:LBinfM}. The second part of the lower bound - namely the part of the order $dn^{1/3}$ - follows from classical lower bound arguments in the isotonic regression model (not permuted, i.e.~where $\pi^*$ is known) - see e.g.~\cite{mao2020towards} where they provide a lower bound on the problem of estimating $d$ isotonic vectors (see in particular Theorem 3.1 therein).

\medskip

\noindent
{\bf Upper bound and associated algorithm.} To build an optimal estimator of $M$, we start from a polynomial estimator $\hat{\pi}$ achieving the ranking risk of  Theorem \ref{th:UB} and we then estimate an isotonic matrix based on $\hat{\pi}$. This approach is similar to what is done in~\cite{mao2020towards,pilliat2022optimal} for related bi-isotonic models. Based on $\hat \pi$, we define $\hat M_{\iso}$ as the projection of $Y_{\hat \pi^{-1}}$ onto the convex set of isotonic matrices. More precisely, we set
\begin{equation*}
	\hat M_{\iso} = \argmin_{\tilde M \in \bbC_{\iso}} \| \tilde M - Y_{\hat \pi^{-1}}\|_F^2 \enspace .
\end{equation*}
Based on $\hat \pi$, this minimization problem can be done efficiently. Based on our results on $\hat \pi$ from Theorem~\ref{th:UB}, it is not difficult to prove that the classical bounds on isotonic regression carry over to regression $ Y_{\hat \pi^{-1}}$, which is approximately isotonic. So that the reconstruction loss of $[\hat M_{\iso}]_{\hat{\pi}}$ is of the same order as the loss incurred when performing isotonic regression, namely $n^{1/3}d$ plus the permutation loss, namely $\rho_{\perm}(n,d)$, which brings us to our result --see~\cite{pilliat2024optimal} for details.

  \subsection{Implication for other models and connection to the literature.}\label{subsec:literature}

The isotonic model is quite general and encompasses both the bi-isotonic model for crowdsourcing problems as well the SST model for tournament problems. Also, in a general setup, the model is typically defined with the possibility of having missing data, and with a more general noise structure. In that case, we observe in total $N$ observations where $N\sim \mathcal P(\zeta)$ for some sampling effort $\zeta >0$ and each observation is a noisy observation of the matrix $M$ where the position is taken uniformly at random. In this case, each observation $t\in [N]$ is defined as follows:
\begin{equation}\label{eq:model_0}
Y_{t} = M_{i_t,j_t} + \varepsilon_{t} \,
\end{equation}
where $\varepsilon_{t}$ is an independent sub-Gaussian noise, and where $i_t,j_t$ are sampled uniformly at random from $[n]\times [d]$.

Let us first focus on the SST model which corresponds to the case where $n=d$ together with a bi-isotonicity and a skew-symmetry assumption. In the full observation scheme (related to the case $\zeta=1$) where one observes the noisy matrix $n\times n$, Shah et al.~\cite{shah2016stochastically} have established that the optimal rates for estimating $\pi^*$ and reconstructing the matrix $M$ are of the order of $n$. In contrast, their efficient procedure which estimates $\pi^*$ according to the row sums of $Y$ only achieves the rate of $n^{3/2}$. In more recent years, there has been a lot of effort dedicated to improving this $\sqrt{n}$ statistical-computational gap. The SST model was also generalized to partial observations by  \cite{chatterjee2019estimation}, which corresponds to  $\zeta \leq 1$. They introduced an efficient procedure that targets a specific sub-class of the SST model, and that achieves a rate of order $n^{3/2}\zeta^{-1/2}$ in the worst case for matrix reconstruction. 

A few important contributions tackling both the bi-isotonic model and the SST model made steps towards better  understanding the statistical-computational gap. We first explain how their results translate in the SST model. Mao et al.~\cite{mao2020towards, mao2018breaking} introduced a polynomial-time procedure handling partial observations and achieving a rate of order $n^{5/4}\zeta^{-3/4}$ for matrix reconstruction. Nonetheless, \cite{mao2020towards} failed to exploit global information shared between the players/experts -- as they only compare players/experts two by two -- as pointed out by \cite{liu2020better}. Building upon this remark, \cite{liu2020better} managed to get the better rate $n^{7/6+o(1)}$ with a polynomial-time method in the case $\zeta = n^{o(1)}$.  

Let us turn to the more general bi-isotonic model. Here, the rectangular matrix $M \in \bbR^{n \times d}$ is bi-isotonic up an unknown permutation $\pi^*$ of the rows and an unknown permutation $\eta^*$ of the columns. Since $M$ is not necessarily square, this model can be used in more general crowd-sourcing problems. The optimal rate for reconstruction in this model with partial observation has been established in \cite{mao2020towards} to be of order $\nu(n,d,\zeta):= (n\lor d)/\zeta + \sqrt{nd/\zeta}\land n^{1/3}d\zeta^{-2/3}\land d^{1/3}n\zeta^{-2/3}$ up to polylog factors, in the non-trivial regime where $\zeta \in [1/(n\land d), 1]$. However, the polynomial-time estimator provided by Mao et al. \cite{mao2020towards} only achieves the rate $n^{5/4}\zeta^{-3/4} + \nu(\zeta, n, d)$. In a nutshell, their procedure computes column sums to give a first estimator of the permutation of the questions. Then, they compare the experts on aggregated blocks of questions, and finally compare the questions on aggregated blocks of experts. As explained in the previous paragraph for SST models, Liu and Moitra~\cite{liu2020better}  improved this rate to $n^{7/6 + o(1)}$ in the square case $(n=d)$, with a subpolynomial number of observations per entry $(\zeta = n^{o(1)})$. 	
Their estimators of the permutations $\pi^*$, $\eta^*$ were based on hierarchical clustering and on local aggregation of high variation areas. Both \cite{liu2020better,mao2020towards} made heavily use of the bi-isotonicity structure of $M$ by alternatively sorting the columns and rows. As mentioned for the SST model, the order of magnitude $n^{7/6+o(1)}$ remains nevertheless suboptimal, and whether there exists an efficient algorithm achieving the optimal rate in this bi-isotonic model remains an important open problem. 

As a matter of fact, the results in~\cite{pilliat2024optimal} also carry out to the bi-isotonic model and the SST model. In the square case for the bi-isotonic model $(n=d)$, ~\cite{pilliat2024optimal} reach in polynomial-time the upper bound $n^{7/6}\zeta^{-5/6}$ up to polylog factors, for both permutation estimation and matrix reconstruction.

The optimal loss and the known polynomial-time upper bounds for the isotonic, bi-isotonic with two permutations and SST models are summarized in \Cref{tab:comparison}. For the sake of simplicity, we focus in this table on the specific case $n=d$ and $\zeta \in [1/n, 1]$.

\begin{table}[h]
\begin{center}
	\begin{tabular}{|P{2cm} | P{1cm} | P{3.3cm} | P{3.7cm} | P{3.7cm} |}
		\noalign{\global\arrayrulewidth=0.3mm}
		\arrayrulecolor{gray}
		  \hline
		  \multicolumn{2}{|c|}{\multirow{2}{*}{\vspace{-1cm}$\begin{array}{cc} 
			\text{Different models, with}\\
			\text{$M \in \bbR^{n\times n}$}\\
		\end{array}$}} & {\bf Isotonic} & Bi-isotonic($\pi^*, \eta^*)$ & SST \\ 
		  \multicolumn{2}{|c|}{}&{\bf$M_{\pi^{*-1}}$ has non-increasing columns } &$M_{\pi^{*-1} \eta^{*-1}}$ has non-increasing columns and rows& $M_{\pi^{*-1} \pi^{*-1}}$ has non-increasing columns and rows, and $M_{i,k} + M_{k,i} = 1$\\
		  \hline
		  Permutation estimation & Poly. Time & {\bf $n^{7/6}\zeta^{-5/6}$ \text{\cite{pilliat2024optimal}}} &
		  $\begin{array}{cc} 
			n^{7/6+o(1)}  & \hspace{-0.3cm}\text{\cite{liu2020better}}(\zeta=n^{o(1)})\\
			n^{7/6}\zeta^{-5/6} & \text{\cite{pilliat2024optimal}}\\
		\end{array}$
		& $\begin{array}{cc} 
			n^{7/6+o(1)}  & \hspace{-0.3cm} \text{\cite{liu2020better}}(\zeta=n^{o(1)})\\
			n^{7/6}\zeta^{-5/6} & \text{\cite{pilliat2024optimal}}\\
		\end{array}$\\
		\cline{2-5}
		~ & optimal rate& {\bf $n^{7/6}\zeta^{-5/6}$ \text{\cite{pilliat2024optimal}}} &
		  $n/\zeta$  \cite{mao2020towards} & $n/\zeta$ \cite{mao2020towards}\\
		  \hline
		  Matrix \newline reconstruction &Poly. Time & 
		  $\begin{array}{cc} 
			n^{3/2} & (\zeta=1)\text{\cite{flammarion2019optimal}}\\
			n^{4/3}\zeta^{-2/3} & \text{\cite{pilliat2024optimal}}\\
		\end{array}$
		& $\begin{array}{cc} 
			n^{7/6+o(1)}  & \hspace{-0.3cm} \text{\cite{liu2020better}}(\zeta=n^{o(1)})\\
			n^{5/4}\zeta^{-3/4} & \text{\cite{mao2020towards}}\\
			n^{7/6}\zeta^{-5/6} & \text{\cite{pilliat2024optimal}}\\
		\end{array}$ & $\begin{array}{cc} 
			n^{7/6+o(1)}  & \hspace{-0.3cm} \text{\cite{liu2020better}}(\zeta=n^{o(1)})\\
			n^{5/4}\zeta^{-3/4} & \text{\cite{mao2020towards}}\\
			n^{7/6}\zeta^{-5/6} & \text{\cite{pilliat2024optimal}}\\
		\end{array}$\\
		\cline{2-5}
		  ~ & optimal rate & $n^{4/3}\zeta^{-2/3}$ \cite{flammarion2019optimal}\newline (also \text{\cite{pilliat2024optimal}}) &
		  $n/\zeta$  \cite{mao2020towards} & $n/\zeta$ \cite{mao2020towards}\\
		  \hline
		\end{tabular} \enspace 
\end{center}
\caption{For the isotonic model, we display here the optimal rate for permutation estimation (resp. matrix reconstruction)  - see Equation~\eqref{eq:estpermM}. For the two other columns, the optimal rates are similarly defined as minimax loss over the corresponding models. The Poly. Time rows correspond to state-of-the art rates achieved by polynomial-time methods. All the rates are given up to polylogarithmic factors in $n$.}
\label{tab:comparison}
\end{table}
Finally, we mention the specific model where the matrix $M$ is bi-isotonic up to a single permutation $\pi^*$ acting on the rows. This corresponds to the case where $\eta^*$ is known in the previous paragraph~\cite{mao2020towards,pilliat2022optimal,liu2020better}, or equivalently, to our isotonic model with the constraint that all the rows are non-increasing, that is $M_{i,k} \geq M_{i,k+1}$. It is possible to leverage the shape constraints on the rows, e.g. by using change-point detection techniques, to build efficient and optimal estimators--see \cite{pilliat2022optimal,liu2020better} at a faster rate than for bi-isotonic models with unknown permutations. 

\subsection*{Acknowledgements}
Most of this paper is based on results co-written with Emmanuel Pilliat, Christophe Giraud and Simone Maria Giancola - we are grateful for their insights and for many helpful brainstorming sessions back then on these results. We are in particular grateful to Christophe Giraud for many illuminating discussions.\\ 
The work of A. Carpentier is partially supported by the Deutsche Forschungsgemeinschaft (DFG)- Project-ID 318763901 - SFB1294 "Data Assimilation", Project A03,  by the DFG on the Forschungsgruppe FOR5381 "Mathematical Statistics in the Information Age - Statistical Efficiency and Computational Tractability", Project TP 02 (Project-ID 460867398), and by  the DFG on the French-German PRCI ANR-DFG ASCAI CA1488/4-1 "Aktive und Batch-Segmentierung, Clustering und Seriation: Grundlagen der KI" (Project-ID 490860858). The work of N. Verzelen has  partially been supported by ANR-21-CE23-0035 (ASCAI, ANR). The work of A.~Carpentier and N.~Verzelen is also supported by the Universite franco-allemande (UFA) through the college doctoral franco-allemand CDFA-02-25 "Statistisches Lernen für komplexe stochastische Prozesse".
\bibliographystyle{abbrv}
\bibliography{biblio}

\newpage

\appendix

\section{Proofs of the main results}\label{s:proofsimp}

\subsection{Proof of Lemma~\ref{lem:peeling}}

~\\

 \noindent
 {\bf Step 1: Choosing a relevant level set of $M$.} We write for any $u \in [0,1]$
$$M(u) = (\mathbf 1\{M_{i,j} \geq 2^{-u}\})_{i,j}.$$
Define  also the $\mathrm{Grid} = \{1, 2,\ldots, p\}$. Note that --see Figure~\ref{fig:sample_figure} for an illustration--
$$M \leq \sum_{u \in \mathrm{Grid}} M(u) 2^{-u+1} + 2^{-p},$$
and that, for any $u\in \mathrm{Grid}$, we have $M(u) 2^{-u}\leq M$. So that, taking $u^* \in \argmax_{u \in \mathrm{Grid}} \|M(u)\|_2^2 $, we get
\begin{equation}\label{eq:boundM1}
\|M\|_F^2 \leq \sum_{u \in \mathrm{Grid}} \|M(u)\|_2^2 2^{-2u+2} + nd2^{-2p} \leq p   \|M(u^*)\|_2^2 2^{-2u^*+2} + nd2^{-2p} ,
\end{equation}
and 
\begin{equation}\label{eq:boundM2}
M(u^*) 2^{-u^*}\leq M.
\end{equation}
See again Figure~\ref{fig:sample_figure} for an illustration.

\noindent
 {\bf Step 2: Approximating a two-valued isotonic matrix through a well-chosen element of $\cH$.} We now focus on $M(u^*)$. Note that $M(u^*)$ takes values in $\{0,1\}$ and is isotonic according to its first index. So there exists a permutation $\pi_{\mathrm{col}}$ of its second index that makes it bi-isotonic with non-increasing rows and columns. Define $\tilde M = M(u^*)_{.,  \pi_{\mathrm{col}}(.)}$ for the resulting bi-isotonic, two-valued matrix. See also Figure~\ref{fig:sample_figure} for an illustration. 

In what follows, we assume w.l.o.g.~that $d \geq n$. Write $\mathrm{Grid}' = \{1, 2, 4, \ldots, 2^{\lfloor \log_2(n)\rfloor}\}$. For any $i \in \mathrm{Grid}'$, we write $j_i$ for the largest $j \in [d]$ such that $\tilde{M}_{i,j}=1$. Define the matrices 
Write 
$$\tilde M(i) = (\mathbf 1\{i'\leq i, j'\leq j_i\})_{i',j'},~~~\mathrm{and}~~~\bar M(i) = (\mathbf 1\{i'\leq 2i, j'\leq j_i\})_{i',j'}\enspace .$$
See also Figure~\ref{fig:sample_figure2} for an illustration. Note that
\begin{equation}\label{eq:boundM3}
2\|\tilde M(i)\|_2^2 = \|\bar M(i)\|_2^2.
\end{equation}

Note that  -see Figure~\ref{fig:sample_figure2} for an illustration-, we have  $\tilde M \leq \sum_{i \in \mathrm{Grid}'} \bar M(i)$, and 
that,  for any $i\in \mathrm{Grid}'$, we have $\tilde M(i) \leq \tilde M$. So that, taking $i^* \in \argmax_{i \in \mathrm{Grid}'} \|\bar M(i)\|_2^2 $, we deduce from~\eqref{eq:boundM3} that 
$$\|\tilde M\|_2^2 \leq \sum_{i \in \mathrm{Grid}'} \|\bar M(i)\|_2^2 \leq \lfloor \log_2(n)\rfloor  \|\bar M(i^*)\|_2^2 \leq 2\lfloor \log_2(n)\rfloor  \|\tilde M(i^*)\|_2^2,$$
and  $\tilde M(i^*) \leq \tilde M$. 

 {\bf Step 3: Conclusion.} Putting the two last displayed equations together with Equations~\eqref{eq:boundM1} and~\eqref{eq:boundM2} and reminding that  $\tilde M = M(u^*)_{.,  \pi_{\mathrm{col}}(.)}$, we have
$$\|M\|_F^2 \leq  2p  \lfloor \log_2(n)\rfloor  \|\tilde M(i^*)\|_2^2 2^{-2u^*+2} + nd2^{-2p} ,$$
and 
$$\tilde M(i^*)_{., \pi_{\mathrm{col}}^{-1}(.)} 2^{-u^*}\leq M.$$
This concludes the proof with $\lambda = 2^{-u^*}$, $K_n = i^*$ and $K_d = j_{i^*}$.

\subsection{Proof of Lemma~\ref{lem:reduction::permutation} and of Proposition~\ref{lem:lower_bound_loss} and of Proposition~\ref{lem:lower_bound_loss2}}

\subsubsection{Proof of Lemma~\ref{lem:reduction::permutation}.}

 Consider any  polynomial $f$ in $\mathcal P_{\leq D}$. Define the permutation-invariant polynomial $f^\star$ by $f^{\star}(Y)= [n!d!]^{-1}\sum_{\pi_r\in \Pi_n, \pi_c\in \Pi_d} f(Y_{\pi_r,\pi_c})$. 
 Because of the permutation invariance of the distribution $\mathbb{P}_{M\sim \mu}$, we have  $\mathbb E_{M \sim \mu}  [f(Y)] = \mathbb E_{M \sim \mu} [f(Y_{ \pi_r, \pi_c})]$ for any permutation $(\pi_r,\pi_c)$ so that $\mathbb E_{M \sim \mu}  [f(Y)]= \mathbb E_{M \sim \mu}  [f^\star(Y)]$. Moreover, Jensen inequality together with the permutation invariance of $\mathbb{P}_{M_0}$ entail 
 \[
 \mathbb E_{M_0} [ f^\star(Y)^2]\leq [n!d!]^{-1}\sum_{\pi_r\in \Pi_n, \pi_c\in \Pi_d} \mathbb{E}_{M_0}\left[f^2(Y_{\pi_r,\pi_c})\right] = \mathbb E_{M_0} [ f(Y)^2]\enspace . 
 \]
We deduce from the two previous inequalities that $\mathrm{Adv}(f) \leq \mathrm{Adv}(f^\star)$. The advantage function is therefore maximized by a permutation-invariant polynomial. Since $(1,(\Psi_G)_{G\in\cG_{\leq D}})$ is an orthonormal basis of permutation invariant polynomials of degree at most $D$, the result follows. 

\subsubsection{Proof of Proposition~\ref{lem:lower_bound_loss}.}

We can restrict our attention to estimators $f$ such that $\E_{M_0}[f^2]>0$. Define the measure $\mu'$  over $\cH \cup \{M_0\}$ by $\mu'(\{M_0\})=1/2$ and $\mu'(\{M\})=1/(2|\cH|)$ for $M\in \cH$. We write $\mathbb E_{M\sim \mu'}[.]$ for the expectation when $M$ is sampled from $\mu'$. 
We have
\begin{align*}
\sup_{M\in \cH \cup \{M_0\}}\E_M[(f-x_0)^2] \geq  \mathbb E_{M\sim \mu'} [(f-x_0)^2] =  \frac{1}{2} + \frac{1}{2} \mathbb E_{M_0} f^2 + \frac{1}{2}  \mathbb E_{M\sim \mu} f^2 - \mathbb E_{M\sim \mu} f \enspace . 
\end{align*}
Write $z =\mathbb E^{1/2}_{M_0} f^2$ and define $g= f/z$. Note that  $\mathrm{Adv}(f) = \mathbb E_{M\sim \mu}(g)$. We get by Jensen's inequality,
\begin{align*}
1+  \mathbb E_{M_0} f^2+   \mathbb E_{M\sim \mu} f^2  - 2\mathbb E_{M\sim \mu} f&\geq  1+ z^2+z^2\mathrm{Adv}^2(f) - 2z\mathrm{Adv}(f) \geq  \frac{1}{1+\mathrm{Adv}^2(f)}\ , 
\end{align*}
where we optimized the polynomial with respect to $z$ in the last inequality. To prove the second inequality, we simply note that  
\[
\frac{2}{1+\mathrm{Adv}^2(f)}= 1 - \frac{\mathrm{Adv}^2(f)-1}{\mathrm{Adv}^2(f)+1 }= 1- (|\mathrm{Adv}(f)|-1)\frac{|\mathrm{Adv}(f)|+1}{\mathrm{Adv}^2(f)+1 }\geq 1- (|\mathrm{Adv}(f)|-1)\ , 
\]
where, in the last inequality, we consider separately the case $|\mathrm{Adv}(f)|\geq 1$ so that $|\mathrm{Adv}(f)|+1\leq |\mathrm{Adv}^2(f)|+1$ and the case 
$|\mathrm{Adv}(f)|\in [0,1)$ so that $|\mathrm{Adv}(f)|+1\geq |\mathrm{Adv}^2(f)|+1$.

\subsection{Proof of Proposition~\ref{lem:lower_bound_loss2}}
We have $\E[(f-x^{\star})^2] =   \mathbb E[x^{\star 2}]+  \mathbb E f^2 - 2\mathbb E [f x^{\star}],$
so that
$$\E[(f-x^{\star})^2]  \geq \mathbb E[x^{\star 2}] + \sqrt{\mathbb E f^2}\left[ \sqrt{\mathbb E f^2}  - 2\frac{\mathbb E[f x^{\star}] }{ \sqrt{\mathbb E f^2}}\right]=\mathbb E[x^{\star 2}] + \sqrt{\mathbb E f^2}\left[ \sqrt{\mathbb E f^2}  - 2\mathrm{Corr}^{\star}(f)\right]\enspace .$$
Since $a\left[ a  - 2b\right] \geq - b^2,$ for $a,b\in \mathbb R$, this concludes the proof.

\subsection{Proof of Lemma~\ref{lem:ortho1} and of Theorem~\ref{thm:worst_case_loss}}

\begin{proof}[Proof of Lemma~\ref{lem:ortho1}]
    Under $\mathbb{P}_{M_0}$, the canonical basis formed by the $(Y^{S})$'s is orthogonal. As a consequence, 
	for any $G^{(1)}$, $G^{(2)} \in \mathcal G_{\leq D}$, and any $\sigma^{(1)}_r\in \GS_{V^{(1)}},\sigma^{(2)}_r\in \GS_{V^{(2)}}, \sigma^{(1)}_c \in \GS_{W^{(1)}},\sigma^{(2)}_c\in \GS_{W^{(2)}}$
	$$\mathbb E_{M_0}[P_{G^{(1)},\sigma^{(1)}_r,\sigma^{(1)}_c} P_{G^{(2)},\sigma^{(2)}_r,\sigma^{(2)}_c}] = \mathbf 1\{P_{G^{(1)},\sigma^{(1)}_r,\sigma^{(1)}_c} = P_{G^{(2)},\sigma^{(2)}_r,\sigma^{(2)}_c}\}.$$
	So that
	\begin{align*}
	\mathbb E_{M_0}[\Psi_{G^{(1)}} \Psi_{G^{(2)}}] &= \frac{1}{\sqrt{\mathbb V(G^{(1)})\mathbb V(G^{(2)})}} \mathbb E_{M_0}\left[\sum_{\substack{\sigma^{(1)}_r\in \GS_{V^{(1)}}, \\ \sigma^{(1)}_c\in \GS_{W^{(1)}}}} \sum_{\substack{\sigma^{(2)}_r\in \GS_{V^{(2)}}, \\ \sigma^{(2)}_c\in \GS_{W^{(2)}}}} P_{G^{(1)},\sigma^{(1)}_r,\sigma^{(1)}_c} P_{G^{(2)},\sigma^{(2)}_r,\sigma^{(2)}_c}\right]\\ 
	&= \frac{\mathbf 1\{G^{(1)}=G^{(2)}\}}{\mathbb V(G^{(1)})}\left|\left\{\sigma^{(1)}_r\in \GS_{V^{(1)}},\sigma^{(2)}_r\in \GS_{V^{(2)}}, \sigma^{(1)}_c\in \GS_{W^{(1)}},\sigma^{(2)}_c\in \GS_{W^{(2)}}: P_{G^{(1)},\sigma^{(1)}_r,\sigma^{(1)}_c} = P_{G^{(2)},\sigma^{(2)}_r,\sigma^{(2)}_c}\right\}\right|.
	\end{align*}
This concludes the proof by definition of $\mathbb V(G)$ in \eqref{eqn:variance of graph} - since
\begin{align*}
&\left|\left\{\sigma^{(1)}_r\in \GS_{V^{(1)}},\sigma^{(2)}_r\in \GS_{V^{(2)}}, \sigma^{(1)}_c\in \GS_{W^{(1)}},\sigma^{(2)}_c\in \GS_{W^{(2)}}: P_{G^{(1)},\sigma^{(1)}_r,\sigma^{(1)}_c}= P_{G^{(2)},\sigma^{(2)}_r,\sigma^{(2)}_c}\right\}\right|\\ 
&= \left|\left\{\sigma^{(1)}_r\in \GS_{V^{(1)}}, \sigma^{(1)}_c\in \GS_{W^{(1)}}\right\}\right||\mathrm{Aut}(G^{(1)})| = \mathbb V(G^{(1)}).
\end{align*}
\end{proof}

\begin{proof}[Proof of Theorem~\ref{thm:worst_case_loss}] 
In light of~\eqref{eq:badv}, we mainly have to bound  $\mathbb E^2_{M \sim \mu}(\Psi_{G})$ for each template $G\in \mathcal{G}_{\leq D}$. 

\medskip

\noindent	
\underline{Step 1: Bound on $\mathbb E_{M \sim \mu}(\Psi_{G})$.} For any $\sigma_r \in \GS_V,\sigma\in \GS_W$, whenever $|V| \leq K_n$ and $|W| \leq K_d$, we have 
$$\mathbb E_{M \sim \mu}[P_{G,\sigma_{r},\sigma_c}] = \lambda^{|E|} \frac{K_n!(n-|V|)! }{n!(K_n - |V|)!}\frac{K_d!(d-|W|)!)}{d!(K_d - |W|)!}\enspace ,$$
and otherwise we have $\mathbb E_{M \sim \mu}[P_{G,\sigma_{r},\sigma_c}] = 0$. So that, whenever $|V| \leq K_n$ and $|W| \leq K_d$,
\begin{equation}~\label{eq:signal}
\mathbb E_{M \sim \mu}(\Psi_{G}) = \sqrt{\frac{(n-|V|)!(d-|W|)!}{|\mathrm{Aut}(G)|n!d!}} \lambda^{|E|} \frac{K_n!}{(K_n - |V|)!}\frac{K_d!}{(K_d - |W|)!} \enspace . 
\end{equation}
Since $|V|\leq K_n\leq n$, $|W|\leq K_d\leq d$, and $|\mathrm{Aut}(G)|\geq 1$, this leads us to
\begin{align*}
0\leq \mathbb E_{M \sim \mu}(\Psi_{G})
&\leq \lambda^{|E|}\left(\frac{K_n}{{\sqrt{n}}}\right)^{|V|} \left(\frac{K_d}{{\sqrt{d}}}\right)^{|W|}\enspace .
\end{align*}
If $|V|> K_n$ or $|W|> K_d$, we have $\mathbb E_{M \sim \mu}(\Psi_{G}) =0$.

\medskip 

\noindent
\underline{Step 2: Bound on $\mathrm{Adv}^2(f)$.} The last equation of the previous step implies together with Equation~\eqref{eq:badv} that
\[
\mathrm{Adv}^2(f) \leq 1+ \sum_{G \in \mathcal{G}_{\leq D} :\, |V| \leq K_n,\, |W| \leq K_d}
\lambda^{2|E|} \left( \frac{K_n}{\sqrt{n}} \right)^{2|V|} \left( \frac{K_d}{\sqrt{d}} \right)^{2|W|}.
\]
Write $G = (G_1,\ldots, G_{\mathrm{cc}})$ for the decomposition of $G$ into its $\mathrm{cc}_G$ connected components. Write $V_i,W_i,E_i$ for the set of nodes and of edges of connected component $G_i$. This leads us to
\[
\mathrm{Adv}^2(f) \leq 1+ \sum_{1\leq c \leq D} \sum_{G \in \mathcal{G}_{\leq D} :\, \mathrm{cc}_{G} = c,\, |V| \leq K_n,\, |W| \leq K_d} \prod_{i=1}^c \lambda^{2|E_i|} \left( \frac{K_n}{\sqrt{n}} \right)^{2|V_i|} \left( \frac{K_d}{\sqrt{d}} \right)^{2|W_i|}\enspace . 
\]
Hence, we mainly have to enumerate the collection of templates with a given value for $\mathrm{cc}_{G}$, $|E|$, $|V|$, and $|W|$. 
The following lemma recalls standard properties of bipartite graphs without isolated edges. 
\begin{lemma}\label{lem:countedges}
Let $G = (V,W,E)\in \mathcal G_{\leq D}$ be a template with $\mathrm{cc}_{G}$. Then, we have
\[
|E|+\mathrm{cc}_{G} \geq |V| + |W|\ , \quad |V|\land |W| \geq \mathrm{cc}_{G}\  ,\quad \text{ and }\,\quad |E| \geq |V|\vee |W|
\]
\end{lemma}
\begin{proof}
The first inequality is standard and is easily proved by counting the edges of a covering forest of $G$. To show the second inequality, we only need to observe that a connected component of $G$ contains a least a node in $V$ and an node in $W$. The last inequality is consequence of the absence of isolated nodes. 
\end{proof}

Given $(r,s,e)$, we define  $\mathcal G(r,s,e)\subset \cG_{\leq D}$ as the set of templates $G$ that are  connected and such that $|V|=r$, $|W|=s$ and $|E|=e$. Given $c\in \mathbb{N}$, we define the sequence
$$\mathcal S(r,s,e,c) := \left\{(r_i,s_i,e_i)_{1\leq i \leq c}\in (\mathbb N^{*})^{c}: \forall i, \sum r_i = r, \sum s_i = s, \sum_i e_i = e, e_i+1 \geq r_i+s_i\right\}\enspace .$$
We have that
$$\left|\{G \in \mathcal G_{\leq D}: \mathrm{cc} = c, |V| = r, |W| = s, |E| = e\}\right| \leq  \sum_{\left((r_i)_i, (s_i)_i, (e_i)_i\right)\in \mathcal S(r,s,e,c)} \prod_{i\leq c}\left|\mathcal G(r_i,s_i,e_i)\right|\enspace .$$
Write also
$$\mathcal S_D(c) = \left\{r,s,e\in \mathbb N^{*}:e+c \geq r+s, c\leq r\leq K_n, c\leq s\leq K_d, c\lor r\lor s\leq e\leq D\right\}\enspace .$$
We have 
$$\left|\{G \in \mathcal G_{\leq D}: \mathrm{cc} = c, |V|\leq K_n, |W|\leq K_d\}\right|\leq \sum_{r,s,e \in \mathcal S_D(c)}\quad \sum_{\left((r_i)_i, (s_i)_i, (e_i)_i\right)\in \mathcal S(r,s,e,c)} \quad  \prod_{i\leq c}\left|\mathcal G(r_i,s_i,e_i)\right|\enspace .$$
This leads us to the following bound for the advantage. 
\begin{align}\label{eq:advinter}
	\mathrm{Adv}^2(f) &\leq 1+ \sum_{c \leq D} \sum_{(r,s,e)\in \mathcal S_D(c)} \sum_{\substack{(r_i)_{i \leq c},\, (s_i)_{i \leq c},\, (e_i)_{i \leq c} \\ \in \mathcal S(r,s,e,c)}} \prod_{i=1}^c \left[ \left|\mathcal G(r_i,s_i,e_i)\right|  \lambda^{2e_i}  \left( \frac{K_n}{\sqrt{n}} \right)^{2r_i} \left( \frac{K_d}{\sqrt{d}} \right)^{2s_i} \right] 
\end{align}
Define the function $\Psi(.,.,.)$ by 
\[
\Psi(r, s, e) := \left|\mathcal G(r,s,e)\right|  \lambda^{2e}  \left( \frac{K_n}{\sqrt{n}} \right)^{2r} \left( \frac{K_d}{\sqrt{d}} \right)^{2s}\enspace . 
\]

\noindent
\underline{Step 3: Upper bound on $\Psi(r, s, e)$}. First, the following lemma bounds the number of non-isomorphic connected bipartite graphs with a given number of nodes and edges.
\begin{lemma}\label{lem:upper:bound:N}
For any $r,s,e\in \mathbb N^{*}$ such that $e+1 \geq r+s$, we have
$$\left|\mathcal G(r,s,e)\right| \leq	\binom{r+s-2}{r-1}\binom{r+s-2}{s-1} \min\left(\binom{e-s}{r-1}s^{e - (r+s-1)}, \binom{e-r}{s-1}r^{e - (r+s-1)}\right)\leq 2^{4e} (r\land s)^{e - (r+s-1)}.$$
\end{lemma}
From this lemma, we deduce that, for any $r,s,e\in \mathbb N^{*}$ such that $r\leq K_n\land D, s\leq K_d\land D$ and $e\geq r+s-1$, we have 
\begin{align*}
\Psi(r,s,e) &\leq {(r\land s)}^{e - (r+s-1)}\, 2^{4e}\lambda^{2e}{\left(\frac{K_n}{\sqrt{n}}\right)}^{2r} {\left(\frac{K_d}{\sqrt{d}}\right)}^{2s}\\
&\leq 2^4\,{\left[2^4(D\land K_n\land K_d) \lambda^2\right]}^{e - (r+s-1)} {\left(2^4\lambda^2 \frac{K_d^2}{d}\right)}^{s-1} {\left(2^4\lambda^2 \frac{K_n^2}{n}\right)}^{r-1}\lambda^2 {\left(\frac{K_nK_d}{\sqrt{nd}}\right)}^2\enspace .
\end{align*}
Since $e\geq r+s-1$, all exponents in the above bounds are non-negative. Using our assumption~\eqref{eq:assump}, we get 
$$\Psi(r,s,e) \leq 2^{-2\bar c e}\enspace .$$

\noindent
\underline{Step 4: Conclusion.} We now plug this upper bound on $\Psi(r,s,e)$ into Equation~\eqref{eq:advinter}. We have that
\begin{align*}
	\mathrm{Adv}^2(f) -1 &\leq  \sum_{1\leq c \leq D}\quad  \sum_{(r,s,e)\in \mathcal S_D(c)}\quad \sum_{\substack{(r_i)_{i \leq c},\, (s_i)_{i \leq c},\, (e_i)_{i \leq c} \\ \in \mathcal S(r,s,e,c)}} \prod_{i=1}^c 2^{-2\bar c e_i}\\
	&=  \sum_{1\leq c \leq D}\quad \sum_{(r,s,e)\in \mathcal S_D(c)} \quad \sum_{\substack{(r_i)_{i \leq c},\, (s_i)_{i \leq c},\, (e_i)_{i \leq c} \\ \in \mathcal S(r,s,e,c)}} 2^{-2\bar c e}.
	\end{align*}
The following lemma upper bounds the cardinality of $\mathcal S(r,s,e,c)$.
\begin{lemma}\label{lem:boundS}
	We have that
	$$| \mathcal S(r,s,e,c)| \leq \binom{r-1}{c-1}\binom{s-1}{c-1}\binom{e-1}{c-1}\leq 2^{r+s+e}.$$
\end{lemma}
Gathering this lemma and the previous bound, we arrive at 
		\begin{align*}
		\mathrm{Adv}^2(f) -1 
			&\leq   \sum_{1\leq c \leq D} \sum_{(r,s,e)\in \mathcal S_D(c)} 2^{e+r+s} 2^{-2\bar c e}\\
			& \leq   \sum_{1\leq c \leq D} \sum_{(r,s,e)\in \mathcal S_D(c)} 2^{3e}  2^{-2\bar c e}\\
			&\leq    \sum_{1\leq e\leq D} e^3 2^{3e}  2^{-2\bar c e}  \leq 2^{-\bar c }\enspace , 
	\end{align*}
	where we use in the second line that, for $(r,s,e)\in \mathcal S_D(c)$, we have that $e\geq (r\vee s)$ and that $\bar c\geq 10$.
This concludes the bound of $\mathrm{Adv}^2(f)$. Finally, we lower bound the quadratic risk with  Lemma~\ref{lem:lower_bound_loss}. This concludes the proof of the theorem.


\begin{proof}[Proof of Lemma~\ref{lem:upper:bound:N}]
Consider any unlabelled bi-partite graph with $(r,s)$ nodes and $e$ edges. To construct it, we can first build an unlabelled tree that is included on this graph and  then add the $e'=e-(r+s-1)$ remaining edges.

\noindent
\underline{Step 1: choosing a spanning tree.} Let us count the number of non-isomorphic bipartite trees with respectively $r$ and $s$ nodes. For $i=1,\ldots, r$ and $j=1,\ldots, s$, we write $d_i$ and $d'_j$ for the degree of node $i$ and node $j$. 
Given positive sequences $(a_1,\ldots, a_r)$ and $(b_1,\ldots, b_s)$ such that $\sum_{i=1}^r a_i=r+s-1$ and $\sum_{j=1}^sb_j=r+s-1$, we can build a bi-partite graph on  using the following process. Fist, connect the node $v_1$ to $w_1,\ldots, w_{a_1}$. Next,  connect $v_2$ to $(a_2-1)$ nodes $w_{a_1+1},\ldots, w_{a_1+a_2-1}$. Then, we  iteratively connect $v_i$ to $(a_i-1)$ nodes $w_{a_1+ \sum_{l=2}^{i-1} (a_i-1)+1},\ldots, w_{a_1+ \sum_{l=2}^{i-1} (a_i-1)+1}$. This way, we have added $s$ edges to graph. 
Then, we repeat the same process by reversing the role of $w_i$ and $v_i$, and counting only once the edge between $v_1$ and $w_1$. In this way, we have built a bipartite graph with degree sequences $(a_1,\ldots, a_r)$ and $(b_1,\ldots, b_s)$ and with $r+s-1$ edges. Conversely, let us show that any unlabelled bipartite tree can be constructed with this process. Given such a tree, we fix arbitrarily a root $v_1$ and denote $d_1$ its degree. 
We explore its $d_1$ neighbors that we label them respectively $w_{1},w_{2}, \ldots,w_{d_1}$ and we denote $d'_1,\ldots, d'_{d_1}$ their degrees.  Then, we explore the $d'_1-1$ new neighbors of $w_1$ and we label them $v_2,\ldots, v_{d'_1-1}$. It we iterate this process, we arrive at degree sequences $(d_1,\ldots, d_r)$ and $(d'_1,\ldots, d'_s)$ that, in turn, allow to recover the tree by the previous construction.

The number of positive integer sequences $(a_1,\ldots, a_r)$ that sum to $r+s-1$ is $\binom{r + s - 2}{r - 1}$. As a consequence, the number is non-isomorphic bipartite trees with respectively $r$ and $s$ nodes is at most
\[
\binom{r + s - 2}{r - 1} \binom{r + s - 2}{s - 1}\enspace .
\]

\noindent
\underline{Step 2: adding additional edges.} 
It remains to count the number of ways to add the $e'= e-(r+s-1)$ additional edges. For that purpose, we now arbitrarily label the nodes of the trees. 
One way of enumerate the ways of adding edges is to build sequences $(c_1,\ldots, c_r)$  of non-negative integers that sum to $e'$. Such sequences correspond to the number $c_i$ of additional edges that are connected to the node $v_i$. Then, there are $s^{e'}$ ways of connecting these $e'$ edges to one of the $[s]$ nodes. Since the number of such sequences $(c_1,\ldots, c_r)$ is equal to  $\binom{r-1}{e'+r-1}$, the number of ways to add these $e'$ edges is at most
\[
\binom{e' + r - 1}{r - 1} s^{e'}\enspace . 
\]
By symmetry, we can reverse the role of $r$ and $s$ to get the tighter bound
\[
\binom{e' + r - 1}{r - 1} s^{e'}\bigwedge \binom{e' + s - 1}{s - 1} r^{e'} \enspace . 
\]

\medskip
Recalling that $e'=e-(r+s-1)$, we combine the two steps. This leads to the first upper bound. The second upper bounds follows from the inequalities $\binom{a}{b} \leq 2^a$ and 
$e \geq r+s-1$.

\end{proof}
\begin{proof}[Proof of Lemma~\ref{lem:boundS}]
	We have
	$$| \mathcal S(r,s,e,c)| \leq \left|\left\{(r_i)_{i \leq c} \in \mathbb N^{*c}: \sum r_i = r, \right\}\right|\times \left|\left\{(s_i)_{i \leq c} \in \mathbb N^{*c}: \sum s_i = s, \right\}\right|\times \left|\left\{(e_i)_{i \leq c} \in \mathbb N^{*c}: \sum e_i = e, \right\}\right|\enspace .$$
	It is known that the number   $\left|\left\{(r_i)_{i \leq c} \in \mathbb N^{*c}: \sum r_i = r, \right\}\right|$ of such sequences is equal to $\binom{r-1}{c-1}$. The same reasoning applies to $s$ and $e$, which brings the results using also that $\binom{a}{b} \leq 2^a$.
\end{proof}
\end{proof}

\subsection{Proof of Lemma~\ref{lem:reduction::permutation2}}

	This proof is similar to the proof of Lemma~\ref{lem:reduction::permutation}, the only difference being that we rely on the permutation invariance of the distribution of $(Y,x^{\star})$.


\subsection{Proofs of Theorems~\ref{thm:isorefo2} and~\ref{prop:compest}}

The proofs follow the sames lines as in~\cite{CGGV25}, the main difference being that we consider here non-symmetric matrices.

\subsubsection{Additional graph notations}\label{sec:additional:graph:notations}
In order to control the covariance terms of the form $\mathbb{E}[\Psi^{\star}_{G^{(1)}}\Psi^{\star}_{G^{(2)}}]$, we need to introduce additional notation for graphs and labelings.

\medskip 

\noindent 
\underline{\bf Matching of nodes.} Consider two templates $G^{(1)}=(V^{(1)},W^{(1)},E^{(1)})$, $G^{(2)}=(V^{(2)},W^{(2)},E^{(2)})$. 
Given labelings $\sigma^{(1)}=(\sigma^{(1)}_r,\sigma^{(1)}_c)$ and $\sigma^{(2)}=(\sigma^{(2)}_r,\sigma^{(2)}_c)$, we say  that two nodes $v^{(1)}$ and $v^{(2)}$ (resp. $w^{(1)}$ and $w^{(2)}$) are matched if $\sigma_r^{(1)}(v^{(1)})=  \sigma_r^{(2)}(v^{(2)})$ (resp. $\sigma_c^{(1)}(w^{(1)})=  \sigma_c^{(2)}(w^{(2)})$). 
More generally,  a matching $\mathbf M_r$ stands for a set of pairs of nodes $(v^{(1)},v^{(2)})\in V^{(1)}\times V^{(2)}$ where no node in $V^{(1)}$ or $V^{(2)}$ appears twice. Similarly, we define matchings $\mathbf M_c$ for $W^{(1)}$ and $W^{(2)}$.  We denote $\mathcal M_r$ and  $\mathcal{M}_c$ for the collection of all possible node matchings, where, for $\mathcal M_r$, we have the additional constraint that $(v_1^{(1)},v_1^{(2)})$ belongs to the matching. 
Given $\mathbf M_r \in \mathcal M_r$ and $\mathbf M_c \in \mathcal{M}_c$, we define the set of labeling compatible with the matchings $\mathbf M_r$ and $\mathbf M_c$ by 
\begin{align*}
&\GS(\mathbf M_r)\\ 
&= \left\{(\sigma^{(1)}_r,\sigma^{(2)}_r)\in \GS^{\star}_{V^{(1)}}\times \GS^{\star}_{V^{(2)}}: \forall (v^{(1)},v^{(2)}) \in V^{(1)}\times V^{(2)}, \{\sigma^{(1)}_r(v^{(1)}) = \sigma^{(2)}_r(v^{(2)})\}\Longleftrightarrow \{(v^{(1)},v^{(2)}) \in \mathbf M_r\}\right\}\enspace ;  \\
&\GS(\mathbf M_c)\\
&= \left\{(\sigma^{(1)}_c,\sigma^{(2)}_c)\in \GS_{W^{(1)}}\times\GS_{W^{(2)}}: \forall (w^{(1)},w^{(2)}) \in W^{(1)}\times W^{(2)}, \{\sigma^{(1)}_c(w^{(1)}) = \sigma^{(2)}_c(w^{(2)})\}\Longleftrightarrow \{(w^{(1)},w^{(2)}) \in \mathbf M_c\}\right\}\enspace .
\end{align*}
For short, we write $\mathbf M = (\mathbf{M}_r,\mathbf M_c)$ and $\GS(\mathbf M)$ for the collections of labelings $(\sigma^{(1)}_r,\sigma^{(1)}_c,\sigma^{(2)}_r,\sigma^{(2)}_c)$ compatible with $\mathbf M$.
Given a matching $\mathbf M$, we write that two edges $e\in E^{(1)}$ and $e'\in E^{(2)}$ are matched if the corresponding
incident nodes are matched. Importantly, as $\mathbb{P}$ is permutation invariant, $\mathbb{E}[P_{G^{(1)},\sigma_r^{(1)},\sigma_c^{(1)}}P_{G^{(2)},\sigma_r^{(2)},\sigma_c^{(2)}}]$ is the same for all $(\sigma^{(1)},\sigma^{(2)})$ in $\GS(\mathbf{M})$.

\medskip

\noindent
\paragraph{Merged graph $G_\cup$, intersection graph $G_{\cap}$, and symmetric difference graph $G_{\Delta}$.} Given a template $G$ and labelings $\sigma_r$ and $\sigma_c$, we write $G[\sigma_r,\sigma_c]$ for the corresponding labelled graph. 
For such labelled bi-partite graphs $G^{(1)}[\sigma^{(1)}_r,\sigma^{(1)}_c]$ and $G^{(2)}[\sigma^{(2)}_r,\sigma^{(2)}_c]$, we define $G_{\cup}$ for the merged graph, with the convention that the two same edges are merged into a single edge. Similarly,  we define the intersection graph 
$G_{\cap}=(V_{\cap},W_{\cap},E_{\cap})$ and the symmetric difference graph $G_{\Delta}=(V_{\Delta},W_{\Delta},E_{\Delta})$ so that $E_{\Delta}=E_{\cup}\setminus E_{\cap}$. Here, $V_{\cap}$ and $W_{\cap}$ (resp. $V_{\Delta}$ and $W_{\Delta}$) are the sets of nodes induced by the edges $E_{\cap}$ (resp. $E_{\Delta}$)
so that $G_{\cap}$ (resp. $G_{\Delta}$) does not contain any isolated node. We also have $|E_{\cup}| = | E^{(1)}| +|E^{(2)}| - |E_\cap|$ and $|V_{\cup}| = |V^{(1)}|+|V^{(2)}|-|\mathbf{M}_r|$ for $(\sigma_r^{(1)},\sigma_r^{(2)})\in \GS(\mathbf M_r)$. Note that, for a fixed matching $\mathbf{M}$, all graphs $G_{\cup}$ (resp. $G_{\cap}$, $G_\Delta$) are isomorphic for $(\sigma_r^{(1)},\sigma_r^{(2)})\in \GS(\mathbf M_r)$ and $(\sigma_c^{(1)},\sigma_c^{(2)})\in \GS(\mathbf M_c)$. Hence, we shall refer to quantities such as $|E_{\Delta}|$, $|V_{\Delta}|$, $|W_{\Delta}|$\ldots associated to a matching $\mathbf{M}$.  
Finally, we write $\mathrm{cc}_{\Delta}$ for the number of connected components in $G_{\Delta}$. 

\medskip

\noindent
\paragraph{Sets of unmatched nodes and of semi-matched nodes.} 
 Write $U^{(1)}=(U^{(1)}_r,U^{(1)}_c)$, resp.~$U^{(2)}=(U^{(2)}_r,U^{(2)}_c)$ for the set of nodes in $G^{(1)}[\sigma^{(1)}_r,\sigma^{(1)}_c]$, resp.~$G^{(2)}[\sigma^{(2)}_r,\sigma^{(2)}_c]$ that are not matched, namely the {\bf unmatched nodes}, that is 
\begin{align*}
U^{(1)}_r =\sigma_r^{(1)}(V^{(1)})\setminus \sigma_r^{(2)}(V^{(2)})\ ; \quad \quad\quad 
U^{(2)}_r =\sigma_r^{(2)}(V^{(2)})\setminus \sigma_r^{(1)}(V^{(1)})\  ; \\ 
U^{(1)}_c =\sigma_c^{(1)}(W^{(1)})\setminus \sigma_c^{(2)}(W^{(2)})\ ; \quad \quad\quad 
U^{(2)}_c =\sigma_c^{(2)}(W^{(2)})\setminus \sigma_c^{(1)}(W^{(1)})\  ; 
\end{align*}
Again, $|U^{(1)}_r|$, $|U^{(2)}_r|$, $|U^{(1)}_c|$, and $|U^{(2)}_c|$ only  depend on $(\sigma^{(1)},\sigma^{(2)})$ through the matching $\mathbf{M}$. 
We have, for $i\in \{1,2\}$,
\begin{align}\label{eq:unmatched}
    |V^{(i)}| = |\mathbf M_r| + |U_r^{(i)}|\, \quad \quad |W^{(i)}| = |\mathbf M_c| + |U_c^{(i)}|\enspace . 
\end{align}
Write also
$$\mathbf M_{\mathrm{SM}}(\mathbf M_r) \subset \mathbf M_r,~~~~~~\mathbf M_{\mathrm{SM}}(\mathbf M_c) \subset \mathbf M_c,~~~~~~\mathrm{and}~~~~~\mathbf M_{\mathrm{SM}}(\mathbf M) = (\mathbf M_{\mathrm{SM}}(\mathbf M_r),\mathbf M_{\mathrm{SM}}(\mathbf M_c)) $$
for the set of node matches of $(G^{(1)},G^{(2)})$ that are matched, and yet that are not pruned when creating the symmetric difference graph $G_{\Delta}$.
This is the set of {\bf semi-matched nodes} - i.e.~at least one of their incident edges is not matched.  The remaining pairs of nodes $\mathbf M \setminus \mathbf M_{\mathrm{SM}}$ are said to be {\bf perfectly matched} as all the edges incident to them are matched. We write  $\mathbf M_{\mathrm{PM}} = \mathbf M_{\mathrm{PM}}(\mathbf M)$ for the set of perfect matches in $\mathbf M$. Note that 
\begin{equation}\label{eq:PMSM}
     |V^{(1)}| +|V^{(2)}| =   |V_{\Delta}| + |(\mathbf M_{\mathrm{SM}})_r| + 2|(\mathbf M_{\mathrm{PM}})_r|\enspace  .
\end{equation}

\medskip 

\noindent
\paragraph{ Definition of some relevant sets of nodes matchings.} 
Given $(G^{(1)},G^{(2)})$ and a matching $\mathbf M=(\mathbf M_r,\mathbf M_c)$, we define the pruned matching $\mathbf{M}^{-}=(\mathbf{M}_r^-,\mathbf{M}_c)$ as the matching $\mathbf{M}_r$ to which we remove $(v_1^{(1)},v_1^{(2)})$  if either $v_1^{(1)}$ is isolated in $G^{(1)}$ or $v_1^{(2)}$ is isolated in $G^{(2)}$. This definition accounts for the fact that, when isolated, the nodes $v_1^{(j)}$ or $j=1,2$ do not play a role in the corresponding polynomials. Then, we define  $\mathcal M^\star\subset \mathcal{M}$ for the collection of matchings $\mathbf{M}$ such that all connected components of $G^{(1)}$ and of $G^{(2)}$ intersect with $\mathbf{M}^-$. 
Finally, we introduce  $\mathcal M_{\mathrm{PM}}\subset \mathcal{M}$ for 
  the collection of perfect matchings, that is matchings $\mathcal M$ such that all the nodes in $(V^{(1)},W^{(1)})$ and $(V^{(2)},W^{(2)})$ are {\bf perfectly matched}. Note that, if $\mathbf M\in \mathcal M_{\mathrm{PM}}$, then $G_{\Delta}$ is the empty graph (with $E_{\Delta} = \emptyset$). Besides, $\mathcal{M}_{\mathrm{PM}} \neq \emptyset$ if and only $G^{(1)}$ and $G^{(2)}$ are equivalent, which amounts to $G^{(1)}= G^{(2)}$ when $G^{(1)}, G^{(2)} \in \mathcal G^{\star}_{\leq D}$.

\paragraph{Shadow matchings} Given two sets $\bar U^{(1)} = (\bar U^{(1)}_r,\bar U^{(1)}_c)\subset V^{(1)} \times W^{(1)}$ and  $\bar U^{(2)} = (\bar U^{(2)}_r,\bar U^{(2)}_c)\subset V^{(2)} \times W^{(2)}$ and for a set of node matches $\underline{\mathbf M} = (\underline{\mathbf M}_r,\underline{\mathbf M}_c) \in \mathcal M_r\times \mathcal{M}_c$, we define $\mathcal M_{\mathrm{shadow}}(\bar U^{(1)},\bar U^{(2)}, \underline{\mathbf M})$
 as the collection of matchings $\mathbf{M}$ satisfying 
 \begin{equation}\label{eq:definition:mshadow}
 \mathbf M_{\mathrm{SM}}(\mathbf M) = \underline{\mathbf M}, \quad  (\sigma^{(a)}_b)^{-1}(U_b^{(a)})=\bar U^{(a)}_b \text{ for }a=1,2 \text{ and }b=r,c \enspace  ,
\end{equation}
 for any  $(\sigma^{(1)},\sigma^{(2)})\in \GS(\mathbf{M})$. Note that, as long as~\eqref{eq:definition:mshadow} is satisfied for one labeling $(\sigma^{(1)},\sigma^{(2)})\in \GS(\mathbf{M})$, is it satisfied for all such $(\sigma^{(1)},\sigma^{(2)})$. Equivalently, $\mathcal M_{\mathrm{shadow}}(\bar U^{(1)},\bar U^{(2)}, \underline{\mathbf M})$ is the collection of all matchings that lead to the set $\underline{\mathbf M}$ of semi-matched nodes and such that  $\bar U^{(1)},\bar U^{(2)}$ correspond  to unmatched nodes in resp.~$G^{(1)}, G^{(2)}$. We say that these matchings satisfy a given {\bf shadow} $(\bar U^{(1)},\bar U^{(2)}, \underline{\mathbf M} )$. The only thing that can vary between two elements of $\mathcal M_{\mathrm{shadow}}(\bar U^{(1)},\bar U^{(2)}, \underline{\mathbf M} )$ is the matching of the nodes that are not in $\bar U^{(1)},\bar U^{(2)}$, or part of a pair of nodes in $\underline{\mathbf M}$. This matching must however ensure that all of these nodes are perfectly matched. We refer to~\cite{CGGV25} for further explanations and illustrations.

\medskip

\noindent
\paragraph{ Edit Distance between graphs.} For any two templates $G^{(1)}$ and $G^{(2)}$, we define  the so-called edit-distance.
\begin{equation}\label{eq:definition:edit:distance}
d(G^{(1)}, G^{(2)}) := \min_{\mathbf M \in \mathcal M} |E_{\Delta}| \enspace .
\end{equation}
Note that $d(G^{(1)}, G^{(2)})=0$ if and only if $G^{(1)}$ and $G^{(2)}$ are equivalent. As a consequence, if $G^{(1)}$ and $G^{(2)}$ are in $\mathcal{G}_{\leq D}$, the edit distance is equal to $0$ if and only if $G^{(1)}=G^{(2)}$.

\subsubsection{Proof of Theorem~\ref{prop:compest}}

\begin{proof}[Proof of Theorem~\ref{prop:compest}]
We first state the following lemma.
\begin{lemma}\label{lem:adv}
    If Equation~\eqref{eq:signal1} holds for $c_{\texttt{s}}>0$ large enough, we have 
\begin{align*}
|\mathbb E[x^{\star} \Psi^{\star}_G]| 
&\leq \frac{K_n}{n} D^{-4c_{\texttt{s}}|E|}\enspace, 
\end{align*}
for any $D\geq 2$ and $G\in \cG^{\star}_{\leq D}$. 
\end{lemma}

As a consequence, we have 
\begin{align*}
\mathbb E^2[x^{\star}]+ \sum_{G\in \cG^{\star}_{\leq D}} \mathbb E^2[x^{\star} \Psi^{\star}_G]
&\leq \frac{K^2_n}{n^2}\left[1+ \sum_{G: G \in \mathcal G_{\leq D}: |E|>1} D^{-8c_{\texttt{s}} |E|}\right] \\
& \leq \frac{K^2_n}{n^2}\left[1+ \sum_{e:1\leq e\leq D} (2e)^{2e}D^{-8c_{\texttt{s}} e}\right]  \leq  \frac{K^2_n}{n^2}\left[1+2D^{-4c_{\texttt{s}}}\right]
\end{align*}
Then, together with Theorem~\ref{thm:isorefo2},  this concludes the proof. 
\end{proof}

\begin{proof}[Proof of Lemma~\ref{lem:adv}]
    
Denote $\mathrm{cc}_G$ the number of connected components of $G$.

\paragraph{Case 1: $\mathrm{cc}_G =1$.} In this case, $v_1$ belongs to this connected component and
\begin{align*}\mathbb E[x^{\star} \Psi^{\star}_G ] &= \frac{1}{\sqrt{\mathbb V^{\star}(G)}}\left[\mathbb E[P_G x^{\star}] - \mathbb E[P_G] \mathbb E[x^{\star}]\right]= \frac{1}{\sqrt{\mathbb V^{\star}(G)}}\left[\mathbb E[P_G] - \mathbb E[P_G] \mathbb E[x^{\star}]\right]\\
	&= \frac{\mathbb E[P_G]}{\sqrt{\mathbb V^{\star}(G)}}\left[1 - \frac{K_n}{n}\right].
\end{align*}
Recall the definition of $\mathbb V^{\star}(G)$. 
Since $|E| \leq D$
, we have 
\begin{align*}
	\left|\mathbb E[x^{\star} \Psi^{\star}_G ]\right| &\leq \frac{\mathbb E[P_G]}{\sqrt{\mathbb V^{\star}(G)}} \leq \sqrt{\frac{(n-1)!d!}{(n-|V|)!(d-|W|)!}} \left(\frac{K_n}{n}\right)^{|V|}\left(\frac{K_d}{d}\right)^{|W|}\lambda^{|E|}\\
	&\leq n^{(|V|-1)/2}d^{|W|/2}\left(\frac{K_n}{n}\right)^{|V|}\left(\frac{K_d}{d}\right)^{|W|}\lambda^{|E|} = \frac{1}{\sqrt{n}}\left(\frac{K_n}{\sqrt{n}}\right)^{|V|}\left(\frac{K_d}{\sqrt{d}}\right)^{|W|}\lambda^{|E|}\\
	&= \frac{K_n}{n}\left(\frac{K_n}{\sqrt{n}}\lambda\right)^{|V|-1}\left(\frac{K_d}{\sqrt{d}}\lambda\right)^{|W|}\lambda^{|E|-|W|-|V|+1}\enspace . 
\end{align*}
Since $|E|\geq |W|+|V|-1$ for a connected bipartite graph, we deduce from  Equation~\eqref{eq:signal1} that 
\begin{align*}
	\left|\mathbb E[x^{\star} \Psi^{\star}_G ]\right| &\leq  \frac{K_n}{n}D^{-8c_{\texttt{s}}(|V|-1)}D^{-8c_{\texttt{s}}|W|}D^{-8c_{\texttt{s}}(|E|-|W|-|V|+1)} = \frac{K_n}{n} D^{-8c_{\texttt{s}}|E|}\enspace . 
\end{align*}

\paragraph{Case 2: $\mathrm{cc}_G >1$.} Then, there exists one connected component $G'=(V',W',E')$ in $G$ such that $v_1 \not\in V'$. Writing  $G''$ for the rest of the graph, we have $\overline{P}^{\star}_G = \overline{P}^{\star}_{G'}\overline{P}^{\star}_{G''}$. We conclude that 
$$\mathbb E[x^{\star} \Psi^{\star}_G ] = \mathbb E[\overline P^{\star}_{G'}]\mathbb E[x^{\star}\overline P^{\star}_{G''}] =0\enspace ,$$
since $\mathbb E[\overline P^{\star}_{G'}] = 0$. 
\end{proof}

\subsubsection{Proof of Theorem~\ref{thm:isorefo2}}

\begin{proof}[Proof of Theorem~\ref{thm:isorefo2}]

Define the Gram matrix $\Gamma$ by $\Gamma_{G^{(1)},G^{(2)}}=\E\left[\Psi^{\star}_{G^{(1)}} \Psi^{\star}_{G^{(2)}}\right]$ for any $G^{(1)}$ and $G^{(2)}$ in $\cG^{\star}_{\leq D}$, $\Gamma_{\emptyset,G}= \Gamma_{G, \emptyset}= \E[\Psi^{\star}_G]$ for any $G\in \cG^{\star}_{\leq D}$, and $\Gamma_{\emptyset,\emptyset}= 1$. To establish Theorem~\ref{thm:isorefo2}, we only have to prove that   $\|\Gamma -Id\|_{op}\leq 4D^{-c_{\texttt{s}/2}}$ where $Id$ is the identity matrix and $\|.\|_{op}$ is the operator norm.

First, observe that $\E[\Psi^\star_G]=0$ for any $G\in \cG^\star_{\leq D}$ because the polynomials $\overline{P}^\star_G$ are centered. Hence, $1$ is orthogonal to all $\Psi^{\star}_G$'s and we only have to focus on the submatrix of $\Gamma$ induced by the templates $G$'s. Then, we shall control individually each entry  $\Gamma_{G^{(1)},G^{(2)}}$. For this purpose, we start with a few simple observations.

Consider two templates $G^{(1)}, G^{(2)}$, some node matching $\mathbf M \in \mathcal M$ and $(\sigma^{(1)}, \sigma^{(2)})\in \GS(\mathbf{M})$. By definition of the model, we have $Y_{i,k}^2=1$ a.s. and $\mathbb{E}[Y_{i,k}|M]= M_{i,k}=\lambda\1\{i\in S_n\}\1\{k\in S_d\}$. This leads us to 
\begin{equation}~\label{eqn:old basis proof ingredients}
     \mathbb{E}\left[P_{G^{(1)}, \sigma^{(1)}}P_{G^{(2)},\sigma^{(2)}}\right] = \lambda^{|E_{\Delta}|}\left(\frac{K_n}{n}\right)^{|V_{\Delta}|}\left(\frac{K_d}{d}\right)^{|W_{\Delta}|}\ ; \quad \quad 
     \mathbb{E}\left[P_{G^{(1)}, \sigma^{(1)}}\right]= \lambda^{|E^{(1)}|}\left(\frac{K_n}{n}\right)^{|V^{(1)}|-\mathbf{1}\{v_1^{(1)} \text{isolated}\}}\left(\frac{K_d}{d}\right)^{|W^{(1)}|}
     \enspace . 
\end{equation}
Besides, as the the variables $\1\{i\in S_n\}$ and $\1\{i\in S_d\}$ are independent, it follows that, for any $(S^{(1)},T^{(1)})$ and $(S^{(2)},T^{(2)})$ with $S^{(1)}\cap S^{(2)}=\emptyset$ and $T^{(1)}\cap T^{(2)}=\emptyset$, and  measurable bounded function $f$ and $g$ we have 
\begin{equation}\label{eq:indep}
    \mathbf E\left[f(Y_{S^{(1)},T^{(1)}})g(Y_{S^{(2)},T^{(2)}})\right]=\mathbf E\left[f(Y_{S^{(1)},T^{(1)}})\right]\mathbf E\left[g(Y_{S^{(2)},T^{(2)}})\right]\enspace .
\end{equation}
In order to control $\E\left[\Psi^{\star}_{G^{(1)}} \Psi^{\star}_{G^{(2)}}\right]$, we first connect the covariance of the  $\overline{P}^{\star}_{G^{(1)}, \sigma^{(1)}}$'s to that of the polynomial $P_{G^{(1)}, \sigma^{(1)}}$. 

\begin{proposition} \label{eq:boundbar}
     Under Condition~\eqref{eq:signal1}, the following holds
    \begin{enumerate}
    \item if $\mathbf{M}\notin\mathcal{M}^{\star}$ we have $\mathbb{E}\left[\overline{P}^{\star}_{G^{(1)}, \sigma^{(1)}}\overline{P}^{\star}_{G^{(2)},\sigma^{(2)}}\right] = 0$ for any $(\sigma^{(1)},\sigma^{(2)})\in\GS[\mathbf M]$\ ; 
    \item if $\mathbf{M}\in \mathcal{M}^\star$ we have
    \begin{equation*}
        \left|\dfrac{\mathbb{E}\left[\overline{P}^{\star}_{G^{(1)}, \sigma^{(1)}}\overline{P}^{\star}_{G^{(2)},\sigma^{(2)}}\right] - \mathbb{E}\left[P_{G^{(1)}, \sigma^{(1)}}P_{G^{(2)},\sigma^{(2)}}\right]}{\mathbb{E}\left[P_{G^{(1)}, \sigma^{(1)},\sigma^{(1)}_c}P_{G^{(2)},\sigma^{(2)}}\right]}\right| \leq 4D\left(\frac{K_n}{n}\land \frac{K_d}{d}\right)\enspace  . 
    \end{equation*}
    \end{enumerate} 
\end{proposition}
Then, working out~\eqref{eqn:old basis proof ingredients} and carefully summing over all labelings $(\sigma^{(1)},\sigma^{(2)})$, we can control $\mathbb{E}[\Psi_{G^{(1)}}\Psi_{G^{(2)}}]$. 
\begin{proposition}\label{prop:scalprod}
    Consider two templates $G^{(1)},G^{(2)} \in \mathcal G_{\leq D}$. We have
    \begin{equation*}
     |\Gamma_{G^{(1)}, G^{(2)}} -1|=    \left|\mathbb E\left[\Psi^{\star}_{G^{(1)}}\Psi^{\star}_{G^{(2)}}\right] - \mathbf 1\{G^{(1)} = G^{(2)}\}\right|\leq 2D^{-c_{\mathrm{s}} (d (G^{(1)}, G^{(2)})\lor 1)}\enspace , 
    \end{equation*}
    where we recall that $d(.,.)$ is the edit distance. 
\end{proposition} 
We can transform the entry-wise result of the last proposition on how $\Gamma$ is close to the identity matrix to an operator norm result, which is then very related to Theorem~\ref{thm:isorefo2} in our specific model.
\begin{proposition}\label{prop:ortho}
We have:
    \[
\|\Gamma - \mathrm{Id}\|_{op}\leq 4 D^{-c_{\texttt{s}}/2}\enspace . 
\]
\end{proposition}

\end{proof}

\begin{proof}[Proof of Proposition~\ref{eq:boundbar}]

\noindent
{\bf Proof of 1):} If $\mathbf M \not\in \mathcal M^\star$, then there exists one connected component $G'$ belonging to either $G^{(1)}$ or $G^{(2)}$ such that none of the nodes of $G'$ is matched in $\mathbf M$. Assume w.l.o.g.~that this connected component is in $G^{(1)}$. By definition of the polynomials that $\overline P^{\star}_{G^{(1)}, \sigma^{(1)}} = \overline P^{\star}_{G', \sigma^{(1)}} \overline P^{\star}_{\underline{G}^{(1)}, \sigma^{(1)}}$ where $\underline{G}^{(1)}$ is the complement of $G'$ in $G^{(1)}$. Then, Identity~\eqref{eq:indep} implies that 
$$\mathbb E[\overline{P}^{\star}_{G^{(1)}, \sigma^{(1)}}\overline{P}^{\star}_{G^{(2)},\sigma^{(2)}}] = \mathbb E[\overline{P}^{\star}_{\underline{G}^{(1)}, \sigma^{(1)}}\overline{P}^{\star}_{G^{(2)},\sigma^{(2)}}] \mathbb E[\overline{P}^{\star}_{G', \sigma^{(1)}}] = 0\enspace ,$$
as $\mathbb E[\overline{P}_{G', \sigma^{(1)}}]=0$.

\noindent
{\bf Proof of 2):} Write $G^{(1)} = (G_1^{(1)}, \ldots, G^{(1)}_{\mathrm{cc}_{G^{(1)} }})$, and $G^{(2)} = (G_1^{(2)}, \ldots, G^{(2)}_{\mathrm{cc}_{G^{(2)} }})$ for the decomposition of $G^{(1)},G^{(2)}$ into their respective $\mathrm{cc}_{G^{(1)} },\mathrm{cc}_{G^{(2)} }$ connected components that contain at least two nodes. We have
\begin{align*}
    \mathbb{E}\left[\overline{P}^{\star}_{G^{(1)}, \sigma^{(1)}_r, \sigma^{(1)}_c}\overline{P}^{\star}_{G^{(2)}, \sigma^{(2)}_r, \sigma^{(2)}_c}\right] & = \sum_{S_1 \subset [\mathrm{cc}_{G^{(1)} }], S_2 \subset [\mathrm{cc}_{G^{(2)} }]} (-1)^{|S_1|+|S_2|}\mathbb E\left[\prod_{i\in [\mathrm{cc}_{G^{(1)} }]\setminus S_1} P_{G_i^{(1)}, \sigma^{(1)}_r, \sigma^{(1)}_c}\prod_{i\in [\mathrm{cc}_{G^{(2)} }]\setminus S_2} P_{G_i^{(2)}, \sigma^{(2)}_r, \sigma^{(2)}_c}\right]\\ 
    & \quad\quad \quad  \times  \prod_{i\in S_1}\mathbb E\left[ P_{G_i^{(1)}, \sigma^{(1)}_r, \sigma^{(1)}_c}\right]\prod_{i\in  S_2}\mathbb E\left[ P_{G_i^{(2)}, \sigma^{(2)}_r, \sigma^{(2)}_c}\right]\enspace . 
\end{align*}
Since $\mathbf M\in \mathcal M^\star$, at least one node of each connected components of $G_i^{(1)}$ (resp. $G_i^{(2)}$) for $i\in \mathrm{cc}_{G^{(1)}}$ (resp. $\mathrm{cc}_{G^{(2)}}$) intersects $\mathbf{M}$. Hence, we deduce from Equation~\eqref{eqn:old basis proof ingredients} that 
\begin{align*}
0 \leq & \mathbb E\left[\prod_{i\in [\mathrm{cc}_{G^{(1)} }]\setminus S_1} P_{G_i^{(1)}, \sigma^{(1)}_r, \sigma^{(1)}_c}\prod_{i\in [\mathrm{cc}_{G^{(2)} }]\setminus S_2} P_{G_i^{(2)}, \sigma^{(2)}_r, \sigma^{(2)}_c}\right]
 \prod_{i\in S_1}\mathbb E\left[ P_{G_i^{(1)}, \sigma^{(1)}_r, \sigma^{(1)}_c}\right]\prod_{i\in  S_2}\mathbb E\left[ P_{G_i^{(2)}, \sigma^{(2)}_r, \sigma^{(2)}_c}\right] \\ & \hspace{3cm}\leq \left(\frac{K_n}{n} \land \frac{K_d}{d}\right)^{|S_1| \lor |S_2|}\mathbb{E}\left[P_{G^{(1)}, \sigma^{(1)}}P_{G^{(2)},\sigma^{(2)}}\right]\enspace . 
\end{align*}
Since $\mathrm{cc}_{G^{(1)}}\vee  \mathrm{cc}_{G^{(2)}}\leq D$, this leads us to
\begin{align*}
    \frac{\left|\mathbb{E}\left[\overline{P}_{G^{(1)}, \sigma^{(1)}}\overline{P}_{G^{(2)},\sigma^{(2)}}\right] - \mathbb{E}\left[P_{G^{(1)}, \sigma^{(1)}}P_{G^{(2)},\sigma^{(2)}}\right]\right|}{\mathbb{E}\left[P_{G^{(1)}, \sigma^{(1)}}P_{G^{(2)},\sigma^{(2)}}\right]} 
	&\leq  \sum_{S_1 \subset [\mathrm{cc}_{G^{(1)} }], S_2 \subset [\mathrm{cc}_{G^{(2)} }]: |S_1|\lor |S_2| \geq 1} \left(\frac{K_n}{n} \land \frac{K_d}{d}\right)^{|S_1| \lor |S_2|}\\
    &\leq \sum_{k=1}^{\infty}\left[2D\left(\frac{K_n}{n}\wedge \frac{K_d}{d}\right)\right]^k \\
    &\leq 4D \left(\frac{K_n}{n}\wedge \frac{K_d}{d}\right)\enspace ,
\end{align*}
since $4D  \left(\frac{K_n}{n}\wedge \frac{K_d}{d}\right) \leq 1$ by Condition~\eqref{eq:signal1}. 
\end{proof}

\begin{proof}[Proof of Proposition~\ref{prop:scalprod}]
By definition of the polynomials, we have 
\begin{align*}
     \mathbb E\left[\Psi^{\star}_{G^{(1)}}\Psi^{\star}_{G^{(2)}}\right]  &= \sum_{\mathbf{M}\in\mathcal{M}}\sum_{(\sigma^{(1)},\sigma^{(2)})\in \GS(\mathbf{M})}\frac{1}{\sqrt{\mathbb{V}^{\star}(G^{(1)})\mathbb{V}^{\star}(G^{(2)})}}\mathbb{E}\left[\overline{P}^{\star}_{G^{(1)}, \sigma^{(1)}}\overline{P}^{\star}_{G^{(2)},\sigma^{(2)}}\right]\\
     &= \sum_{\mathbf{M}\in\mathcal{M}^{\star}}\sum_{(\sigma^{(1)},\sigma^{(2)})\in \GS(\mathbf{M})}\frac{1}{\sqrt{\mathbb{V}^{\star}(G^{(1)})\mathbb{V}^{\star}(G^{(2)})}}\mathbb{E}\left[\overline{P}^{\star}_{G^{(1)}, \sigma^{(1)}}\overline{P}^{\star}_{G^{(2)},\sigma^{(2)}}\right],
    \end{align*}
where the second line follows from Proposition~\ref{eq:boundbar}. 

\noindent
{\bf Step 1: Decomposition of the cross-moment over $\mathcal M^* \setminus \mathcal M_{\mathrm{PM}}$ and $\mathcal M_{\mathrm{PM}}$.} The collections $\mathcal{M}_{\mathrm{PM}}$ is non-empty only if $G^{(1)} = G^{(2)}$. And if $G^{(1)} = G^{(2)}$, we have for any $ \sigma \in \GS^{\star}_{V^{(1)}}\times \GS_{W^{(1)}}$
$$\sum_{\mathbf{M}\in\mathcal{M}_{\mathrm{PM}}}\sum_{(\sigma^{(1)},\sigma^{(2)})\in \GS(\mathbf{M})}\frac{\mathbb{E}\left[\overline{P}^{\star}_{G^{(1)}, \sigma^{(1)}}\overline{P}^{\star}_{G^{(2)},\sigma^{(2)}}\right]}{\sqrt{\mathbb{V}^{\star}(G^{(1)})\mathbb{V}^{\star}(G^{(2)})}} =  \frac{(n-1)!d!|\mathrm{Aut}^{\star}(G^{(1)})|\mathbb{E}\left[\left(\overline{P}^{\star}_{G^{(1)}, \sigma}\right)^2\right]}{\left(n-|V^{(1)}|\right)!\left(d-|W^{(1)}|\right)!\mathbb{V}^{\star}(G^{(1)})}  =  \mathbb E\left[\overline{P}_{G^{(1)}, \sigma}^{\star2}\right],$$
since  
\begin{align}\label{eq:boubou}
    |\mathcal{M}_{\mathrm{PM}}| =|\mathrm{Aut}^\star(G^{(1)})|~~~~\mathrm{and}~~~~|\GS(\mathbf{M})| = \frac{(n-1)!d!}{\left(n-(|V^{(1)}|+ |V^{(2)}| - |\mathbf M_r|)\right)!\left(d-(|W^{(1)}|+ |W^{(2)}| - |\mathbf M_c|)\right)!},
\end{align}
and by definition of $\mathbb V(G^{(1)})$. Since $P_{G^{(1)}, \sigma}^{2}=1$ almost surely, it follows from Proposition~\ref{eq:boundbar} that 
$$\left|  \mathbb E[\overline{P}_{G^{(1)}, \sigma}^{\star2}]- 1\right| \leq 4D\left(\frac{K_n}{n}\land \frac{K_d}{d}\right) \ ,$$
so that
\begin{align} \nonumber
    \left|\mathbb E\left[\Psi^{\star}_{G^{(1)}}\Psi^{\star}_{G^{(2)}}\right] - \mathbf 1\{G^{(1)} = G^{(2)}\} \right|&\leq  \left|\sum_{\mathbf{M}\in\mathcal{M}^{\star}\setminus \mathcal M_{\mathrm{PM}}}\sum_{(\sigma^{(1)},\sigma^{(2)})\in \GS(\mathbf{M})}\frac{\mathbb{E}\left[\overline{P}^{\star}_{G^{(1)}, \sigma^{(1)}}\overline{P}^{\star}_{G^{(2)},\sigma^{(2)}}\right]}{\sqrt{\mathbb{V}^{\star}(G^{(1)})\mathbb{V}^{\star}(G^{(2)})}}\right|  + 4D \left(\frac{K_n}{n}\land \frac{K_d}{d}\right)\mathbf 1\{G^{(1)} = G^{(2)}\}\\
   &\quad \quad   := A  + 4D\left(\frac{K_n}{n}\land \frac{K_d}{d}\right)\mathbf 1\{G^{(1)} = G^{(2)}\}\enspace . \label{eq:general:bound:psi}
    \end{align}

\noindent
{\bf Step  2: Making $A$ explicit as a sum of $A(\mathbf M)$.} Observe that, for any $\mathbf M \in \mathcal M$, $\mathbb{E}\left[\overline{P}^{\star}_{G^{(1)}, \sigma^{(1)}}\overline{P}^{\star}_{G^{(2)}, \sigma^{(2)}}\right]$ does not depend on the choice of  $(\sigma^{(1)}, \sigma^{(2)})\in \GS(\mathbf{M})$. Here, we write $E(\mathbf M)$ for this value.  
So that by Equation~\eqref{eq:boubou}
\begin{align}
    A&= \left|\frac{1}{\sqrt{\mathbb{V}^{\star}(G^{(1)})\mathbb{V}^{\star}(G^{(2)})}}\sum_{\mathbf{M}\in\mathcal{M}^{\star}\setminus \mathcal{M}_{\mathrm{PM}}}\sum_{(\sigma^{(1)},\sigma^{(2)})\in \GS(\mathbf{M})}E(\mathbf M)\right|\nonumber\\
    &= \left|\frac{1}{\sqrt{\mathbb{V}^{\star}(G^{(1)})\mathbb{V}^{\star}(G^{(2)})}}\sum_{\mathbf{M}\in\mathcal{M}^{\star}\setminus \mathcal{M}_{\mathrm{PM}}}\frac{(n-1)!}{\left(n-(|V^{(1)}|+ |V^{(2)}| - |\mathbf M_r|)\right)!}\frac{d!}{\left(d-(|W^{(1)}|+ |W^{(2)}| - |\mathbf M_c|)\right)!}E(\mathbf M)\right|.\nonumber
    \end{align}
    By Proposition~\ref{eq:boundbar} and Equation~\eqref{eqn:old basis proof ingredients} we have
    $$|E(\mathbf M)| \leq \lambda^{|E_{\Delta}|}\left(\frac{K_n}{n}\right)^{|V_{\Delta}|}\left(\frac{K_d}{d}\right)^{|W_{\Delta}|} 
    \left(1+4D\left(\frac{K_n}{n}\land \frac{K_d}{d}\right)\right)\leq 2 \lambda^{|E_{\Delta}|}\left(\frac{K_n}{n}\right)^{|V_{\Delta}|}\left(\frac{K_d}{d}\right)^{|W_{\Delta}|} \enspace ,$$
where we rely again on Condition~\eqref{eq:signal1}. By definition, we have
\[
\frac{(n-1)!}{\left(n-(|V^{(1)}|+ |V^{(2)}| - |\mathbf M_r|)\right)!} \frac{\sqrt{(n-|V^{(1)|})!(n-|V^{(2)|})!}}{(n-1)!}\leq n^{\frac{|U_r^{(1)}| + |U_r^{(2)}|}{2}}\ ,
\]
where the numbers  $|U_r^{(1)}|$ and $|U_r^{(2)}|$ of unmatched nodes do not depend on the choice of  $(\sigma^{(1)}, \sigma^{(2)})\in \GS(\mathbf{M})$. Arguing similarly for the terms in $d$, we then deduce from the 
definition of $\mathbb V(G)$ that 
    \begin{align} \nonumber
    A\leq & \frac{2}{\sqrt{\left|\mathrm{Aut}^{\star}(G^{(1)})\right|\left|\mathrm{Aut}^{\star}(G^{(2)})\right|}}\sum_{\mathbf{M}\in\mathcal{M}^{\star}\setminus \mathcal{M}_{\mathrm{PM}}}n^{\frac{|U^{(1)}_r| + |U^{(2)}_r|}{2}}d^{\frac{|U^{(1)}_c| + |U^{(2)}_c|}{2}}\lambda^{|E_{\Delta}|}\left(\frac{K_n}{n}\right)^{|V_{\Delta}|}\left(\frac{K_d}{d}\right)^{|W_{\Delta}|} \nonumber\\ \nonumber
        \leq & \frac{2}{\sqrt{\left|\mathrm{Aut}^{\star}(G^{(1)})\right|\left|\mathrm{Aut}^{\star}(G^{(2)})\right|}}\sum_{\mathbf{M}\in\mathcal{M}^{\star}\setminus \mathcal{M}_{\mathrm{PM}}}\left(\frac{\lambda K_n}{\sqrt{n}}\right)^{|U^{(1)}_r| + |U^{(2)}_r|}\left(\frac{\lambda K_d}{\sqrt{d}}\right)^{|U^{(1)}_c| + |U^{(2)}_c|}\\
		&\hspace{3cm}\quad \quad \times \lambda^{|E_{\Delta}| - |U^{(1)}_r| -|U^{(1)}_c| -|U^{(2)}_r|- |U^{(2)}_c|}\left(\frac{K_n}{n}\right)^{|V_{\Delta}| - |U^{(1)}_r| - |U^{(2)}_r|}\left(\frac{K_d}{d}\right)^{|W_{\Delta}| - |U^{(1)}_c| - |U^{(2)}_c|}\enspace ,\label{eq:A(M)}
\end{align}
where we  rearranged terms in the last line. Write $A(\mathbf M)$ for the summand in the last line.
    

\medskip 

\noindent
{\bf Step 3: Grouping matching $\mathbf M$ sharing the same shadow.} Recall the definition of shadows in Section~\ref{sec:additional:graph:notations}. Since $A(\mathbf{M})$ only depends on the number of edges and nodes of the symmetric difference graphs and on the number of unmatched nodes, two matchings $\mathbf{M}$ and $\mathbf{M}'$ that share the same shadow satisfy $A(\mathbf{M})=A(\mathbf{M}')$. Regrouping the sum, we arrive at  
\begin{align*}
       A &\leq  \frac{2}{\sqrt{|\mathrm{Aut}^{\star}(G^{(1)})| |\mathrm{Aut}^{\star}(G^{(2)})|}} 
       \sum_{\substack{\bar U^{(1)} \subset V^{(1)}\times W^{(1)}\\ \bar U^{(2)} \subset V^{(2)}\times W^{(2)}, \\ \underline{\mathbf M} \in \mathcal M\setminus \mathcal M_{\mathrm{PM}}}}
     \quad \quad   \sum_{\mathbf M \in \mathcal M_{\mathrm{shadow}}(\bar U^{(1)},\bar U^{(2)}, \underline{\mathbf M})} A(\mathbf M)\enspace .
\end{align*}
Now, we have to bound the cardinality $| \mathcal M_{\mathrm{shadow}}(\bar U^{(1)},\bar U^{(2)}, \underline{\mathbf M})|$. 
\begin{lemma}\label{lem:control:shadow}
    For any $\bar U^{(1)}$, $\bar U^{(2)}$, and any $\underline{\mathbf M}$, we have 
    $$|\mathcal M_{\mathrm{shadow}}(\bar U^{(1)}, \bar U^{(2)}, \underline{\mathbf M} )| \leq \min(|\mathrm{Aut}^{\star}(G^{(1)})|, |\mathrm{Aut}^{\star}(G^{(2)}))|\enspace.$$    
\end{lemma}
  It readily follows from this lemma that
\begin{align}
       A &\leq  2 \sum_{\substack{\bar U^{(1)} \subset V^{(1)}\times W^{(1)}\\ \bar U^{(2)} \subset V^{(2)}\times W^{(2)}, \\ \underline{\mathbf M} \in \mathcal M\setminus \mathcal M_{\mathrm{PM}}}} A(\mathbf M )\enspace , \label{eqn:sum for proof sketch}
\end{align}
where $\mathbf M$ is any element of $ \mathcal M_{\mathrm{shadow}}(\bar U^{(1)},\bar U^{(2)}, \underline{\mathbf M})$.

\noindent
{\bf Step 4: Bounding $A(\mathbf M )$.} Consider any $\mathbf{M}\in  \mathcal M^*\setminus \mathcal M_{\mathrm{PM}}$ and any labeling. As $\mathbf{M}\in  \mathcal M^*$, any connected component of $G_{\Delta}$ intersects a semi-matched node, that is $|\mathbf{M}_{\mathrm{SM}}(\mathbf M_r)|+ |\mathbf{M}_{\mathrm{SM}}(\mathbf M_c)|\geq \mathrm{cc}_{\Delta}$. For any bi-partite graph $G = (V,W,E)$, we have $|E|\geq |V|+|W| - \mathrm{cc}_G$. Hence it follows that 
\[
|E_{\Delta}|\geq |V_{\Delta}|+ |W_{\Delta}|- |\mathbf{M}_{\mathrm{SM}}(\mathbf M_r)|- |\mathbf{M}_{\mathrm{SM}}(\mathbf M_c)|\enspace . 
\]
By construction of $G_{\Delta}$ and definition of the unmatched nodes, we have $|\mathbf{M}_{\mathrm{SM}}(\mathbf M_r)| + |U^{(1)}_r| + |U^{(2)}_r| = |V_{\Delta}|$ and $|\mathbf{M}_{\mathrm{SM}}(\mathbf M_c)| + |U^{(1)}_c| + |U^{(2)}_c| = |W_{\Delta}|$. Thus, we get 
\begin{align*}
|E_{\Delta}| - |U^{(1)}| - |U^{(2)}|\geq 0 \enspace . 
\end{align*}
We deduce from this that and from the condition~\eqref{eq:signal1} and from the definition~\eqref{eq:A(M)} of $A(\mathbf M)$ that
\begin{align*}
    A(\mathbf M ) &\leq D^{-8c_{\texttt{s}}(|U^{(1)}| + |U^{(2)}|)} D^{-8c_{\texttt{s}}  [ |E_\Delta| -(|U^{(1)}| + |U^{(2)}|)]} D^{-8c_{\texttt{s}}|\mathbf M_{\mathrm{SM}}|}\leq D^{-8c_{\texttt{s}}[|E_\Delta| + |\mathbf M_{\mathrm{SM}}|]}\\ 
    &\leq D^{-4c_{\texttt{s}}\left[d(G^{(1)}, G^{(2)})\lor 1+|U^{(1)}|+ |U^{(2)}| + |\mathbf M_{\mathrm{SM}}|\right]}\ , 
\end{align*}
since, by definition of the edit distance, $|E_{\Delta}| \geq d(G^{(1)}, G^{(2)})\lor 1$ as long as  $\mathbf M\not\in \mathcal M_{\mathrm{PM}}$.

\noindent
{\bf Step 5: Final bound on $A$.} Plugging this bound on $A(\mathbf M )$ back in Equation~\eqref{eqn:sum for proof sketch} we get
\begin{align}
       A &\leq  2 \sum_{\substack{\bar U^{(1)} \subset V^{(1)}\times W^{(1)}\\ \bar U^{(2)} \subset V^{(2)}\times W^{(2)}, \\ \underline{\mathbf M} \in \mathcal M\setminus \mathcal M_{\mathrm{PM}}}}  D^{-4c_{\texttt{s}}\left[d(G^{(1)}, G^{(2)})\lor 1+|\bar U^{(1)}|+ |\bar U^{(2)}| + |\underline {\mathbf M}|\right]}.
\end{align}
So when we enumerate over all possible sets $U^{(1)}, U^{(2)}, \underline{\mathbf M} $ that have respective cardinality $u_1$, $u_2$, and  $m$, and since these sets have bounded cardinalities resp.~by $(2D)^{u_1}, (2D)^{u_2}$ and $(2D)^{2m}$, we get
\begin{align*}
       A &\leq  2
       \sum_{\substack{u_1, u_2, m \geq 0,\\ u_1+u_2+m\geq 1}}(2D)^{u_1+u_2+2m}  D^{-4c_{\texttt{s}}\left[d(G^{(1)}, G^{(2)})\lor 1+u_1+u_2+m\right]} \leq D^{-c_{\texttt{s}}  (d(G^{(1)},G^{(2)})\lor 1)},
\end{align*}
since $D\geq 2$ and since we assume that  $c_{\texttt{s}} \geq 4$. Together with~\eqref{eq:general:bound:psi}, this concludes the proof. 
\end{proof}

\begin{proof}[Proof of Proposition~\ref{prop:ortho}]
    Recall that $\Gamma_{1,G}=0$ for any $G\in \mathcal G^\star_{\leq D}$, we only have to focus on the submatrix of $\Gamma$ indexed by $\mathcal G^\star_{\leq D}$. Since the operator norm of a symmetric matrix is bounded by the maximum $\ell_1$ norm of its rows, we have 
\[
\|\Gamma - \mathrm{Id}\|_{op}\leq \max_{G^{(1)}}\left\{\left|\Gamma_{G^{(1)}, G^{(1)}}-1\right|+  \sum_{G^{(2)} \in \mathcal G^\star_{\leq D}, G^{(2)} \neq G^{(1)} }\left|\Gamma_{G^{(1)}, G^{(2)}}\right|\right\} \enspace . 
\] 
To bound the latter sum, we use that for a fixed template $G^{(1)}$, the number of templates $G^{(2)}$ such that $d( G^{(1)}, G^{(2)}) = u$ is bounded by $(u+2D+1)^{2u}$. Also, if $G^{(2)} \neq G^{(1)}$, then $d(G^{(1)}, G^{(2)})\geq 1$. It then follows from Proposition~\ref{prop:scalprod} that
\begin{align*}
    \sum_{G^{(2)} \in \mathcal G_{\leq D}, G^{(2)} \neq G^{(1)} }\left|\Gamma_{G^{(1)}, G^{(2)}}\right|
    &\leq \sum_{G^{(2)} \in \mathcal G_{\leq D}, G^{(2)} \neq G^{(1)}} 2D^{-c_{\texttt{s}}d(G^{(1)},G^{(2)})}
    \leq \sum_{2D \geq  u \geq 1} |\{G^{(2)}: d( G^{(1)}, G^{(2)}) = u \}|  2D^{-c_{\texttt{s}}u}\enspace\\
    &\leq \sum_{2D \geq u\geq 1} 2(u+2D+1)^{2u}  D^{-c_{\texttt{s}}u}\enspace
    \leq \sum_{2D \geq u\geq 1} 2D^{-(c_{\texttt{s}}-8)u} \leq  2D^{-c_{\texttt{s}}/2}\enspace ,
\end{align*}
since $D \geq 2$ provided we have $c_{\texttt{s}} \geq 18$. 
Using Proposition~\ref{prop:scalprod}, to bound $\left|\Gamma_{G^{(1)}, G^{(1)}}-1\right|$, we conclude the proof.
\end{proof}

\begin{proof}[Proof of Lemma~\ref{lem:control:shadow}]
This proof if analogous to that of Lemma A.5 in~\cite{CGGV25}.
\end{proof}

\section{Appendix}\label{sec:proof:thm:UBtest}

This section gathers the proofs of the upper bounds as well the information lower bounds for hidden submatrix problems. Variants of these proofs already appeared in the literature --see e.g.~\cite{butucea2013detection,butucea2015sharp}-- and we mostly provide them for the sake of completeness.

\subsection{Proof of Proposition~\ref{thm:UBtest}}

    Note first that for $f$ taking values in $\{0,1\}$:
\begin{align}\label{eq:genloss}
\mathrm{if}~~M\neq M_0~~~\E_M[(f-x_0)^2] =1 - \mathbb E_M f,~~~~~\mathrm{and~otherwise}~~~~~\E_{M_0}[(f-x_0)^2] = \mathbb E_{M_0} f.
\end{align}

    We will prove the theorem by proving the three items one by one. In this proof write $\overline{\cH} = \overline{\cH}(\lambda,K_n,K_d)$.
    
    \noindent
    {\bf Proof of item 1.} We have by Hoeffding's inequality that for any $M \in \overline{\cH}$:
    \begin{align*}
    \mathbb P_{M}\left(\sum_{i,k} Y_{i,k} - K_nK_d\lambda \geq -\sqrt{2nd\log(1/\delta)} \right) \geq 1-\delta,
    \end{align*}
    and
    \begin{align*}
        \mathbb P_{M_0}\left(\sum_{i,k} Y_{i,k} < \sqrt{2nd\log(1/\delta)} \right) \geq 1-\delta.
    \end{align*}
    So that under the signal condition of item 1:
    $$\sup_{M \in \overline{\cH}}\mathbb P_{M}(f_{\mathrm{GS}} = 0) \lor \mathbb P_{M_0}(f_{\mathrm{GS}} = 1) \leq \delta.$$
    This concludes the proof by Equation~\eqref{eq:genloss}.

    \noindent
    {\bf Proof of item 2.} We have by Hoeffding's inequality and an union bound that for any $M \in \overline{\cH}$:
    \begin{align*}
    \mathbb P_{M}\left(\sup_k \sum_{i} Y_{i,k} - K_n\lambda \geq -\sqrt{2n\log(d/\delta)} \right) \geq 1-\delta,
    \end{align*}
    and
    \begin{align*}
        \mathbb P_{M_0}\left(\sup_k \sum_{i} Y_{i,k} < \sqrt{2n\log(d/\delta)} \right) \geq 1-\delta.
    \end{align*}
    So that under the signal condition of item 1:
    $$\sup_{M \in \overline{\cH}}\mathbb P_{M}(f_{\mathrm{cs}} = 0) \lor \mathbb P_{M_0}(f_{\mathrm{cs}} = 1) \leq \delta.$$
    This concludes the proof by Equation~\eqref{eq:genloss} for $f_{\mathrm{cs}}$. We do it similarly for $f_{\mathrm{rs}}$.

        \noindent
    {\bf Proof of item 3.} We have by Hoeffding's inequality and an union bound that for any $M \in \overline{\cH}$ and $m \leq K_n \land K_d$:
    \begin{align*}
    \mathbb P_{M}\left(\sup_{S_n\subset [n], S_d\subset [d]: |S_n|\lor |S_d| \leq m}\sum_{i\in S_n, k\in S_d} Y_{i,k} - m^2 \lambda \geq -m\sqrt{2m\log(nd/\delta)} \right) \geq 1-\delta,
    \end{align*}
    and
    \begin{align*}
        \mathbb P_{M_0}\left(\sup_{S_n\subset [n], S_d\subset [d]: |S_n|\lor |S_d| \leq m}\sum_{i\in S_n, k\in S_d} Y_{i,k} \leq m\sqrt{2m\log(nd/\delta)} \right) \geq 1-\delta.
    \end{align*}
    So that under the signal condition of item 1:
    $$\sup_{M \in \overline{\cH}}\mathbb P_{M}(f_{\mathrm{SS}}^{(m)} = 0) \lor \mathbb P_{M_0}(f_{\mathrm{SS}}^{(m)} = 1) \leq \delta.$$
    This concludes the proof by Equation~\eqref{eq:genloss}.

\subsection{Proof of Proposition~\ref{prop:ubestM}}
    Note first that for $f$ taking values in $\{0,1\}$, if $M\in \cH$:
\begin{align}\label{eq:genloss2}
\mathrm{if}~~M_{1,k} \geq 0 \ , \forall k~~~\E_M[(f-x^{\star})^2] =1 - \mathbb E_M f,~~~~~\mathrm{and~otherwise}~~~~~\E_M[(f-x^{\star})^2] = \mathbb E_{M} f.
\end{align}

    We will prove the theorem by proving the three items one by one. Write $S^*_E$ for the set of experts $i$ such that $\exists k: M_{i,k} >0$, and $S^*_Q$ for the set of questions $k$ such that $\exists i: M_{i,k} >0$.
    
    \noindent
    {\bf Proof of item 1.} For any $M \in \cH$, with probability higher than $1-\delta$, we get
    \begin{align*}
   \forall i\in S^*_E, \sum_{k} Y_{i,k} - K_d\lambda  \geq -\sqrt{2d\log(n/\delta)} \ ;  \forall i\notin S^*_E, \sum_{k} Y_{i,k} - K_d \leq \sqrt{2n\log(d/\delta)} \ . 
    \end{align*}
    In particular, under the signal conditions of item 1, we have $\mathbb P_{M}\left(f_{SE} = x^\star \right) \geq 1-\delta$. 
        This concludes the proof by Equation~\eqref{eq:genloss2}, by definition of the estimator $f_{SE}$.

        \noindent
    {\bf Proof of item 2.} We have similarly to item 1 that under the signal conditions of item 2:
    $$\mathbb P_{M}\left(\hat S_Q = S_Q^* \right) \geq 1-\delta.$$
    Bby Hoeffding's inequality, we have
      \begin{align*}
    \mathbb P_{M}\left(\sum_{k\in S_Q^*} Y_{1,k} - K_d\lambda \mathbf 1\{i \in S^*_E\} \geq -\sqrt{2K_d\log(d/\delta)} \right) \geq 1-\delta \ ; \mathbb P_{M}\left(\sum_{k\in S_Q^*} Y_{1,k} - K_d\lambda \mathbf 1\{i \in S^*_E\} \leq -\sqrt{2K_d\log(d/\delta)} \right) \geq 1-\delta\ . 
    \end{align*}
    So that under the signal conditions of item 2, we have by a union bound
    $$\mathbb P_{M}\left(f_{SE,2} = x^{\star} \right) \geq 1-2\delta.$$
        This concludes the proof by Equation~\eqref{eq:genloss2}, by definition of the estimator $f_{SE,2}$.

                \noindent
    {\bf Proof of item 3.} Again, we apply Hoeffding's inequality. Fix any $M \in \cH$ and $m \leq (K_n-1) \land K_d$. With probability higher than $1-\delta$, we have 
    \begin{align*}
    \max_{\substack{S_n\subset [2, n], S_d\subset [d],\\ |S_n|\lor |S_d| \leq m\\
    |S_n \cap S_Q^*|\leq m/2 
    }}\sum_{i\in S_n, k\in S_d} Y_{i,k} - |S_n \cap S_E^*||S_d \cap S_Q^*|\lambda  \leq -m\sqrt{2m\log(nd/\delta)} \\
    \min_{\substack{S_n\subset [2, n], S_d\subset [d],\\ |S_n|\lor |S_d| \leq m\\
    |S_n \cap S_Q^*|\geq m/2 
    }}\sum_{i\in S_n, k\in S_d} Y_{i,k} - |S_n \cap S_E^*||S_d \cap S_Q^*|\lambda  \geq -m\sqrt{2m\log(nd/\delta)} \ . 
    \end{align*}
    So that under the signal conditions of item 2, we have  $\mathbb P_{M}\left( |\hat S_{Q}^{(m)} \cap S_Q^*| \geq m/2\right) \geq 1-\delta$. 
    Since $\hat S_{Q}^{(m)}$ is computed without using the first row, we deduce that, with probability higher than $1-2\delta$, 
    \begin{align*}
    \sum_{k\in \hat S_{Q}^{(m)}} Y_{1,k}  \geq m\lambda/2 -\sqrt{2m\log(1/\delta)}  \text{ if }1\in S_E^* \  ; \sum_{k\in \hat S_{Q}^{(m)}} Y_{1,k}  \leq \sqrt{2m\log(1/\delta)} \text{ otherwise }. 
    \end{align*}
    So that under the signal condition of item 3 we have $\mathbb P_{M}\left( f_{SE,2}^{(m)} = x^\star\right) \geq 1-2\delta$. Together  with Equation~\eqref{eq:genloss2}, this concludes the proof.

\subsection{Proof of Proposition~\ref{prop:LBinfestM}.}
The proof mainly follows the steps of generalized Le Cam's method~\cite{tsybakov}.  We consider two cases.

\noindent
{\bf Case 1: $K_d \leq K_n$.} In this case, Condition~\eqref{eq:condition:info:lower:estion} amounts to assuming that $\lambda \sqrt{K_d}\leq \bar c$ for a small constant $\bar c>0$.

We consider the simplest problem where some oracle tells us which questions are such that not all entries are $0$, namely the set $S_d \subset [d]$ of the questions of interest - which is a random subset whose size is random and distributed as $\mathcal B(d,K_d/d)$. In this case, a sufficient statistic for the problem of estimating $x^{\star}$ is the sum for expert $1$ over all questions in $S_d$, namely $Z = \sum_{j \in S_d} Y_{1,j}$. There exists an optimal estimator for $x^{\star}$ that is a measurable function of $Z$. Besides, we have the following conditional distributions. 
\[
\left(\frac{|S_d|+ Z}{2}\Bigg|\left[x^\star=0, S_d\right]\right) \sim \mathcal B(|S_d|,1/2)\text{ and }\left(\frac{|S_d|+ Z}{2}\Bigg|\left[x^\star=1, S_d\right]\right) \sim \mathcal B(|S_d|,(1+\lambda)/2) \ .
\]
We deduce from Pinsker's inequality that  
\begin{align*}
\inf_{f~\mathrm{measurable}}\mathbb P\left(f=1\Bigg|x^{\star}=0, S_d\right) \lor \mathbb P\left(f=0\Bigg|x^{\star}=1, S_d\right) &\geq \frac{1}{2} - \frac{1}{2}\sqrt{\frac{1}{2}\mathrm{KL}(\mathcal  \mathcal B(|S_d|,1/2);B(|S_d|,(1+\lambda)/2))}\\
&\geq \frac{1}{2} - \frac{1}{2}\sqrt{\frac{1}{2}|S_d|\mathrm{KL}( \mathcal B(1/2);\mathcal B((1+\lambda)/2))}\enspace  ,
\end{align*}
by the chain rule. Since $0\leq \lambda \leq 1/2$, we know that $\mathrm{KL}( \mathcal B(1/2);\mathcal B((1+\lambda)/2))= -0.5\log(1-\lambda^2) \leq \lambda^2$. So that
\begin{align*}
\inf_{f~\mathrm{measurable}}\mathbb P\left(f=1\Bigg|x^{\star}=0, S_d\right) \lor \mathbb P\left(f=0\Bigg|x^{\star}=1, S_d\right) \geq \frac{1}{2} - \frac{\lambda\sqrt{|S_d|}}{2^{3/2}}\enspace .
\end{align*}
By definition of this model, we have $\mathbb P(x^{\star}=1|S_d) =K_n/n$. Since $K_n \leq n/2$, we get
\begin{align*}
\inf_{f~\mathrm{measurable}}\mathbb P\left(|f - x^{\star}| \geq 1/2\Bigg|S_d\right) \geq \left(\frac{1}{2} - \frac{\lambda\sqrt{|S_d|}}{2^{3/2}}\right) \frac{K_n}{n}\enspace  .
\end{align*}
Finally, averaging over $S_d$ and reminding that $|S_d| \sim \mathcal B(d,K_d/d)$, and using Jensen's inequality
\begin{align*}
\inf_{f~\mathrm{measurable}}\mathbb P\left(|f - x^{\star}| \geq 1/2\right) &\geq \left(\frac{1}{2} - \frac{\lambda\mathbb E[\sqrt{|S_d|}]}{2^{3/2}}\right) \frac{K_n}{n}  \geq \left(\frac{1}{2} - \frac{\lambda\sqrt{K_d}}{2^{3/2}}\right) \frac{K_n}{n}\enspace .
\end{align*}
Finally since $\mathbb E[x^{\star 2}] = K_n/n$ and since $\lambda\sqrt{K_d}\leq \bar c$ for $\bar c$ small enough, we deduce that 
\begin{align*}
\inf_{f~\mathrm{measurable}}\mathbb E\left((f - x^{\star})^2\right) &\geq \frac{1}{16}\mathbb E[x^{\star 2}]\enspace .
\end{align*}

\noindent
{\bf Case 2: $K_d \geq K_n$.} Condition~\eqref{eq:condition:info:lower:estion} amounts to assuming that $\lambda \sqrt{K_n} \lor \frac{\lambda K_d}{\sqrt{d}}\leq \bar c$ for a small constant $\bar c>0$.

We consider the simplest problem where some oracle tells us which experts are such that not all entries are $0$ except for the first expert of interest, namely the set $S_n' = S_n \setminus \{1\} \subset [d]$ of the experts of interest - which is a random subset whose size is random and distributed as $\mathcal B((n-1),K_n/n)$. In this case, a sufficient statistic for the problem of estimating $x^{\star}$ is the average vector of all experts over  $S_n'$, namely $W = \sum_{i \in S_n'} Y_{i,.}$ and the questions of the first expert, namely $Y' = Y_{1,.}$. There exists an optimal estimator for $x^{\star}$ that is a measurable function of $(W,Y')$. Conditionally to $S_n'$ and $x$, the $(\frac{W_k+|S_n'|}{2},\frac{Y'_k+1}{2})_{k}$ are i.i.d. and follow the mixture distribution 
\[
\left(1-\frac{K_d}{d}\right) \mathcal B(|S_n'|,1/2) \otimes \mathcal B(1,1/2) + \frac{K_d}{d} \mathcal B(|S_n'|, (1+\lambda )/2)\otimes\mathcal B(1, (1+\lambda x^\star)/2) \ . 
\]
We define $f_1$ (resp. $f_0$) for the conditional density of $(\frac{W_k+|S_n'|}{2},\frac{Y'_k+1}{2})$ given $S'_n$ and $x^*=1$ (resp. $x^*=0$) with respect to the counting measure on $\mathbb{N}$. For any $(a,b)\in \{0,\ldots, |S'_n|\}\times \{0,1\}$, we have 
\begin{align*}
f_1(a,b) &= \binom{|S_n'|}{a} 2^{-|S_n'|-1} \left[\frac{d-K_d}{d} + \frac{K_d}{d} (1+\lambda)^{a+b}(1-\lambda)^{|S_n'|+1 - a - b}\right]\\
&= f_0(a,b) + \binom{|S_n'|}{a} 2^{-|S_n'|-1}  \frac{K_d}{d} (1+\lambda)^{a}(1-\lambda)^{|S_n'| - a}  2\lambda (b-1/2)\enspace . 
\end{align*}
Define $u(b) =  2\lambda (b-1/2)K_d/d$ and $\tilde f$ as the density of $\mathcal B(|S_n'|, (1+\lambda )/2)\otimes\mathcal B(1, 1/2)$. 
It follows that $f_1(a,b) = f_0(a,b)+ u(b)\tilde{f}(a,b)$. 
Let us upper bound the $\chi^2$ divergence between these two densities.
\begin{align*}
\chi^2\left(f_1, f_0 | S_n'\right) &= \sum_{a=0}^{|S_n'| }\sum_{b=0}^1 \frac{f_1^2(a,b)}{f_0(a,b)} - 1= \sum_{a=0}^{|S_n'| }\sum_{b=0}^1 \frac{f_0^2(a,b)  + 2f_0(a,b)  u(b) \tilde f(a,b) + u^2(b) \tilde f^2(a,b) }{f_0(a,b) } - 1 \\ & = \sum_{a=0}^{|S_n'| }\sum_{b=0}^1 \frac{2f_0(a,b) \tilde f(a,b)  u(b)+ \tilde f^2(a,b)  u^2(b)}{f_0(a,b) }\enspace . 
\end{align*}
By definition, we have $\sum_{b=0}^1 u(b) = 0$, $\tilde{f}(a,0)= \tilde{f}(a,1)$, and $f_0(a,1)= f_0(a,0)$. Thus, we get
\begin{align*}
\chi^2\left(f_1, f_0 | S_n'\right) &= \sum_{a=0}^{|S_n'| }\sum_{b=0}^1 \frac{\tilde f^2(a,b) u^2(a,b)}{f_0(a,b)} = \lambda^2 \frac{K_d^2}{d^2} \sum_{a=0}^{|S_n'| } \frac{2\tilde f^2(a,0)}{f_0(a,0)}.
\end{align*}
The function  $g:\lambda \rightarrow (1+\lambda)^{-a}(1-\lambda)^{-|S_n'| + a}$ is convex for any $a \in [0,S_n]$. Hence, we derive that 
\begin{align*}
\frac{\tilde f^2(a,0)}{f_0(a,0)} =  \frac{\tilde f(a,0)}{\frac{d-K_d}{d} (1+\lambda)^{-a}(1-\lambda)^{-|S_n'| + a}+ \frac{K_d}{d}} \leq \frac{\tilde f(a,0) }{(1+\frac{d-K_d}{d}\lambda)^{-a}(1-\frac{d-K_d}{d}\lambda)^{-|S_n'| + a}}\enspace . 
\end{align*}
Then, we come back to the $\chi^2$ divergence and we work out the density to obtain
\begin{align*}
\chi^2\left(f_1, f_0 | S_n'\right) &\leq \lambda^2 \frac{2K_d^2}{d^2} \sum_{a=0}^{|S_n'| } \tilde f(a,0) \left(1+\frac{d-K_d}{d}\lambda\right)^{a}\left(1-\frac{d-K_d}{d}\lambda\right)^{|S_n'| - a}\\
&\quad \quad = \lambda^2 \frac{K_d^2}{d^2} \left[\mathbb E_{B'\sim \mathcal B(1,1/2)}\left[\left(1-\frac{d-K_d}{d} \lambda\right)(1- \lambda) \left[\frac{(1+\frac{d-K_d}{d} \lambda)(1+ \lambda)}{(1-\frac{d-K_d}{d} \lambda)(1- \lambda)}\right]^{B'}\right]\right]^{|S_n'|}.
\end{align*}
Define $v = \log\left(\frac{(1+\frac{d-K_d}{d} \lambda)(1+ \lambda)}{(1-\frac{d-K_d}{d} \lambda)(1- \lambda)}\right)$ and 
$t= (1+\frac{d-K_d}{d} \lambda)(1+ \lambda)(1-\frac{d-K_d}{d} \lambda)(1- \lambda)$.  We obtain
\begin{align*}
\chi^2\left(f_1, f_0 | S_n'\right) &\leq 
\lambda^2 \frac{K_d^2}{d^2}  \left[\mathbb E_{B'\sim \mathcal B(1,1/2)}\left[t^{1/2} e^{v(B'-1/2)}\right]\right]^{|S_n'|}\enspace .
\end{align*}
Then,  Hoeffding's inequality ensures that $\mathbb E_{B'\sim \mathcal B(1,1/2)}[e^{v(B'-1/2)}]\leq e^{v^2/2}$. Moreover, we have $|t|\leq 1$. This yields 
\begin{align*}
\chi^2\left(f_1, f_0 | S_n'\right) &\leq 
\lambda^2 \frac{K_d^2}{d^2} e^{v^2|S_n'|/2}\enspace . 
\end{align*}
Since 
\[
v = \log\left(1+ 2\frac{\lambda}{1-\lambda}\right)+ \log\left(1+ 2(1-K_d/d)\frac{\lambda}{1-(1-K_d/d)\lambda}\right)\leq \frac{4\lambda}{1-\lambda}\leq 8\lambda\enspace , 
\]
for $\lambda \leq 1/2$, we get 
\begin{align*}
\chi^2\left(f_1, f_0 | S_n'\right) &\leq 
\lambda^2 \frac{K_d^2}{d^2} e^{32\lambda^2|S_n'|}.
\end{align*}
Now going to the entire vector of sufficient statistics, we obtain by tensorization that 
\begin{align}\label{eq:chi_2_f1_f0}
\chi^2\left(f_1^{\otimes d}, f_0^{\otimes d} | S_n'\right) &= (1+\chi^2\left(f_1, f_0| S_n'\right))^d -1 
\leq 
\left[1+\lambda^2 \frac{K_d^2}{d^2} e^{32\lambda^2|S_n'|}\right]^d- 1 \enspace . 
\end{align}
Moreover, we have already established in {\bf Case 1} that 
\begin{align*}
\inf_{f~\mathrm{measurable}}\mathbb E\left(|f - x^{\star}|^2\right)  \geq \frac{1}{4}\left(\frac{1}{2} - \frac{1}{2} \mathbb E\left[\mathrm{TV}\left(f_1^{\otimes d}, f_0^{\otimes d} | S_n'\right)\right] \right)\frac{K_n}{n}\enspace . 
\end{align*}
Since $|S_n'| \sim \mathcal B(n-1, K_n/n)$. We have $\mathbb{P}[|S'_n|\geq 32K_n]\leq \exp\left[- 16K_n\right]$. Hence, we otbain
\begin{align*}
\inf_{f~\mathrm{measurable}}\mathbb E\left(|f - x^{\star}|^2\right)  \geq \frac{1}{4}\left(\frac{1}{2} - \frac{1}{2} \mathbb E\left[\mathrm{TV}\left(f_1^{\otimes d}, f_0^{\otimes d} | S_n'\right) \mathbf 1\{|S_n'| \leq 32 K_n\}\right]  - 2\exp( - 16K_n)\right)\frac{K_n}{n}\enspace .
\end{align*}
Since $\mathrm{TV}\left(f_1^{\otimes d}, f_0^{\otimes d} | S_n'\right) \leq \sqrt{\chi^2\left(f_1^{\otimes d}, f_0^{\otimes d} | S_n'\right)}$, together with~\eqref{eq:chi_2_f1_f0}, this leads us to 
\begin{align*}
\inf_{f~\mathrm{measurable}}\mathbb E\left(|f - x^{\star}|^2\right)  \geq \frac{1}{4}\left(\frac{1}{2} - \frac{1}{2} \mathbb E\left[\sqrt{\left[1+\lambda^2 \frac{K_d^2}{d^2} e^{2^{10}\lambda^2K_n}\right]^d - 1}\right] - 2\exp( - 16K_n)\right)\frac{K_n}{n}\enspace .
\end{align*}
Since we assumeed that  $\lambda \sqrt{K_n}\vee \lambda\frac{K_d}{\sqrt{d}}\leq \bar{c}$ for a small numerical constant $\bar{c}$, we conclude that 
\begin{align*}
\inf_{f~\mathrm{measurable}}\mathbb E\left(|f - x^{\star}|^2\right)  \geq \frac{1}{4}\left(\frac{1}{4} - 2\exp( - 16K_n)\right)\frac{K_n}{n}.
\end{align*}
This concludes the proof as $\mathbb E[x^{\star 2}]=K_n/n$.

\end{document}